\documentclass[letterpaper,12pt,titlepage,twoside,final]{book}

\newcommand{\href}[1]{#1} 

\usepackage{ifthen}
\usepackage{lmodern}
\newboolean{PrintVersion}
\setboolean{PrintVersion}{false}

\usepackage{amsmath,amssymb,amstext} 
\usepackage[pdftex]{graphicx} 

\usepackage{xcolor}
\usepackage[pdftex,pagebackref=false,colorlinks=true]{hyperref}
\hypersetup{
    plainpages=false,       
    unicode=false,          
    pdftoolbar=true,        
    pdfmenubar=true,        
    pdffitwindow=false,     
    pdfstartview={FitH},    
    pdfnewwindow=true,      
    colorlinks=true,        
    linkcolor=blue,         
    citecolor=red,        
    filecolor=magenta,      
    urlcolor=cyan           
}
\ifthenelse{\boolean{PrintVersion}}{   
\hypersetup{	
    citecolor=black,%
    filecolor=black,%
    linkcolor=black,%
    urlcolor=black}
}{} 

\usepackage{caption}
\usepackage{tikz-cd}
\usetikzlibrary{cd}

\usepackage{amsthm}
\usepackage{bm}
\usepackage{adjustbox}
\usepackage{float}
\usepackage{stmaryrd}
\usepackage[toc,page]{appendix}
\usepackage{amssymb,amsfonts,amsthm,amsmath,amsopn,amstext,amscd,latexsym,color}
\usepackage[all]{xy}
\usepackage[nottoc,notlot,notlof]{tocbibind}
\usepackage{mathtools}
\usepackage{enumerate}
\xyoption{web}
\usepackage[utf8]{inputenc}
\usepackage[OT2,T1]{fontenc}
\usepackage{enumitem}
\usepackage{longtable}
\usepackage[automake,toc,abbreviations]{glossaries-extra}
\usepackage{bigints}
\usepackage{hyperref}
\usepackage{cases}
\usepackage{calc}
\usepackage{epigraph}
\usepackage{cleveref}
\usepackage[utf8]{inputenc}

\usepackage[toc,page]{appendix}
\usepackage{mathrsfs} 
\usepackage{amsfonts}
\usepackage{amsmath}
\usepackage{amsthm}
\usepackage{amssymb}
\usepackage[english]{babel}
\usepackage{graphics}
\usepackage{fancyhdr}
\usepackage{latexsym}
\usepackage{longtable} 
\usepackage{mathrsfs} 
\usepackage{graphicx}
\usepackage{wrapfig} 

\let\hun\H

\def\Z{{\mathbb Z}}
\def\N{{\mathbb N}}
\def\R{{\mathbb R}}

\def\C{{\mathbb C}}
\def\H{{\mathbb H}}
\def\Q{{\mathbb Q}}
\def\A{{\mathbb A}}
\def\F{{\mathbb F}}
\def\O{{\mathcal O}}
\def\a{{\mathfrak a}}
\def\p{{\mathfrak p}}
\def\m{{\mathfrak m}}
\def\n{{\mathfrak n}}
\def\c{{\mathfrak c}}
\def\D{{\mathfrak D}}
\def\W{{\mathcal W}}
\def\X{{\mathfrak X}}
\def\Qp{{\mathbb{Q}_p}}
\def\AF{{\mathbb{A}_F}}

\def\iotaf{{\iota_\text{f}(\sigma^{-1})}}

\def\congsub{{\Gamma_\mu(\n)}}
\def\E(Q){{E(\mathbb{Q})}}
\def\wt{\widetilde}
\def\bE(Q){{\bar{E}(\mathbb{Q})}}
\def\eps{\varepsilon}
\def\ro{\varrho}
\def\ph{\varphi}

\def\gl2R{\GL_2(\mathbb{R})}
\def\sl2Z{\SL_2(\mathbb{Z})}
\def\af{a_f(n;\sigma)}
\def\Wp{W_{\pi_p}}
\def\Wv{W_{\Pi_v}}

\def\smatrix{(\begin{smallmatrix} a & b \\ c & d \end{smallmatrix})}
\def\gtlv{g_{t,l,v}}
\def\ctl{c_{t,l}(\chi)}
\def\cc{\omega_\pi}
\def\bs{\backslash}
\def\f{{\mathbf f}}
\def\bi{\beta_i}
\def\bj{\beta_j}
\def\bone{\beta_1}
\def\btwo{\beta_2}
\def\qci{q^{c_i}}

\def\qcone{q^{c_1}}
\def\qctwo{q^{c_2}}
\def\valp{\val_p}
\def\1{\bm{1}}

\def\wtpi{\wt{\pi}}
\DeclareMathOperator{\Ima}{Im}

\DeclareMathOperator{\GL}{GL}
\DeclareMathOperator{\SL}{SL}
\DeclareMathOperator{\GSp}{GSp}
\DeclareMathOperator{\Sp}{Sp}

\DeclareMathOperator{\Hom}{Hom}
\DeclareMathOperator{\Aut}{Aut}
\DeclareMathOperator{\Gal}{Gal}
\DeclareMathOperator{\Tr}{Tr}
\DeclareMathOperator{\val}{val}

\newtheorem{IntroThm}{Theorem}
\newtheorem{Thm}{Theorem}[section]
\newtheorem{Lemma}[Thm]{Lemma}

\newtheorem{Prop}[Thm]{Proposition}
\newtheorem{Defn}[Thm]{Definition}

\theoremstyle{definition}

\newtheorem{Question}[Thm]{Question}

\theoremstyle{remark}
\newtheorem{Remark}[Thm]{Remark}

\let\origdoublepage\cleardoublepage
\newcommand{\clearemptydoublepage}{%
  \clearpage{\pagestyle{empty}\origdoublepage}}
\let\cleardoublepage\clearemptydoublepage

\newglossaryentry{computer}
{
name=computer,
description={A programmable machine that receives input data,
               stores and manipulates the data, and provides
               formatted output}
}

\newglossary*{nomenclature}{Nomenclature}
\newglossaryentry{dingledorf}
{
type=nomenclature,
name=dingledorf,
description={A person of supposed average intelligence who makes incredibly brainless misjudgments}
}

\newabbreviation{aaaaz}{AAAAZ}{American Association of Amateur Astronomers and Zoologists}

\newglossary*{symbols}{List of Symbols}
\newglossaryentry{rvec}
{
name={$\mathbf{v}$},
sort={label},
type=symbols,
description={Random vector: a location in n-dimensional Cartesian space, where each dimensional component is determined by a random process}
}
\makeglossaries

\begin{document}

\pagestyle{empty}
\pagenumbering{roman}

\begin{titlepage}

        \begin{center}
        
        \Huge
        {\bf The Arithmetic of the Fourier Coefficients of Automorphic Forms}

        \vspace*{1.0cm}
        
        \large
        by
        
        \vspace*{0.5cm}

        \Large 
        Tim Davis \\

 \vspace*{1cm}

        \large
        {Submitted in partial fulfilment of the requirements of the degree of\\
        \emph{Doctor of Philosophy}} \\
       
        \vspace*{2cm}

        \vspace*{2cm}

        School of Mathematical Sciences \\
        Queen Mary, University of London\\
        August 2023


        \end{center}
\cleardoublepage
\end{titlepage}

\pagestyle{plain}
\setcounter{page}{2}



\cleardoublepage


\chapter*{Declaration}
\addcontentsline{toc}{chapter}{Declaration}
\markboth{Acknowledgements}{Declaration}
\pagenumbering{roman}

  
 \noindent  I, Tim Davis, confirm that the research included within this thesis is my own work or that where it has been carried out in collaboration with, or supported by others, that this is duly acknowledged below and my contribution indicated. Previously published material is also acknowledged below.   

\noindent I attest that I have exercised reasonable care to ensure that the work is original, and does not to the best of my knowledge break any UK law, infringe any third party’s copyright or other Intellectual Property Right, or contain any confidential material.    

 \noindent I accept that Queen Mary University of London has the right to use plagiarism detection software to check the electronic version of the thesis.    I confirm that this thesis has not been previously submitted for the award of a degree by this or any other university.    

 \noindent The copyright of this thesis rests with the author and no quotation from it or information derived from it may be published without the prior written consent of the author.   
 
 \noindent Signature: Tim Davis
 
 \noindent Date : 31st August 2023

 \noindent Details of collaboration and publications:

\noindent Chapter 3 of this thesis is based on my article \textit{Fourier coefficients of Hilbert modular forms at cusps} \cite{TD}, which appeared in the Ramanujan Journal, 2023.
 \noindent

  


\chapter*{}
\begin{center}
\large
\textbf{To my family}
\end{center}
\cleardoublepage

\chapter*{Abstract}
\addcontentsline{toc}{chapter}{Abstract}
\markboth{Abstract}{Abstract}

This thesis studies modular forms from a classical and adelic viewpoint. We use this interplay to obtain results about the arithmetic of the Fourier coefficients of modular forms and their generalisations. 

In Chapter \ref{Chap2}, we compute lower bounds for the $p$-adic valuation of local Whittaker newforms with non-trivial central character. We obtain these bounds by using the local Fourier analysis of these local Whittaker newforms and the $p$-adic properties of $\eps$-factors for $\GL_1$.

In Chapter \ref{Chap3}, we study the fields generated by the Fourier coefficients of Hilbert newforms at arbitrary cusps. Precisely, given a cuspidal Hilbert newform $f$ and a matrix $\sigma$ in (a suitable conjugate of) the Hilbert modular group, we give a cyclotomic extension of the field generated by the Fourier coefficients at infinity which contains all the Fourier coefficients of $f||_k\sigma$. 

Chapters \ref{Chap2} and \ref{Chap3} are independent of each other and can be read in either order. In Chapter \ref{Chap4}, we briefly discuss the relation between these two chapters and mention potential future work.

\cleardoublepage
\newpage


 \cleardoublepage


\chapter*{Acknowledgments}
\addcontentsline{toc}{chapter}{Acknowledgments}
\markboth{Acknowledgments}{Acknowledgments}

Firstly, I want to thank my supervisor Professor Abhishek Saha whose guidance and support has always been there. Without his infectious knowledge of the subject, this work would not be here. I would also like to thank my yearly panel reviewers, Dr Steve Lester, Dr Behrang Noohi and Dr Felipe Rin\'{c}on for providing me with useful feedback.  Two more people which I would like to thank are Dr Chris Daw and Dr Rachel Newton whose passion for number theory and their encouragement made me think about doing a PhD in the first place. I am very grateful to the Leverhulme Trust who funded this research, grant number RPG-2018-401.

To all of my office friends, Danilo, Elisa, Evelyn, Federica, Gabriele, James, Lexi and Marco, I thank you all for making my time as a PhD student so much more enjoyable. To my great friend Dr Ed Clark a special thank you, our shared interest in mathematics and our discussions over the years have been incredibly influential to me. To Marica, I cannot thank you enough for all the love and support you have given me over the years. This has been invaluable to me. 

Finally, I thank my family, Mum, Dad and Gran, your support no matter what has always kept me going and believing in myself. This work is for you.

\cleardoublepage

\renewcommand\contentsname{Contents}
\addcontentsline{toc}{chapter}{Contents}
\tableofcontents
\cleardoublepage
\phantomsection    



\pagenumbering{arabic}

\pagestyle{fancy}

\fancyhead{}

\fancyhead[CE]{\nouppercase{\leftmark}}
\fancyhead[CO]{\nouppercase{\rightmark}}


\chapter{Background and Motivations}\label{Chap1}

\epigraph{\textit{If numbers aren't beautiful I don't know what is}}{P\'{a}l Erd\hun{o}s}

\section{Introduction}\label{intro}

Mathematicians have been studying number theory and asking number theoretic questions ever since mathematics has been written down. In ancient times the number theoretic questions that were being asked were all about prime numbers. For example, in Ancient Greece it was known that there are infinitely many prime numbers. The questions being thought about in these times were very simple to state, such as, which primes can we write as the sum of two squares and are there infinitely many twin primes? Although, questions like these are often very easy to state the proofs tend to be much more difficult or they are still unsolved to this day. The proofs usually give deep insight into the area. Indeed, for twin primes, it was groundbreaking work by Yitang Zhang \cite{zhang} to show that the gap between consecutive primes is finite infinitely often, let alone showing that this gap is two infinitely often.  There have been successful attempts to refine the method used to reduce the bound provided by Zhang. However, there are currently no realistic methods to prove the conjecture.

Another example is that of Fermat's Last Theorem. This took over 350 years for this problem to finally be solved by Wiles \cite{wiles}. The proof of this theorem is incredibly complicated. Wiles proved this theorem by showing a deep connection between two different mathematical objects, namely, Elliptic curves $E$ defined over $\Q$ and holomorphic functions, $f$, that live on the upper-half plane, which satisfy certain symmetries. 

There are several different ways to formulate this relation. Possibly, the most useful way for our purposes is as follows. To arithmetic and geometric objects, we can associate a complex analytic function known as an $L$-function. The simplest example of an $L$-function is the Riemann zeta function defined by \[\zeta(s)=\sum_{n=1}^\infty\frac{1}{n^s}=\prod_{p \ \text{prime}}(1-p^{-s})^{-1},\] for Re$(s)>1$ and by analytic continuation elsewhere. This function with its deep connection to prime numbers has been studied in great detail in its own right. The subject of the zeros of this function is one of the Millennium Prize problems. 

To $E$ and $f$ we can associate an $L$-function and what Wiles showed is that the $L$-function attached to $E$ is the same as the $L$-function attached to $f$. Actually, $f$ is what is called a modular form and these will be studied throughout this thesis. Modular forms and their generalisations arise naturally in many different areas of mathematics, not only number theory and arithmetic geometry but also mathematical physics. 

All of these studies lie in the theory of automorphic representations and these are at the heart of modern number theory. One of the driving forces behind this is the fascinating work of Langlands who associates $L$-functions to automorphic representations. There are now a great many mathematicians now working on what is called the Langlands' Program. 

We will see that modular forms and their generalisations have a Fourier expansion. This means to understand the properties of modular forms we are required to study the properties of their Fourier coefficients. In this thesis, we will investigate some of the algebraic and arithmetic properties that these Fourier coefficients have. Instead of tackling this problem directly, we will often transfer the problem to one involving automorphic representations. This link between the two different settings is discussed in the next section. 

\subsection{General notation}

We collect some notation which will be used regularly throughout the thesis. At the beginning of Chapters \ref{Chap2} and \ref{Chap3} we will collect notation which is specific to that chapter. The symbols $\Z$, $\Q$, $\R$, $\C$, $\Z_p$ and $\Qp$ all have the usual meaning. For a complex number $z$, we denote $e(z)=e^{2\pi i z}$. Let $N\geq1$, we define $\zeta_N=e^{2\pi i/N}$. Let $\H$ denote the upper half plane, that is, \[\H=\{z=x+iy\in\C: y>0\}.\]For any commutative ring, $R$, we define the groups $\GL_2(R)$ and $\SL_2(R)$ as \[\GL_2(R)=\left\{g\in M_2(R):\det g\in R^\times\right\}\] and \[\SL_2(R)=\{g\in\GL_2(R):\det g=1\}\] respectively. We define the following matrices as \[a(y)=\begin{pmatrix} y & \\ & 1 \end{pmatrix}, \ n(x)=\begin{pmatrix} 1 & x \\ & 1 \end{pmatrix}, \ z(t)=\begin{pmatrix} t & \\ & t \end{pmatrix}, \ \text{with} \ y,t\in R^\times, x\in R.\] We define the subgroup $B(R)$ of $\GL_2(R)$ by \[B(R)=\left\{\begin{pmatrix} a & b \\ & d \end{pmatrix}: a,d\in R^\times, b\in R\right\}.\] For a character or representation we denote $c(\cdot)$ to be the conductor of the character or the representations.  

\section{Classical and adelic setting}\label{ctoa}

In this section, we give an overview and discuss the classical and adelic settings of modular forms. We start with the classical setting and then move to the adelic approach, whereby we see some of the benefits of these two different approaches. Finally, we see how these two different settings interplay. This interaction is pivotal to the work carried out in this thesis. This is the real motivation for the work carried out during this thesis. In this section, we only consider the setting of modular forms over $\Q$, in Chapter \ref{Chap3} we see how this works more generally for any totally real number field. 

In this section, we will discuss the general motivation for my work during my PhD. This will include viewing automorphic forms from a classical and adelic viewpoint and explaining the benefits of viewing these objects from different viewpoints.

\subsection{The classical setting}\label{Class}

We start by overviewing the classical setting and defining classical holomorphic modular forms. 

Consider $g=\smatrix\in\GL_2^+(\R)$, where $\GL_2^+(\R)$ denotes the subgroup of $\GL_2(\R)$ whose elements have positive determinant. Then $\GL_2^+(\R)$ acts on $\H$ via \[g\cdot z=\smatrix\cdot z=\frac{az+b}{cz+d}.\] Let $f:\H\rightarrow\C, \ k$ an integer and $g\in\GL_2^+(\R)$ then we can define a function $f|_kg:\H\rightarrow\C$ as \[f|_kg(z):=\frac{(\det g)^{k/2}}{(cz+d)^k}f(g\cdot z), \ g=\smatrix\in\GL_2^+(\R).\] Let $\Gamma(N)$ be the subgroup of $\SL_2(\Z)$ defined by  \[\Gamma(N)=\left\{\gamma\in\sl2Z:\gamma\equiv
(\begin{smallmatrix} 
1 & 0 \\
0 & 1
\end{smallmatrix}) \pmod{N}\right\}.\] This is called the principal congruence subgroup of $\SL_2(\Z)$. Any subgroup $\Gamma$ of $\SL_2(\Z)$ such that $\Gamma(N)\leq\Gamma\leq\SL_2(\Z)$ is called a congruence subgroup. Note that taking $N=1$ gives $\Gamma(1)=\SL_2(\Z)$. For our purposes, we will focus mainly on $\Gamma=\Gamma_0(N), \Gamma_1(N)$. These subgroups are defined as \[\Gamma_0(N)=\left\{\begin{pmatrix}a & b \\ c & d \end{pmatrix} \in \SL_2(\Z): c\equiv0 \pmod{N}\right\}\] and \[\Gamma_1(N)=\left\{\begin{pmatrix} a & b \\ c & d \end{pmatrix}\in \SL_2(\Z): \begin{pmatrix} a & b \\ c & d \end{pmatrix} \equiv \begin{pmatrix} 1 & \ast \\ 0 & 1 \end{pmatrix} \pmod{N}\right\}.\] 

We are now able to define a modular form for $\Gamma$.

\begin{Defn}\label{modform}
A function $f:\H\rightarrow\C$ is a modular form of weight $k\in\Z$ with respect to $\Gamma$ if
\begin{enumerate}
\item[(i)] $f$ is holomorphic on $\H$,
\item[(ii)] $f|_k\gamma=f,$ for every $\gamma\in\Gamma$,
\item[(iii)] $f|_kg$ is holomorphic at $\infty$, for every $g\in\sl2Z$.
\end{enumerate}
\end{Defn}

Note that each $\Gamma$ contains the element $(\begin{smallmatrix} 1 & w \\ & 1 \end{smallmatrix})$ for some minimal $w\in\N$. Therefore, from Definition \ref{modform} (ii), we have that $f(z+w)=f(z)$. Using this as well as the other properties from Definition \ref{modform}, we have that $f$ has a Fourier expansion of the form \[f(z)=\sum_{n\geq0}a_f(n)e^{2\pi i nz/w},\] where $a_f(n)$ are the Fourier coefficients of $f$ and $w$ is the smallest integer such that $(\begin{smallmatrix} 1 & w \\0 & 1\end{smallmatrix})\in\Gamma$.  In particular, for $\Gamma=\Gamma_0(N)$ we have $w=1$. A cuspform is a modular form such that the constant term in the Fourier expansion of $f|_kg$ is zero for every $g\in\SL_2(\Z)$. The space of all modular forms of weight $k$ on $\Gamma$ is denoted by $M_k(\Gamma)$, and the space of all cuspforms of weight $k$ on $\Gamma$ is denoted by $S_k(\Gamma)$.

\begin{Remark}
Let $\chi:(\Z/N\Z)^\times\rightarrow\C^\times$ be a Dirichlet character. In which case we say $f$ is a modular form with character $\chi$ if $f$ satisfies (i)-(iii) and in addition for every $\gamma=\smatrix\in\Gamma_0(N)$, we have that $f|_k\gamma=\chi(d)f$. The space of all such modular forms is denoted \[M_k(N,\chi)=\left\{f\in M_k(\Gamma_1(N)):f|_k\gamma=\chi(d)f, \  \forall\gamma\in\Gamma_0(N)\right\}.\] Note that, for $\chi=\1$ we have that $M_k(N,\1)=M_k(\Gamma_0(N))$ and also \[M_k(\Gamma_1(N))=\bigoplus_\chi M_k(N,\chi).\] If $f$ also satisfies the condition of a cusp form then we say $f$ is a cusp form with character $\chi$. 
\end{Remark}

\begin{Remark}
One can show that \[\eta(z)=e^{2\pi i nz}\prod_{n\geq1}\left(1-e^{2\pi i nz}\right)^{24},\] is an element of $S_{12}(\SL_2(\Z))$.  The LMFDB \cite{LMFDB} contains a wide variety of examples of modular forms for various weights and levels. For our purposes, examples are not the most instructive as we require results for all modular forms of certain weights and levels. 
\end{Remark}

Let $\sigma\in\SL_2(\Z)$ be such that $\c=\sigma\infty$ where $\c$ is a cusp of $\Gamma_0(N)$. If $f$ is a modular form with respect to $\Gamma_0(N)$, then $f|_k\sigma$ is a modular form for the congruence subgroup $\sigma^{-1}\Gamma_0(N)\sigma$. We have that $f|_k\sigma$ has a Fourier expansion of the form \[f|_k\sigma(z)=\sum_{n\geq0}\af e^{2\pi i nz/w(\sigma)},\] where $\af$ are the Fourier coefficients and $w(\sigma)$ is the width of the cusp. The width of the cusp is defined to be the integer $N/(L^2,N)$, where $L$ is the denominator of the cusp. Specifically, we have $L|N$ and $\c=a/L$. 

If $M|N$, then we have $\Gamma_0(N)\subset\Gamma_0(M)$ and so $S_k(\Gamma_0(M))\subset S_k(\Gamma_0(N))$. Therefore, every cuspform of level $M$ is also one of level $N$. As such, it makes sense to distinguish the forms coming from lower levels and those strictly of level $N$. We do this as follows. As $\H$ is a hyperbolic space we can equip $\H$ with the hyperbolic measure $dxdy/y^2$. Let $f,g\in S_k(\Gamma)$. Then we define the Petersson inner product \[\langle\cdot,\cdot\rangle_\Gamma:S_k(\Gamma)\times S_k(\Gamma)\rightarrow\C\] as \[\langle f,g\rangle_\Gamma=\frac{1}{\text{Vol}(\Gamma)}\int_{\Gamma\bs\H}f(z)\overline{g(z)}y^k \ \frac{dxdy}{y^2},\] where \[\text{Vol}(\Gamma)=\int_{\Gamma\bs\H} 1 \ \frac{dxdy}{y^2}.\] This is well defined since for $\gamma\in\Gamma$ we have, 
\begin{align*}
f(\gamma\cdot z)\overline{g(\gamma\cdot z)}\Ima(\gamma\cdot z)&=f|_k\gamma(z)j(g,z)^k\overline{g|_k\gamma}\overline{j(\gamma,k)}^k\frac{\Ima(z)}{|j(\gamma,z)|^{2k}}\\
&=f|_k\gamma(z)\overline{g|_k\gamma(z)} \Ima(z)^k\\
&=f(z)\overline{g(z)}\Ima(z)^k,
\end{align*}
as $f,g$ are elements of $S_k(\Gamma)$. Equipped with this inner product $S_k(\Gamma)$ becomes a finite dimensional Hilbert space, with the decomposition \[S_k(\Gamma)=S_k(\Gamma)^{\text{old}}\oplus S_k(\Gamma)^{\text{new}}.\] Now we briefly define Hecke operators, more details can be found in \cite[Section 2.7-2.8]{TM} for example. For each prime $p$ consider the double coset \[\Gamma_0(N)\begin{pmatrix}1 & \\ & p \end{pmatrix}\Gamma_0(N)=\bigcup_j\Gamma_0(N)\gamma_j.\] This is well defined as this union is finite see \cite[Chapter 5]{1st}. The $p$-th Hecke operator $T_p$ is the operator defined on $S_k(\Gamma_0(N))$ through the natural action of this double coset. Precisely, we have \[T_pf=p^{\frac{k}{2}-1}\sum_j f|_k\gamma_j.\] We have that $T_p$ is hermitian for $S_k(\Gamma_0(N))$. Let $f\in S_k(\Gamma_0(N))$ have a Fourier expansion of the form \[f(z)=\sum_{n\geq1}a_f(n)e^{2\pi i nz},\] then \[T_pf(z)=\sum_{n\geq1}a_{T_pf}(n)e^{2\pi inz},\] where \[a_{T_pf}(n)=\sum_{d|(n,p)}d^{k-1}a_f\left(\frac{np}{d^2}\right)=a_f(np)+p^{k-1}a_f\left(\frac{n}{p}\right).\] Using this definition we can generalise $T_p$ to $T_m$ where $m\geq1$ see \cite[Section 5.3]{1st}. We say that $f$ is an eigenform for $T_m$ if we have $T_mf=\lambda_mf$, with $(m, N)=1$. Then a newform is a cuspform $f\in S_k(\Gamma_0(N))^{\text{new}}$ that is also an eigenform of all the Hecke operators normalised so that $a_f(1)=1$. The newforms are a standard basis for the vector space $S_k(\Gamma_0(N))^{\text{new}}$. In our work, we focus on newforms, and this gives one of the reasons why. We will see later that there is also another very useful result for why we restrict our attention to newforms. 

\subsection{Automorphic forms and automorphic representations}

In number theory, we are interested in the ring $\Z$. We can embed $\Z$ as a discrete subring in the ring $\R$. We have that the quotient $\R/\Z$ is compact. We also have that $\R$ is a locally compact topological group for which an integration theory exists. This means if we want to study $\Z$ we can do so by considering periodic functions. We would like to do something analogous for $\Q$. So, we need to find a locally compact topological group that contains $\Q$ as a discrete subring and the quotient is compact. We now construct such a ring and discuss some of its properties. Let $p$ be a prime and $\Qp$ the $p$-adic numbers. Then the ring of adeles of $\Q$ is defined as \[\A=\sideset{}{'}\prod_{p\leq\infty}\Qp=\big\{a=(a_p)_{p\leq\infty}\in\prod_{p\leq\infty}\Qp:a_p\in\Z_p \ \text{for all but finitely many} \ p\big\},\] where $\Q_\infty=\R$. We can see that $\A$ contains information about all the completions of $\Q$ and so we have the archimedean and non-archimedean knowledge of $\Q$. We will see later that this will be incredibly useful for our purposes. We will also be interested in studying the subset $\A_{\text{fin}}\subset\A$, these are all the elements $a=(a_v)_v$ where at infinite places we have $a_v=1$. The natural topology to equip $\A$ with is the restricted direct product topology, defined by, the one which is generated by the sets \[\prod_{p\leq\infty}U_p, \ U_p \ \text{open in} \ \Qp, \ U_p=\Z_p \ \text{for almost all} \ p.\] Let $x\in\A$, then there exists $\gamma\in\Q$ such that \[x-\gamma\in\R\times\prod_{p<\infty}\Z_p.\] By adding an appropriate integer, if needed, we can assume the archimedean component lies in $[0,1)$. Assuming this, then $\gamma$ is uniquely determined. In this way, we obtain a topological group isomorphism \[\A/\Q\cong\R/\Z\times\prod_{p<\infty}\Z_p.\] To make full use of this new setting we also want to understand what happens in the multiplicative setting. The group of ideles is defined as \[\A^\times=\sideset{}{'}\prod_{p\leq\infty}\Q_p^\times=\big\{a=(a_p)_{p\leq\infty}\in\prod_{p\leq\infty}\Q_p^\times:a_p\in\Z_p^\times \ \text{for all but finitely many} \ p\big\}.\] The topology on $\A^\times$ is the one generated by the sets \[\prod_{p\leq\infty}U_p, \ U_p \ \text{open in} \ \Q_p^\times, \ U_p=\Z_p^\times \ \text{for almost all} \ p.\] With this topology $\A^\times$ becomes a locally compact topological group. 

We now state a key result that we will make great use of throughout the rest of this thesis. Before we do this we define \[\GL_2(\A)=\Big\{g=(g_p)_p\in\prod_{p\leq\infty}\GL_2(\Qp):g_p\in\GL_2(\Z_p) \ \text{for almost all} \ p\Big\}.\] This will be the group which we will define automorphic forms and automorphic representations for.  

\begin{Thm}[Strong approximation for $\Q$]\label{SA} For $n=2$, we have
\begin{enumerate}
\item[(i)] $\SL_2(\R)\SL_2(\Q)$ is dense in $\SL_2(\A)$.
\item[(ii)] Let $K$ be an open compact subgroup of $\GL_2(\A_{\mathrm{fin}})$. Assume that the image of $K$ in $\A_{\mathrm{fin}}^\times$ under the determinant map is $\prod_{p<\infty}\Z_p^\times$. Then the cardinality of \[\GL_2(\Q)\GL_2^+(\R)\bs\GL_2(\A)/K\] is one. 
\end{enumerate}
\end{Thm}

\begin{Remark}
This actually, holds more generally for $\GL_n$ over number fields where the appropriate changes are made. Specifically, for (ii) we have that the cardinality of the set is the size of the class group, see \cite[Theorem 3.3.1]{Bump} for a precise statement.
\end{Remark}

Using the above theorem, we have that \[\A^\times=\Q^\times \times \R_{>0} \times \prod_{p<\infty} \Z_p^\times.\] We are now able to define and state some of the properties of automorphic forms for $\GL_2(\A)$.  Informally, these are functions that will have nice properties which will allow us to relate them to classical modular forms.  Before, we can state the formal definition of an automorphic form we need to define some general properties which we want these functions to have.  Let $\phi:\GL_2(\A)\rightarrow\C$ be a continuous function and for simplicity suppose that \[\phi\left(g\begin{pmatrix} z & \\ & z \end{pmatrix}\right)=\phi(g), \ \text{for every} \ g\in\GL_2(\A) \ \text{and} \ z\in\A^\times.\] This relates to those invariant under the centre of $\GL_2(\A)$. In \cite[Chapter 3]{Bump} they deal with the general case. We say that $\phi:\GL_2(\A)\rightarrow\C$ is smooth, if the function 
\begin{align*}
\GL_2(\R)\times\GL_2(\A_{\text{fin}})&\rightarrow \C\\
(g,h)&\mapsto \phi(gh),
\end{align*}
is $C^\infty$ in $g$ for fixed $h$ and is locally constant in $h$ for a fixed $g$.  

Next, consider functions $\phi:\GL_2(\R)\rightarrow\C$. Let $g=\smatrix\in\GL_2(\R)$ then define \[|g|:=\sqrt{a^2+b^2+c^2+d^2}.\] We want to find a condition such that $\phi((\begin{smallmatrix} 1 & \\ & y \end{smallmatrix}))$ grows polynomially as $y\rightarrow0$. We also want a condition which is invariant under changing $g$ and $g^{-1}$. Due to this, define $||g||=|g|+|g^{-1}|$. If for appropriate $C$ and $n$ we have 
\begin{equation}\label{grow}
|\phi(g)|\leq C||g||^n, \ \text{for every} \ g\in\GL_2(\R),
\end{equation}
then we say $\phi$ is slowly increasing. One can show using the sub-multiplicity of the Euclidean norm of $\GL_2(\R)$ that if $h_1,h_2\in\GL_2(\R)$ and $\phi$ is slowly increasing in the sense of \eqref{grow} then so is the function $g\mapsto \phi(h_1gh_2)$. We can now define what it means for an adelic function to be slowly increasing. Let $\phi:\GL_2(\A)\rightarrow\C$ be a continuous function. We say $\phi$ is slowly increasing, if for each $h\in\GL_2(\A_\text{fin})$ the function on $\GL_2(\R)$ given by $g\mapsto\phi(gh)$ is slowly increasing as in \eqref{grow}.

We are now able to state the definition of automorphic forms for $\GL_2(\A)$. 

\begin{Defn}\label{auto}
Let $\phi:\GL_2(\A)\rightarrow\C$ be a function with the property that \[\phi(g(\begin{smallmatrix} z & \\ & z \end{smallmatrix}))=\phi(g), \]for every $g\in\GL_2(\A)$ and $z\in\A^\times$. Then $\phi$ is an automorphic form if it satisfies the following properties,
\begin{enumerate}
\item[(i)] $\phi(\gamma g)=\phi(g)$, for every $g\in\GL_2(\A)$ and $\gamma\in\GL_2(\Q)$,
\item[(ii)] $\phi$ is smooth,
\item[(iii)] $\phi$ is right-invariant under $\GL_2(\Z_p)$ for almost all primes $p$,
\item[(iv)] $\phi$ is $K_\infty=\text{SO}(2)$ finite,
\item[(v)] $\phi$ is $\mathcal{Z}$ finite, where $\mathcal{Z}$ denotes the centre of the universal enveloping algebra of $\GL_2^+(\R)$, see \cite[Chapter 2]{DG2},
\item[(vi)] $\phi$ is slowly increasing.
\end{enumerate}
\end{Defn}

If in addition to (i)-(vi) $\phi$ also satisfies the condition \[\int_{\Q\bs\A}\phi\left(\begin{pmatrix}1 & x \\ & 1 \end{pmatrix}g\right)=0, \ \text{for almost every} \ g,\] then we say $\phi$ is cuspidal. 

We can now define the notation of automorphic representations. As the name suggests these will be related to the automorphic forms defined above.  What we will do is construct a representation space whose elements are automorphic forms. 

Let $G$ be a group and $B$ a subgroup of $G$. Let $\chi:B\rightarrow\C^\times$. Given such a $\chi$ the representation of $G$ induced by $\chi$ has a representation space \[V=\left\{\ph:G\rightarrow\C^\times:\ph(bg)=\ph(g), \ \text{for every} \ g\in G, \ b\in B\right\}.\] This means $V$ consists of functions on $G$ with left transformation property. Therefore, the group $G$ acts on this space by right translation. So, we can define a representation $\ro$ of $G$ on $V$ by \[\ro(g)\ph(h)=\ph(hg), \ \text{for every} \ \ph\in V, \ g,h\in G.\] For us, we will have $G=\GL_2(\A)$ and $B=B(\A)$. 

Let $\ph:\GL_2(\A)\rightarrow\C$ which has the transformation property 
\begin{equation}\label{phtran}
\ph\left(\begin{pmatrix} a & b \\ & d \end{pmatrix} g \right)=\left|\frac{a}{d}\right|^{\frac{s+1}{2}}\ph(g), \ g\in\GL_2(\A), \ a,d \in \A^\times , \ b\in\A,
\end{equation}
with $s\in\C$. The space of functions $\ph$ which satisfy \eqref{phtran} is denoted $V(s)$. We know that the global absolute value is invariant under rationals. So \eqref{phtran} implies \[\ph(bg)=\ph(g), \ \text{for every} \ b\in B(\Q) \ \text{and} \ g\in\GL_2(\A).\] Now we can define a new function $\phi:\GL_2(\A)\rightarrow\C$ defined by 
\begin{equation}\label{phi}
\phi(g)=\sum_{\gamma\in B(\Q)\bs\GL_2(\Q)} \ph(\gamma g).
\end{equation}
One can show for Re$(s)$ large enough \eqref{phi} is convergent for any $\ph$ in the space of $V(s)$, see \cite[Chapter 2]{JL70}. Let $W(s)$ be the space of all functions $\phi$ where $\ph$ runs through $V(s)$.  By construction, we have a surjective map
\begin{align*}
V(s)&\rightarrow W(s)\\
\ph&\mapsto \phi.
\end{align*}
If $V(s)$ is an irreducible representation of $\GL_2(\A)$ then the map is also injective. We will always assume that our representations are indeed irreducible. From \eqref{phi} we see that $\phi(\gamma g)=\phi(g)$ for $g\in\GL_2(\A)$ and $\gamma\in\GL_2(\Q)$. This is exactly (i) of Definition \ref{auto}. One can show that the functions $\phi$ in $W(s)$ satisfy all the properties of Definition \ref{auto} and so are automorphic forms. 

\begin{Defn}
\begin{enumerate}
\item[(i)] Let $\pi$ be an irreducible representation of $\GL_2(\A)$. We say that $\pi$ is an automorphic representation if there exists a representation $\pi'$ consisting of automorphic forms, invariant under right translation, such that $\pi\cong\pi'$ as $\GL_2(\A)$ representations. 
\item[(ii)] An automorphic representation, $\pi$, is called a cuspidal automorphic representation if every vector $\phi\in\pi$ is a cuspform.
\end{enumerate}
\end{Defn}

We have that if $\pi$ is cuspidal then it is not induced from $B(\A)$. Therefore, cuspidal automorphic representations are the ones that are of most interest.  Note, that we have not shown the existence of such representations. 

Here, we have given an overview of this setting for a more detailed study of this setting see \cite[Chapter 5]{DGJH}. In the next section, we will see how this adelic representation theory can be viewed in terms of the classical theory described in Section \ref{Class}.

\subsection{The dictionary}\label{gobetween} 

At a first look, these two different settings which we have described do not look so compatible with each other. However, we will see that this is not the case and the relation between these two settings is very explicit. 

We will first see how we can view a classical cuspform as a function on $\GL_2(\A)$ then how one can show a relation between them and representations of $\GL_2(\A)$.  

For ease, we explain this relation in the case when $\chi$ a Dirichlet character is trivial, for $\chi$ non-trivial see \cite[Section 3]{gelbart}.  Let $f\in S_k(\Gamma_0(N))$ be fixed then we can define a function $\phi_f$ on $\GL_2(\A)$ by 
\begin{equation}\label{ftophi}
\phi_f(g)=\det(g_\infty)^{k/2}(ci+d)^{-k}f(g_\infty\cdot i),
\end{equation}
where $g=\gamma g_\infty \kappa_0$, with $\gamma\in\Gamma_0(N), \ g_\infty\in\GL_2^+(\R)$ and $\kappa_0\in K_0(N)=\prod_{p<\infty}K_p(N)$.  The function $\phi_f$ is well defined as we have \[\Gamma_0(N)=\GL_2^+(\R)K_0(N)\cap \GL_2(\Q).\] As $f$ is an element of $S_k(\Gamma_0(N))$, we have that $\phi_f(\gamma g)=\phi_f(g)$, for every $\gamma\in\Gamma_0(N)$.  In general, using strong approximation we have the following.

\begin{Thm}[{\cite[Proposition 3.1]{gelbart}}]\label{fphi}
The map $f(z) \mapsto \phi_f(g)$ is an isomorphism between $S_k(\Gamma_0(N))$ and the space of functions $\phi$ on $\GL_2(\A)$ satisfying the following conditions, 
\begin{enumerate} 
\item[(i)] $\phi(\gamma g)=\phi(g)$, for every $\gamma\in\GL_2(\Q)$,
\item[(ii)] $\phi(g\kappa_0)=\phi(g)$, for every $\kappa_0\in K_0(N)$,
\item[(iii)] $\phi(gr(\theta))=e^{-ik\theta}\phi(g)$, if $r(\theta)=(\begin{smallmatrix} \cos\theta & -\sin\theta \\ \sin\theta & \cos\theta\end{smallmatrix})$,
\item[(iv)] the function $\phi$, viewed as a function of $\GL_2^+(\R)$ alone, satisfies the differential equation \[\Delta\phi=-\frac{k}{2}\left(\frac{k}{2}-1\right)\phi,\]
\item[(v)] $\phi(zg)=\phi(g)$, for every $z\in Z(\A)$,
\item[(vi)] $\phi$ is slowly increasing,
\item[(vii)] $\phi$ is cuspidal, that is, \[\int_{\Q\bs\A}\phi\left(\begin{pmatrix} 1 & x \\ & 1 \end{pmatrix} g\right)=0,\] for almost every $g$.
\end{enumerate}
\end{Thm}
This shows that classical cuspforms can be viewed as automorphic forms on $\GL_2(\A)$. As we saw in Section \ref{Class} the space of cuspforms is made up of oldforms and newforms. Therefore in this new viewpoint, it is reasonable to ask about what happens if $f$ is a newform or not. Can we say more about the relation when $f$ is a newform? To answer this question we need to look at automorphic representations.  

As we saw in the previous section there exist representations $\pi$ of $\GL_2(\A)$ for which automorphic forms are elements of that representation space. Therefore, we should be able to rewrite the idea of the previous theorem into one involving automorphic representations.  To do this consider the following space, \[L^2(\GL_2(\Q)\bs\GL_2(\A))\] and let $\ro$ be the natural representation of $\GL_2(\A)$ on the space \[L^2(\GL_2(\Q)\bs\GL_2(\A)).\] Then denote the restriction of $\ro$ to \[L^2_0(\GL_2(\Q)\bs\GL_2(\A))\] by $\ro_0$.  In Theorem \ref{fphi} we showed that classical cuspforms are specific elements of the space \[L^2_0(\GL_2(\Q)\bs\GL_2(\A)).\] Therefore, this is the space and the representation $\ro$ which are of interest. The representation $\ro$ is not quite the representation we require, because for a representation to be automorphic it must be irreducible, which $\ro$ might not be. However, we have that $\ro_0$ can be broken down into irreducible parts and these will be the ones which we will study. Precisely, $\ro_0$ is the discrete direct sum of irreducible unitary representations each occurring with finite multiplicity. We can show a little more if $\pi$ is one of these irreducible components then $\pi$ has multiplicity one \cite[Theorem 5.7]{gelbart}.

We are now able to state one of the key results, which underpins this area of study. 

\begin{Thm}[{\cite[Theorem 5.19]{gelbart}}]\label{ftopi}
\begin{enumerate}
\item[(i)] Suppose $f\in S_k(\Gamma_0(N))$ is an eigenfunction of $T_p$ for every $p\nmid N$. Define $\pi_f$ to be the representation of $\GL_2(\A)$ determined by the translates of $\phi_f(g)$. Then $\pi_f$ is irreducible. 
\item[(ii)] Conversely, suppose $\pi$ is an irreducible constituent of $\ro_0$. Define $[f_\pi]$ to be the equivalence classes of the functions $f_\pi$ in $S_k(\Gamma_0(N))$ which share the same eigenvalues for every $T_p$ with $p\nmid N$, the conductor of $\pi$ being $N$ and which are such that $\phi_{f_\pi}(g)$ is contained in the space of $\pi$. Then the correspondence between $[f_\pi]$ and $\pi_f$, that is, between equivalence classes of cuspforms in the above sense and irreducible constituents of $\ro_0$ is bijective.
\item[(iii)] The mapping $f\rightarrow\pi_f$ from functions in $S_k(\Gamma_0(N))$ to representations of $\GL_2(\A)$ is bijective when restricted to the subspace of newforms. It is not bijective when restricted to the space of oldforms. 
\end{enumerate}
\end{Thm}

The third part of this theorem explains why for the most part in the work that we carry out we restrict our attention to newforms since we know that given a newform there is a unique representation $\pi_f$ which it maps to.

If $f$ is a newform, then we have $\pi_f=\otimes_{p\leq\infty}\pi_p$, where $\pi_p$ are representations of $\GL_2(\Qp)$. This means that we can turn the global problem about $f$ or $\pi_f$ into one about understanding $\pi_p$ at every prime $p$. Therefore, we can work purely locally and study the representations of $p$-adic groups, which for $\GL_2$ is pretty well understood. 

We now see an example of how this explicit relation between $f$ and $\phi_f$ can be used in terms of Fourier coefficients. The characters $\psi:\Q\bs\A$ are of the form $\psi(\xi x)=\psi_\xi(x)$ where $\xi\in\Q$ is arbitrary. This means that $\phi$ has a Fourier expansion of the form \[\phi\left(\begin{pmatrix}1 & x \\ & 1 \end{pmatrix} g\right)=\sum_{\xi\in\Q}\phi_\xi(g)\psi(\xi x),\] where $\phi_\xi(g)$ are the Fourier coefficients. Now, suppose that $f\in S_k(\SL_2(\Z))$ and $\phi=\phi_f$. Then, for each $y>0$, we have \[\phi_\xi\left(\begin{pmatrix}y & \\ & 1 \end{pmatrix}\right)=
\begin{cases}
a_f(n)e^{-2\pi n y}, \ \text{if} \ \xi=n\in\Z\\
0, \ \text{otherwise}.
\end{cases}
\] As such, \[\phi\left(\begin{pmatrix}1 & x \\ & 1\end{pmatrix} \begin{pmatrix} y & \\ & 1 \end{pmatrix}\right)=\sum_{\xi\in\Q} \phi_\xi((\begin{smallmatrix}y & \\ & 1\end{smallmatrix}))\psi(\xi x)=\sum_{n\in\Z} a_f(n)e^{2\pi inz}=f(z).\] We will see in Chapter \ref{Chap3} how we can generalise this result to a much larger family of automorphic forms.  What this result shows is why we want to study $\phi$ as if we can understand $\phi$ then we can understand $f$. Moreover, understanding $\phi_\xi((\begin{smallmatrix}y & \\ & 1 \end{smallmatrix}))$ allows us to gain knowledge about $a_f(n)$.

This section has been quite technical and it is easy to get lost in the details. So, we review the main strategy that will be employed in later chapters.  Theorem \ref{ftopi} means that we can reformulate (when needed) the questions for classical cuspforms in terms of questions about representations of $\GL_2(\A)$. This, in turn, allows us to use all the tools from representation theory.  Once, we have proven results on the adelic side we then use these to prove the original classical question we were interested in. Of course, in general, there is no guarantee that using this overall approach will be any easier or more fruitful than tackling the original classical problem directly. However, if we can prove results it gives us insight into how we can generalise results to other types of modular forms, such as Hilbert modular forms, this is something we will see in Chapter \ref{Chap3}.

\section{Whittaker model}\label{whittaker}

The aim of this section is to introduce and provide the definition of the local and global Whittaker model for $\Q$. We will see in Chapter \ref{Chap3} how to do this in a more general setting. From Section \ref{gobetween} we have that there is a bijection between classical and adelic newforms, $f\leftrightarrow\phi_f$. The overall aim of this work is to study the arithmetic of the Fourier coefficients of $f$ via studying them adelically. We would also like a method that turns the global classical problem into a local one, where we can study prime by prime. To do this we will use the Whittaker model. This will be one of the most important objects studied in this thesis, the methods we use to investigate the problems discussed in this thesis rely heavily on the properties of the Whittaker model.  We will show that for $\GL_2$ the Whittaker model always exists and is unique.  This is something that is not true for other groups. 

We first construct the local Whittaker model and then use this to construct the global Whittaker model. Let $\psi$ be a non-trivial additive character of $\Qp$ and let $\pi_p$ be an irreducible admissible representation of $\GL_2(\Qp)$, with $p<\infty$. Define $\W(\psi_p)$ as the space of functions satisfying \[W_p\left(\begin{pmatrix} 1 & x \\&1\end{pmatrix} g\right)=\psi_p(x)W_p(g), \ x\in F \ \text{and} \ g\in\GL_2(\Qp).\] Inside $\W(\psi_p)$ there is a unique subspace $\W(\pi_p,\psi_p)$ isomorphic to $\pi_p$ (under the right translation action). This result is known as the multiplicity one theorem a proof of which can be found in \cite[Chapter 3]{Bump} for example. A space of functions $\W(\pi_p,\psi_p)$ satisfying these conditions is called a Whittaker model for the representation $\pi_p$ with respect to $\psi_p$. To utilise this result we also require an analogous result for the archimedean place. For $\GL_2(\R)$ this is described in \cite[Chapter 3]{DG2}. In the case of $\GL_2(\C)$ this can be found in \cite{TadM} and \cite{popa}.

So, now we have a Whittaker model for all places we can think about how to construct a global Whittaker model. As one might suspect this will just be the product of all the local Whittaker functions. With this in mind, it will look very similar to the local model. Let $\pi$ be an irreducible admissible representation of $\GL_2(\A)$. Then we mean  a Whittaker model of $\pi$ with respect to $\psi$ a non-trivial character of $\A/\Q$, the space of smooth $K$-finite functions on $\GL_2(\A)$ satisfying \[W\left(\begin{pmatrix} 1 & x \\ & 1 \end{pmatrix} g \right)=\psi(x)W(g).\] It is assumed that the functions $W\in \W(\pi,\psi)$ are slowly increasing. 

Let $\pi$ be an irreducible admissible representation of $\GL_2(\A)$. Then $\pi$ has a Whittaker model $\W(\pi,\psi)$ with respect to $\psi$ an additive character $\A/\Q$ if and only if $\pi_p$ has a Whittaker model $\W(\pi_p,\psi_p)$ with respect to $\psi_p$ an additive character of $\Qp$. If this is the case then $\W(\pi,\psi)$ is unique and consists of all finite linear combinations of functions of the form \[W(g)=\prod_{p\leq\infty} W_p(g_p),\] where $W_p\in\W(\pi_p,\psi_p)$ and $W_p=W_p^\circ$ for almost all $p$, where $W_p^\circ$ is the spherical element of $\W(\pi_p,\psi_p)$ normalised so that $W_p^\circ(k_p)=1$ for $k_p\in\GL_2(\Z_p)$.

The if part of this statement is proved by using the properties of the local Whittaker model. The only if part is done by saying if $\pi$ has a Whittaker model it must be the one described above, meaning that if $\pi$ does have a Whittaker model it is unique. For a full proof see \cite[Theorem 3.5.4]{Bump}. 

We have now come to the important case, for us, when $\pi$ is a cuspidal automorphic representation. As one might expect cuspidal automorphic representations for $\GL_2$ do have Whittaker models. However, if we consider other groups this might not necessarily be the case, for example, this is not true for the group $\GSp_4$. 

\begin{Thm}[Existence of Whittaker model for automorphic representations]
Let $\pi$ be a cuspidal automorphic representation of $\GL_2(\A)$. If $\phi\in\pi$ and $g\in\GL_2(\A)$ let 
\begin{equation}\label{WphiQ}
W_\phi(g)=\int_{\A/\Q}\phi\left(\begin{pmatrix} 1 & x \\ & 1 \end{pmatrix}g\right) \psi(-x) \ dx.
\end{equation}
 Then the space $\W(\pi,\psi)$ of functions $W_\phi$ is a Whittaker model for $\pi$. We have the expansion 
\begin{equation}\label{phiWQ}
\phi(g)=\sum_{\alpha\in\Q^\times}W_\phi\left(\begin{pmatrix} \alpha & \\ & 1 \end{pmatrix}g\right)
\end{equation}
\end{Thm}
As this result is rather important for our study we provide a proof of it here, following \cite{Bump}.
\begin{proof}
The function \[F(x)=\phi\left(\begin{pmatrix} 1 & x \\ & 1 \end{pmatrix}g\right), \ x\in\A,\] is continuous and since $\phi$ is an automorphic form we have $F(x+\alpha)=F(x)$, if $\alpha\in\Q$. So, we can view $F$ as a function on the compact group $\A/\Q$. Each such character has the form $x\mapsto \psi(\alpha x)$, where $\alpha\in\Q$. This means we can write $F$ as \[\phi\left(\begin{pmatrix} 1 & x \\ & 1 \end{pmatrix}g\right)=\sum_{\alpha\in\Q} C(\alpha)\psi(\alpha x),\] where the Fourier coefficients are \[C(\alpha)=\int_{\A/\Q}\phi\left(\begin{pmatrix} 1 & x \\ & 1 \end{pmatrix}g\right)\psi(-\alpha x) \ dx.\] Considering $\alpha=0$, we see that $C(0)=0$, since $\phi$ is cuspidal. As such, we can restrict to $\alpha\in\Q^\times$. Using the fact that $\phi$ is automorphic, specifically (i) from Definition \ref{auto}, we have 
\begin{align*}
C(\alpha)&=\int_{\A/\Q}\phi\left(\begin{pmatrix} \alpha & \\ & 1 \end{pmatrix} \begin{pmatrix} 1 & x \\ & 1 \end{pmatrix}g\right) \psi(-\alpha x) \ dx\\
&=\int_{\A/\Q}\phi\left(\begin{pmatrix} 1 & \alpha x\\ & 1 \end{pmatrix} \begin{pmatrix} \alpha &  \\ & 1 \end{pmatrix}g\right) \psi(-\alpha x) \ dx
\end{align*}
Now we can make a change of variables $x\mapsto \alpha^{-1}x$, which is unimodular since $\alpha\in\Q$. Therefore, 
\begin{align*}
C(\alpha)&=\int_{\A/\Q}\phi\left(\begin{pmatrix} 1 & \alpha\alpha^{-1} x\\ & 1 \end{pmatrix} \begin{pmatrix} \alpha &  \\ & 1 \end{pmatrix}g\right) \psi(-\alpha\alpha^{-1} x) \ dx\\
&=\int_{\A/\Q}\phi\left(\begin{pmatrix} 1 &  x\\ & 1 \end{pmatrix} \begin{pmatrix} \alpha &  \\ & 1 \end{pmatrix}g\right) \psi(- x) \ dx\\
&=W_\phi\left(\begin{pmatrix} \alpha & \\ & 1 \end{pmatrix} g\right), \ \text{by using \eqref{WphiQ}}.
\end{align*}
Therefore, \[\phi\left(\begin{pmatrix} 1 & x \\ & 1 \end{pmatrix}g\right)=\sum_{\alpha\in\Q} C(\alpha)\psi(\alpha x)=\sum_{\alpha\in\Q} W_\phi\left(\begin{pmatrix} \alpha & \\ & 1 \end{pmatrix} g\right)\psi(\alpha x).\] If we now set $x=0$, we obtain \eqref{phiWQ}. The properties that $\phi$ has from being a cuspidal automorphic form show that $W_\phi$ is a Whittaker model for $\pi$.
\end{proof}
Note, in addition, that this global Whittaker model is unique. We also have for $g=(g_p)_p\in\GL_2(\A)$ that \[W_\phi(g)=\prod_{p\leq\infty}\Wp(g_p).\] In the specific case $\pi$ is an automorphic representation of $\GL_2$ we denote $W_\phi$ to be the global Whittaker function and $\Wp$ for the local Whittaker function. This means understanding properties of $\phi$ can be reinterpreted into understanding $W_\phi$ which itself can be understood by studying $\Wp$. This is exactly what we were looking for as we can turn our global study into a local one. 

The final thing which we need to define in this section is the (local) Whittaker newform. The normalised Whittaker newform is the unique $K_p(p^{c(\pi_p)})$-invariant vector $\Wp\in\W(\pi_p,\psi_p)$ satisfying $\Wp(1)=1$.

The work carried out in this thesis exploits the bijection discussed in Section \ref{gobetween} and the Whittaker model described above. Precisely, we want to construct (or use) an explicit formula for the Fourier coefficients of modular forms in terms of adelic objects. In the setting of modular forms this was done in \cite[Proposition 3.3]{ACAS}, they prove,

\begin{equation}\label{af}
\af=\frac{a_f(n_0)}{n_0^{k/2}}\Big(\frac{n}{\delta(\c)}\Big)^{k/2}\prod_{p|N}\Wp \left(\begin{pmatrix}n/\delta(\c)& \\ & 1\end{pmatrix}\sigma^{-1}\right).
\end{equation}
Having this explicit formula allows us to reinterpret our arithmetic questions about the Fourier coefficients of classical newforms to that of the arithmetic of local Whittaker newforms. We construct an analogous formula to \eqref{af}  in Chapter \ref{Chap3} for Hilbert modular forms. In Chapter \ref{Chap4} we give the heuristics on how one can use \eqref{af} to obtain lower bounds for the $p$-adic valuation of $\af$.

\subsection{$p$-adic valuation}

In Chapter \ref{Chap2}, we explicitly compute lower bounds for the $p$-adic valuation of the local Whittaker newforms. The motivation for studying this problem is that we are interested in computing lower bounds for the $p$-adic valuations of the Fourier coefficients of modular forms at cusps. From \eqref{af} we see that if we can understand the $p$-adic valuation of $\Wp$ then this enables us to understand the $p$-adic valuation of $\af$.  

As we will be working purely locally in that chapter, instead of just studying $\Qp$ we can move to a finite extension of $\Qp$,  denoted $F$.  As mentioned above we know that the Whittaker model for $\GL_2(F)$ exists and is unique.  Let $\pi$ denote a unitary, irreducible, admissible representation of $\GL_2(F)$ then this question of computing $p$-adic valuations was considered by \u{C}esnavi\u{c}us, Neururer and Saha in \cite{MCMD}, where for their applications it was sufficient to consider the case of $\pi$ having trivial central character. We consider the case of non-trivial central character. We will impose that the central character of $\pi$ is an element of $\Hom_{\text{cts}}(\O_F^\times, \C^\times)$, where $\O_F^\times$ is the group of units of $F$. In this setting, all the representations of $\GL_2(F)$ are classified. This means that for $\pi$ being of a specific type, we have a Whittaker model and so we can study its properties. This is done by writing out explicitly the values of $W_\pi(g)$, for specific $g\in\GL_2(F)$.

Our method will rely on establishing lower bounds for the $p$-adic valuation for the Fourier coefficients of the local Whittaker newform. Once these results have been proven, in Chapter \ref{Chap4}, we explain how one can use them to compute $p$-adic valuations of the Fourier coefficients of newforms at cusps.

\section{Hilbert modular forms}\label{hilbert}

All of the discussion so far has been working over $\Q$. There is no reason why we cannot study the objects previously mentioned in more general settings, such as number fields. Indeed, many of the questions over $\Q$ can be reformulated to the setting of number fields. 

As we are interested in generalising modular forms it makes sense to consider the case of $F$ being a totally real number field of degree $n$. This is because there exists a natural generalisation of the action of $\GL_2^+(\R)$ on $\H$ to \[\GL_2^+(F)=\left\{g\in\GL_2(F): \det g \ \text{totally positive}\right\}\] on $\H^n$.

Due to this natural action, we can define modular forms over totally real number fields, these are called Hilbert modular forms. Statements for modular forms over $\Q$ can often be reformulated for Hilbert modular forms. In some situations the case of $\Q$ is resolved but for number fields remain open. If the problem stated for $\Q$ is proven, the method can be quite instructive and give an idea of how one should prove the case over number fields. Note, this might not always be such a fruitful approach as the work in Chapter \ref{Chap3} will show. 

An example that we have already seen is the modularity of elliptic curves over $\Q$.  The analogous statement would be replacing elliptic curves over $\Q$ with those over $F$ and newforms with Hilbert newforms. In the case of totally real number fields, this problem is wide open. A result of Freitas, Le Hung and Siksek \cite{FHS} show that any elliptic curve over a real quadratic number is modular. For number fields with degree greater than two, little is known. We have some results for specific number fields but few results for whole families of number fields. 

The work carried out in Chapter \ref{Chap3} looks at the algebraic properties of the Fourier coefficients of Hilbert modular forms. The way we will prove our result is by translating the problem into the adelic viewpoint like that described in Section \ref{ctoa}.  As such, there are some technical details that can cause confusion on first reading, especially for those unfamiliar with Hilbert modular forms and/or the adelic interpretation. Therefore, here we give a brief introduction to Hilbert modular forms and also an overview of the method used to prove our main result in Chapter \ref{Chap3}. 

Let $F$ be a totally real number field of degree $n$ and ring of integers $\O_F$. The size of the narrow class group is denoted by $h$ and $1\leq\mu\leq h$. For this basic introduction let the narrow class group of $F$ be of size one, we will refer to what happens in the general case. The work in Chapter \ref{Chap3} has no assumption on the size of the narrow class group, just that the central character of the automorphic representation is trivial. The $n$ embeddings of $F$ into $\R^n$ we denote by $\eta_1,...,\eta_n$. Then $\GL_2^+(F)$ acts on $\H^n$ by \[g\cdot z=\bigg(\frac{\eta_1(a_1)+\eta_1(b_1)}{\eta_1(c_1)z_1+\eta(d_1)},...,
\frac{\eta_n(a_n)+\eta_n(b_n)}{\eta_n(c_n)z_n+\eta_n(d_n)}\bigg).\] For an integral ideal $\n$ of $\O_F$ we can define the congruence subgroup $\Gamma(\n)<\GL_2^+(F)$ as \[\Gamma(\n)=\left\{\begin{pmatrix} a & b \\  c & d \end{pmatrix}: a,d \in\O_F, b\in\mathfrak{D}_F^{-1}, c\in\n \mathfrak{D}_F,ad-bc\in\O_F^\times\right\},\] where $\D_F$ is the absolute different of $F$.  This is one of the differences for this setting is that we need to consider congruence subgroups for $\GL_2^+(F)$. This is due to the fact that in general these congruent subgroups are indexed by the narrow class number. In this case, we would alter the definition by having $b\in(t_\mu)^{-1}\O_F\D_F^{-1}$ and $d\in\n t_\mu\O_F\D_F$ where $t_\mu$ an element of the narrow class group. For a function $\f:\H^n\rightarrow\C$ and $k=(k_1,...,k_n)\in\Z^n$ we can define a function \[\f||_k\alpha(z):=(\det\alpha)^{k/2}(cz+d)^{-k}\f(\alpha\cdot z),\] with $\alpha=\smatrix\in\Gamma(1)$.  

\begin{Defn}[Hilbert modular form,  $h=1$]
A holomorphic function $\f:\H^n\rightarrow\C$ is a Hilbert modular form for $\Gamma(\n)$ and weight $k\in\Z^n$, if \[\f||_k\alpha=\f, \ \text{for every} \ \alpha\in\Gamma(\n).\] 
\end{Defn}
Note that, in the case of $n>1$ the condition for $\f$ to be holomorphic at the cusps is automatic by the Koecher's principle, see \cite[Section 1.4]{HolHMF}.  Just as in the case of modular forms, we have that Hilbert modular forms have a Fourier expansion. Let $\Gamma(\n)$ be a congruence subgroup, then define \[\Lambda=\left\{\nu\in F:\begin{pmatrix} 1 & \nu \\ & 1 \end{pmatrix} \in \Gamma(\n)\right\}.\] Then $\f$ has a Fourier expansion of the form \[\f(z)=\sum_{\xi\in\Lambda^\ast} a_\f(\xi)e^{2\pi i \Tr(\xi z)},\] where \[\Lambda^\ast=\left\{\nu\in F:\Tr(\nu\Lambda)\subset\Z\right\}.\] A Hilbert modular form, $\f$, is a cuspform of weight $k$ with respect to $\Gamma(\n)$ if for every $g\in\GL_2^+(F)$ the Fourier expansion \[\f||_kg(z)=\sum_\xi a_\f(\xi;g)e^{2\pi i \Tr(\xi z)},\] has $a_\f(\xi;g)=0$ unless $\xi$ is totally positive.  In a similar way to how we defined newforms in Section \ref{ctoa} we have Hilbert newforms (which we define in Chapter \ref{Chap3}), $\f$ for $\n$ and weight $k$. In general, we will have that Hilbert newforms are a vector of cuspforms indexed by the size of the narrow class group. Precisely, we have $\f=(f_1,...,f_h)$, where $f_\mu$ are cuspforms of weight $k$ with respect to $\Gamma_\mu(\n)$. 

From now on we make no assumption on the size of the narrow class group. Let $a_\mu(\xi)$ be the Fourier coefficients of $f_\mu$ and define \[c_\mu(\xi;f_\mu)=N(t_\mu\O_F)^{-k_0/2}a_\mu(\xi)\xi^{(k_0\1-k)/2},\] where $k_0=\max\{k_1,...,k_n\}$.  Let $\Q(\f)$ denote the field generated by $c_\mu(\xi;f_\mu)$ as $\xi$ varies over $F$ and $\mu$ varies over $1\leq\mu\leq h$. This is a number field under the assumption that $k_1\equiv...\equiv k_n \pmod{2}$. We will be interested in the coefficients $c_\mu(\xi;f_\mu||_k\sigma)$ where $\sigma\in\Gamma_\mu(1)$. Specifically, an application of the $q$-expansion principle says that the field generated by $c_\mu(\xi;f_\mu||_k\sigma)$, as $\xi$ varies over $F$ and $\mu$ varies $1\leq\mu\leq h$,  lies in some cyclotomic extension of $\Q(\f)$. The work in Chapter \ref{Chap3} is all about studying this cyclotomic extension. 

We now give an overview of Chapter \ref{Chap3} which the reader can refer back to whilst reading that chapter. In Chapter \ref{Chap3}, as mentioned above the coefficients $c_\mu(\xi;f_\mu||_k\sigma)$ and the number field which they generate are the objects of interest. We know by the $q$-expansion principle that they must lie in the number field $\Q(\f)(\zeta_{N'})$, where $N'$ is the smallest integer such that $N'\Z=\n\cap\Z$.  We can prove that they actually lie in a smaller cyclotomic extension of $\Q(\f)$. 

For $\f$ a Hilbert newform, we can associate a representation $\Pi=\prod_{v\leq\infty}\Pi_v$, where the product is over all places of $F$. We will prove our result by moving to the adelic side and studying $\Pi$ and the Whittaker model associated to $\Pi$.

The method used to prove the main result in Chapter \ref{Chap3}, will rely on rewriting the Fourier coefficients of $c_\mu(\xi,f_\mu||_k\sigma)$ in terms of the global Whittaker newform. We can then break up the global Whittaker newform into a product of local Whittaker newforms. As a result, we can obtain an analogous formula to that of \eqref{af} but for Hilbert modular forms. 

The reason why we do this is because we know how $\tau\in\Aut(\C)$ acts on $\Wv$, where $\Wv$ is the Whittaker newform for $\Pi_v$. We then find sufficient conditions on $g\in\GL_2(F_v)$ (which will be a specific matrix) such that \[\tau (\Wv(g))=\Wv(g).\] Finding these conditions on $g$ will give us the main result. 

The question of finding the field generated by the Fourier coefficients of modular forms at cusps was already investigated by Brunault and Neururer in \cite{FEatC} for the case of $F=\Q$. The methods they use are classical in nature and do not use any properties of Whittaker newforms. A reasonable question to ask is why we do not just generalise their method to Hilbert modular forms.  In general, a Hilbert newform is a $h$-tuple $\f=(f_1,...,f_h)$ and so studying $\f$ and $c_\mu(\xi;f_\mu||_k\sigma)$ directly can be challenging, especially in the case of $h>1$.  Whilst,  the method we employ is not affected by the size of the narrow class group, it only appears in the element $g$ which $\Wv$ is evaluated at.  Another advantage to our method is that we obtain an explicit formula relating $c_\mu(\xi;f_\mu||_k\sigma)$ to $\Wv$. This explicit formula could then be used to prove other results.

The slight disadvantage to our method is that we find sufficient conditions, not necessary ones. This means we do not know if this field is the one generated by the Fourier coefficients or not. Note, that this is not immediate in the work of \cite{FEatC} and requires knowledge about the structure of congruence subgroups of $\SL_2(\Z)$. A brief discussion of this problem in our setting is provided in Chapter \ref{Chap4}.


\chapter{$p$-adic Valuation of Local Whittaker Newforms}\label{Chap2}
\chaptermark{Local Whittaker Newforms}

\section{Motivation}

Let $f$ be a normalised newform of level $N$ and weight $k$. Let $\sigma\in\SL_2(\Z)$ and $\c=\sigma\infty$ the corresponding cusp. Recall that $f|_k\sigma$ has a Fourier expansion of the form \[f|_k\sigma(z)=\sum_{n\geq0} \af e^{2\pi inz/w(\sigma)}, \ z\in\H.\] Let $p$ be an odd prime, then for every $q\in\Q$ we can write this as $q=\frac{a}{b}p^r$, for some $r\in\Z$ and $(a,b)=1$. Then $\valp(q):=r$, which can extended to $\overline{\Qp}\cong\C$. Let $\valp:\overline{\Qp}\rightarrow\Q\cup\{\infty\}$. Then, the question we are interested in is computing lower bounds for the $p$-adic valuation of the Fourier coefficients $\af$.  Specifically, we are interested in bounding \[\valp(f|_\c):=\inf_{n\geq0}\{\valp(\af)\}.\] From Section \ref{ctoa}, we have seen that if $f$ is a normalised newform then $f$ corresponds to a representation $\pi_f$ of $\GL_2(\A)$. By considering \eqref{af} we have one way to tackle the problem of computing $\valp(f|_\c)$ is by studying the $p$-adic valuation of the right-hand side of \eqref{af}. That is, we need to understand the $p$-adic valuation of \[\frac{a_f(n_0)}{n_0^{k/2}}\Big(\frac{n}{\delta(\c)}\Big)^{k/2}\prod_{p|N}\Wp\left( \begin{pmatrix}n/\delta(\c) & \\ & 1 \end{pmatrix}\sigma^{-1}\right).\] This means computing the $p$-adic valuation of $\af$, we need to study the $p$-adic valuation of the local Whittaker newform. If we consider the case of $f$ a newform for $\Gamma_0(N)$ then the corresponding representation $\pi_f$ has trivial central character. This case was studied by \u{C}esnavi\u{c}us, Neururer and Saha in \cite{MCMD}.

The work from now on in this chapter is purely local and so to ease the notation we remove any notation that refers to local places. In Chapter \ref{Chap4} we give an outline of how this local work can be used to help answer a global question, related to the Fourier coefficients of modular forms. As this is now a purely local question we can replace $\Qp$ with any finite extension $F$ of $\Qp$. The residue field of $F$ is denoted by $\F$ and the order of $\F$ is denoted by $q_F$. For any non-archimedean local field, we have that the Whittaker model exists and is unique for $\GL_2(F)$. Let $W_\pi$ be the normalised Whittaker newform associated to $\pi$, with $\pi$ being an irreducible, admissible, infinite-dimensional representation of $\GL_2(F)$. Let \[\gtlv=\begin{pmatrix} & \varpi_F^t \\ -1 & -v\varpi_F^{-l} \end{pmatrix}\in\GL_2(F), \ \text{for} \ t,l\in\Z \ \text{and} \ v\in\O_F^\times,\] where $\varpi_F$ is a uniformiser of $\O_F$.  Due to the decomposition of $\GL_2(F)$ we can restrict our attention to $W_\pi(\gtlv)$. Then, in \cite[Section 3]{MCMD} under the assumption that the central character is trivial, they prove.

\textit{For a finite extension $F/\Qp$, with $p$ odd, an irreducible, admissible, infinite-dimensional representation $\pi$ of $\GL_2(F)$ with $c(\pi)\geq2$ and trivial central character, an additive character $\psi:F\rightarrow\C^\times$ with $c(\psi)=0$, an isomorphism $\C\cong\overline{\Qp}$, a $t\in\Z$, an $0\leq l\leq c(\pi)$, and a $v\in\O_F^\times$, we have \[\valp(W_\pi(\gtlv))\geq 
\begin{cases}
0, \ \text{if} \ l\in\{0,c(\pi)\},\\
0, \ \text{if} \ l\in\{1,c(\pi)-1\}, \ \text{if} \ c(\pi)>2,\\
[\F:\F_p]\left(1-\frac{\min\{l,c(\pi)-l\}}{2}\right), \ \text{if} \  l\in\{0,1,c(\pi)/2,c(\pi)-1, c(\pi)\},\\
-[\F:\F_p]+\min\left\{\frac{[\F:\F_p]}{2},\frac{1}{2}+\frac{1}{p-1}\right\},\ \text{if} \  l=1, c(\pi)=2, t=-2,\\
[\F:\F_p](1-c(\pi)/4),  \ \text{if} \ l=c(\pi)/2, c(\pi)>2, t=-c(\pi),
\end{cases}\]
and for $l=c(\pi)/2$ and $c(\pi)$ even the following additional bounds hold
\begin{enumerate}
\item[(i)] if $\pi$ is supercuspidal with $c(\pi)=2$, then \[\valp(W_\pi(g_{t,1,v}))\geq-[\F:\F_p]+\frac{1}{2}+\frac{1}{p-1},\]
\item[(ii)] if $\pi$ is twist of the Steinberg with $c(\pi)=2$, then \[\valp(W_\pi(g_{t,1,v}))\geq-\frac{t+4}{2}[\F:\F_p]+\min\left\{-[\F:\F_p]\left(\frac{t+1}{2}\right),\frac{1}{2}+\frac{1}{p-1}\right\},\]
\item[(iii)] if $\pi$ is a principal series of the form $\mu| \cdot |_F^\vartheta\boxplus\mu|\cdot|_F^{-\vartheta}$, with $\mu^2=\1$, then $c(\pi)=2$ and \[\valp(W_\pi(g_{t,1,v}))\geq-[\F:\F_p]-(t+2)|\valp(q_F^\vartheta)|+\min\left\{-[\F:\F_p]\left(\frac{t+1}{2}\right),\frac{1}{2}+\frac{1}{p-1}\right\},\]
\item[(iv)] if $\pi$ is a principal series of the form $\mu| \cdot |_F^\vartheta\boxplus\mu^{-1}|\cdot|_F^{-\vartheta}$, with $\mu^2\neq\1$, then \[\valp(W_\pi(g_{t,l,v}))\geq
\begin{cases}
-\frac{[\F:\F_p](t+4)}{2}+\frac{1}{2}+\frac{1}{p-1}-(t+2)|\valp(q_F^\vartheta)|, \ c(\pi)=2,\\
-\frac{[\F:\F_p]\max\{t+c(\pi),c(\pi)/2-2\}}{2}-(t+c(\pi))|\valp(q_F^\vartheta)|, \ c(\pi)>2.
\end{cases}
\]
\end{enumerate}
}

They also obtain similar bounds for $p=2$. We do not study this case here due to some technical complications. A short discussion on this problem is given in Section \ref{unQ} and the subtle difficulties that appear. 

This leads to asking if we can obtain similar bounds for when we have a non-trivial central character. For a finite extension $F/\Qp$ let $\pi$ be a unitary, irreducible, admissible representation of $\GL_2(F)$ with central character $\cc$. Let $\X$ denote the set of multiplicative characters of $\chi:F^\times\rightarrow\C^\times$ such that $\chi(\varpi_F)=1$. Assuming that $\cc\in\X$, then we prove.

\begin{IntroThm}(Theorem \ref{api2})\label{pi2}
For a finite extension $F/\Qp$, with $p$ odd, an irreducible, admissible, infinite-dimensional representation $\pi$ of $\GL_2(F)$ with $c(\pi)=2$ and $\cc\in\X$. An additive character $\psi:F\rightarrow\C^\times$ with $c(\psi)=0$, an isomorphism $\C\cong\overline{\Qp}$, a $t\in\Z$, an $0\leq l\leq 2$, and a $v\in\O_F^\times$. We have 
\begin{enumerate}
\item[(i)] if $\pi$ is supercuspidal, then \[\valp(W_\pi(\gtlv))\geq-[\F:\F_p],\]
\item[(ii)] if $\pi$ is twist of Steinberg, then  \[\val_p(W_{\pi}(g_{t,l,v}))\geq -(t+3)[\F:\F_p]+\frac{1}{p-1},\]
\item[(iii)] if $\pi$ is a principal series representation of the form $\bone|\cdot|^{c_1}\boxplus\btwo|\cdot|^{c_2}$, with $\bone\neq\btwo$, then \[\valp(W_\pi(\gtlv)\geq-\frac{t+4}{2}[\F:\F_p]+\frac{2}{p-1}-(t+2)|\valp(\qcone)|,\]
\item[(iv)] if $\pi$ is principal series representation of the form $\bone|\cdot|^{c_1}\boxplus\btwo|\cdot|^{c_2}$, with $\bone=\btwo$, then \[\valp\left(W_\pi(\gtlv)\right)\geq-\frac{t+4}{2}[\F:\F_p]+\frac{1}{p-1}-(t+2)|\valp(\qcone)|.\]
\end{enumerate}
\end{IntroThm}

In the setting of $c(\pi)>2$ we prove.

\begin{IntroThm}(Theorem \ref{apiplus})\label{pi2plus}
For a finite extension $F/\Qp$, with $p$ odd, an irreducible, admissible, infinite-dimensional representation $\pi$ of $\GL_2(F)$ with $c(\pi)>2$ and $\cc\in\X$. An additive character $\psi:F\rightarrow\C^\times$ with $c(\psi)=0$, an isomorphism $\C\cong\overline{\Qp}$, a $t\in\Z$, an $0\leq l\leq c(\pi)$, and a $v\in\O_F^\times$. We have 
\begin{enumerate}
\item[(i)] if $\pi$ is supercuspidal, then \[\valp(W_\pi(\gtlv))\geq(1-l/2)[\F:\F_p],\]
\item[(ii)] if $\pi$ is twist of Steinberg, then \[\valp\left(W_\pi(\gtlv)\right)\geq\min\left\{-\frac{c(\pi)}{2}[\F:\F_p]+\frac{2}{p-1},-\left(2+t+\frac{c(\mu^{-1})}{2}\right)[\F:\F_p]\right\},\]
\item[(iii)] if $\pi$ is principal series representation of the form $\bone|\cdot|^{c_1}\boxplus\btwo|\cdot|^{c_2}$, with $\bone\neq\btwo$, 
\begin{enumerate}
\item if both of $c(\bone)$ and $c(\btwo)$ are greater than one, then 
\begin{align*}
\valp(W_\pi(\gtlv))&\geq\min\Bigg\{-\frac{c(\pi)}{2}[\F:\F_p]+\min\left\{\frac{1}{p-1},[\F:\F_p]-c(\pi)|\valp(\qcone)|\right\},\\
&-\frac{t+c(\bone^{-1})+1+c(\bi^{-1}\bj)}{2}[\F:\F_p]+\frac{1}{p-1}\\
&-(t+c(\bi^{-1}\bj)+1)|\valp(\qci)|\Bigg\}\\
&+
\begin{cases}
-|c(\bone^{-1})-c(\btwo^{-1})|\cdot|\valp(\qcone),\ \text{if} \ \valp(\eps(1/2,\pi))>0,\\
0, \ \text{if} \ \valp(\eps(1/2,\pi))\leq0,
\end{cases}
\end{align*}
\item if one of $c(\bone)$ or $c(\btwo)$ is equal to one, then 
\begin{align*}
\valp\left(W_\pi(\gtlv)\right)&\geq\min\Bigg\{-\frac{c(\pi)}{2}[\F:\F_p]+\frac{2}{p-1}, \\
&-\frac{[\F:\F_p]}{2}+\frac{1}{p-1}-(c(\pi)-2)|\valp(\qcone)|,\\
&-\frac{t+3+c(\bi^{-1}\bj)}{2}[\F:\F_p]+\frac{2}{p-1}\\
&-(t+c(\bi^{-1}\bj)+1)|\valp(\qci)|\Bigg\}\\
&+
\begin{cases}
-(c(\pi)-2)|\valp(\qcone)|-\frac{[\F:\F_p]}{2}, \ \text{if} \ \valp(\eps(1/2,\pi))>0,\\
0, \ \text{if} \ \valp(\eps(1/2,\pi))\leq0.
\end{cases}
\end{align*}
\end{enumerate}
\item[(iv)] if $\pi$ principal series representation of the form $\bone|\cdot|^{c_1}\boxplus\btwo|\cdot|^{c_2}$, with $\bone=\btwo$, then 
\begin{align*}
\valp\left(W_\pi(\gtlv)\right)\geq\min\bigg\{&-\frac{c(\pi)[\F:\F_p]}{2}+\frac{2}{p-1},\\
&-\frac{(t+c(\beta^{-1})+2)[\F:\F_p]}{2}-(t+2)|\valp(\qcone)|\bigg\}.
\end{align*}
\end{enumerate}
\end{IntroThm}

The proof for these results as well as the one stated above from \cite{MCMD} follow a very similar method. To obtain these bounds we use the local Fourier expansion of $W_\pi$ and analyse the Fourier coefficients $\ctl$. The work of Assing \cite{ASS19} gives under the assumption that $\cc\in\X$ explicit formulas for $\ctl$ (which we state in Section \ref{formulas}). These formulas involve Gauss sums which can be rewritten in terms of $\GL_1$ $\eps$-factors. The study of the $p$-adic valuation of these $\eps$-factors was carried out in \cite[Section 2]{MCMD}, where the assumption on $\pi$ having trivial central character is not needed. The main result for us from this work is \cite[Corollary 2.7]{MCMD} which tells us explicitly what the $p$-adic valuation of these $\GL_1$ $\eps$-factors is equal to. 

\subsection{Notation}

Here we collect the notation that will be used throughout this chapter. The majority of our notation follows that defined in \cite{MCMD}.

For $p$ a prime we denote $\val_p:\overline{\Q}_p\rightarrow\Q\cup\{\infty\}$ to be the $p$-adic valuation with $\val_p(p)=1$. We denote $F$ to be a non-archimedean field with ring of integers $\O_F$. We denote the maximal ideal of $\O_F$ by $\m_F$ and $\varpi_F$ a uniformiser. We denote the residue field $\mathbb{F}_F=\O_F/\m_F$ and $q_F:=|\mathbb{F}_F|$. We normalise the absolute value $|\cdot|_F$ on $F$ by $|\varpi_F|_F=1/q_F$. We set \[\zeta_F(s)=\frac{1}{1-q_F^{-s}},\] for which we only require \[\zeta_F(1)=\frac{q_F}{q_F-1} \ \text{and} \ \zeta_F(2)=\frac{q^2_F}{q^2_F{-1}}.\] To ease the notation, when clear,  we will drop the subscript $F$.

We let $\chi: F^\times\rightarrow\C^\times$ be a continuous character. We denote the conductor exponent of $\chi$ by $c(\chi)$, that is, $c(\chi)=n>0$ with $\chi(1+\m^n)=1$ if $\chi(\O_F^\times)\neq1$ and $c(\chi)=0$ if $\chi(\O_F^\times)=1$. If $c(\chi)=0$ we say $\chi$ is unramified. Let $\psi: F\rightarrow\C^\times$ be an additive character. Let $c(\psi)$ be the smallest integer $n$ such that $\psi(\m^n)=\{1\}$ (we will often normalise so that $c(\psi)=0$).  

\subsection{Local field Gauss sums}

For a finite extension $F/\Qp$, a multiplicative character $\chi:F^\times\rightarrow\C^\times$, a non-trivial additive $\psi:F\rightarrow\C^\times$, then the Gauss sum of $\chi$ with respect to $\psi$ is defined by \[G_\psi(x,\chi):=\int_{\O_F^\times}\chi(y)\psi(xy)\ d^\times y, \ \text{for} \ x\in F^\times, \ \text{with normalisation} \ \int_{\O_F^\times} d^\times y=1.\] Note that $G_\psi(x,\chi)$ only sees $\chi|_{\O_F^\times}$ so it is not affected by multiplying $\chi$ by an unramified character and so without loss of generality we can suppose $\chi$ is an element of \[\mathfrak{X}:=\{\text{continuous character} \ \chi: F^\times\rightarrow\C^\times \ \text{with} \ \chi(\varpi)=1\}\cong\Hom_{\text{cts}}(\O_F^\times,\C^\times).\] Characters in $\X$ are unitary and finite order. Consider the subsets of $\X$ defined by, \[\X_{\leq k}:=\{\chi\in\X:c(\chi)\leq k\} \ \text{and} \ \X_k:=\{\chi\in\X:c(\chi)=k\}.\] The Gauss sum $G_\psi(x,\chi)$ is related to $\eps$-factors of $\GL_1$, denoted $\eps(s,\chi,\psi)$, since in both of these $\psi$ is fixed we suppress the dependence of $\psi$ in the notation. The only expectation to this is when $\psi$ is composed with another character in which case we write the notation including that specific character. Under the common normalisation, $c(\psi)=0$, we have 
\begin{equation}\label{gauss}
G(x,\chi)=
\begin{cases} 1, \ \text{if} \  c(\chi)=0 \ \text{and} \ \val_F(x)\geq0,\\ 
-\frac{1}{q-1}, \ \text{if} \ c(\chi)=0 \ \text{and} \val_F(x)=-1,\\ 
\frac{q^{1-c(\chi)/2}}{q-1}\eps(1/2, \chi^{-1})\chi(x^{-1}), \ \text{if} \ c(\chi)>0 \ \text{and} \ \val_F(x)=-c(\chi),\\
0, \ \text{otherwise}.
\end{cases}
\end{equation}
Let $\chi,\chi'$ be multiplicative characters then, whenever $c(\chi')=0$, we have
\begin{equation}\label{twoc}
\eps(1/2,\chi\chi')=\chi'(\varpi)^{c(\chi)}\eps(1/2,\chi).
\end{equation}
For any $\chi$, we have 
\begin{equation}\label{product}
\eps(1/2,\chi)\eps(1/2,\chi^{-1})=\chi(-1).
\end{equation}
For both of these facts, see for example \cite[Section 1.1]{RS}. 

From \eqref{gauss} we see that the only time the $p$-adic valuation of $G(x,\chi)$ is non-trivial to compute is when $\chi$ is ramified and $\val_F(x)=-c(\chi)$. By a change of variables, we have \[G(xu,\chi)=\chi(u^{-1})G(x,\chi), \ \text{for} \ u\in \O_F^\times\] so it suffices to consider $G(\varpi^{-c(\chi)},\chi)$. For a detailed study of the $p$-adic valuation of $G(\varpi^{-c(\chi)},\chi)$ see \cite[Section 2]{MCMD}. For our purposes we only require the following result \cite[Corollary 2.7]{MCMD} that states $\eps(1/2,\chi)\in\overline{\Z}[\frac{1}{p}]^\times$ and for any isomorphism $\overline{\Qp}\cong\C$ we have
\begin{equation}\label{l1}
\val_p(\eps(1/2,\chi))=
\begin{cases}
-\frac{[\F:\F_p]}{2}+\frac{s(\chi^{-1})}{p-1}, \ \text{if} \ c(\chi)=1,\\
0, \ \text{if} \ c(\chi)>1.
\end{cases}
\end{equation}
All we need to know about $s(\chi^{-1})$ is that $s(\chi^{-1})=0$, if $\chi=\1$ and for $\chi\neq\1$, we have $1\leq s(\chi^{-1})\leq (p-1)[\F:\F_p]$, for more details see \cite[Chapter 6]{Was97}.

\subsection{Classification of $\pi$}\label{classofpi}

As discussed above the study of the $p$-adic valuation of $W_\pi$ relies on studying the $p$-adic valuation of $\ctl$. In turn, to study this, it will rely on the well-known classification of unitary, irreducible, admissible representations of $\GL_2(F)$, with $\cc\in\X$. The case of $c(\pi)=0$ is well understood so we focus on $c(\pi)\geq1$. Note that the assumption that $\cc\in\X$ does not affect the generality of our results as we can twist $\pi$ by an unramified character to ensure $\cc\in\X$.

\begin{enumerate}

\item $\pi$ is \textit{Supercuspidal}. In our case, since we are considering $p$ odd, we have that $\pi$ is \textit{dihedral}. This means that $\pi$ is associated to a character $\xi: E^\times\rightarrow\C^\times$ of a quadratic extension $E/F$ such that $\xi$ is non-trivial on the kernel of the norm map from $E^\times$ to $F^\times$. We have by \cite[Theorem 2.3.2]{RS}, 
\begin{equation}\label{apiSup}
c(\pi)=[\F_E:\F_F]c(\xi)+d_{E/F},
\end{equation}
where $d_{E/F}$ is the valuation of the discriminant of $E/F$. We have that 
\begin{equation}\label{epspiSup}
\eps(1/2,\pi)=\gamma\eps(1/2,\xi),
\end{equation}
for some $\gamma\in S^1$ and $\cc=\chi_{E/F}\cdot\xi|_{F^\times}$. The representation $\chi\pi$, where $\chi$ is a $\GL_1$ twist is also dihedral, as such, it is associated to a character of $E^\times$. This character is given by $\xi(\chi\circ N_{E/F}):E^\times\rightarrow\C^\times$ and we have \[\eps(1/2,\chi\pi)=\lambda\eps(1/2,\xi(\chi\circ N_{E/F}),\psi\circ\Tr_{E/F}), \ \text{for some} \ \lambda\in\{\pm1,\pm i\}.\]

\item $\pi$ is the \textit{twist of Steinberg} representation, $\pi=\mu$St, where $\mu$ is some unitary character satisfying $\mu(\varpi)=1$. In this case we have $\cc=\mu^2$ and $c(\pi)=\max\{1,2c(\mu)\}$. Furthermore, the $\eps$-factor is given by 
\begin{equation}\label{epspiSt}
\eps(1/2,\pi)=
\begin{cases}
-1, \ \text{if} \ \mu=\1,\\
\eps(1/2,\mu)^2, \ \text{if} \ \mu\neq\1.
\end{cases}
\end{equation}
We have the following two subcases. 
\begin{enumerate}
\item If $\mu\neq\1$, so $\pi=\mu$St and $c(\pi)=2c(\mu)$.

\item If $\mu=\1$, so $\pi=$ St and $c(\pi)=1$.
\end{enumerate}

\item $\pi$ is \textit{principal series}, $\pi=\mu_1\boxplus\mu_2$ for unitary characters $\mu_1$ and $\mu_2$. In particular $c(\pi)=c(\mu_1)+c(\mu_2)$ and $\omega_\pi=\mu_1\mu_2$. We can rewrite this representation in the following way. We have $\mu_1=\bone| \ |^{c_1}$ and $\mu_2=\btwo| \ |^{c_2}$, where $\bone,\btwo\in\X$ and $c_1+c_2\in i\R$. Therefore we have $\pi=\beta_1| \ |^{c_1}\boxplus\beta_2| \ |^{c_2}$. For $\pi$ being a principal series representation the $\eps$-factor is given by
\begin{equation}\label{epspiPS}
\eps(1/2,\pi)=\eps(1/2,\mu_1)\eps(1/2,\mu_2).
\end{equation}
This representation can be split into the following three subcases.

\begin{enumerate} 

\item If $\bone\neq\btwo$ and $c(\bone),c(\btwo)>0$. This can be treated as described above. 

\item If $\bone=\btwo$ and $c(\bone)=c(\btwo)>0$. This means that we can write \[\pi=\beta| \ |^{c_1}\boxplus\beta| \ |^{c_2},\] where $\beta\in\X$. Therefore, we have $c(\pi)=2c(\beta)$.

\item If $c(\bone)>c(\btwo)=0$. Thus, $\pi=\bone| \ |^{c_1}\boxplus| \ |^{c_2}$ and $c(\pi)=c(\bone)$.

\end{enumerate}
Note that since $\cc\in\X$ we have that $c_1=-c_2$.
\end{enumerate}

For ease, we will refer to these cases as $\pi$ being of Type 1, 2a, 2b, 3a, 3b or 3c (this numbering is not standard). If $\pi$ is of Type 3c, then this case has been studied in great detail in for example \cite{felicien} and \cite{Sah16} so we exclude this case. So, if $\pi$ is not of Type 3c and $c(\pi)>0$ then $c(\pi)\geq2$ and so the case of $\pi$ being of Type 2b also does not need to be considered. 

For each of these representation we can also define an $L$-factor for $\pi$. This is denoted by $L(s,\pi)$.  For the representations which we will be studying we have $L(s,\pi)=1$ see \cite[Section 2.1.6]{Sah16}. 

\section{Formulas for $c_{t,l}(\chi)$}\label{formulas}

In this section, we state the formulas for $c_{t,l}(\chi)$ for $\pi$ being of Type 1, 2a, 3a or 3b. All of the formulas stated in this section can be found in \cite[Section 2]{ASS19}. Here, we rewrite them in a way that is more convenient for our purposes.  The reason we do this is that $\ctl$ are the Fourier coefficients of $W_\pi$ and we can rewrite these in terms of objects for which we can directly compute the $p$-adic valuation of. The main one of these is the $\eps$-factors for $\chi$, which we can study the $p$-adic valuation of via \eqref{l1}.

For $\pi$ being of a specific type, the formulas for $\ctl$ depend on the specific multiplicative character $\chi$ and the values $t$ and $l$. In practice, we will have $t\in\Z$ and $0\leq l \leq c(\pi)$. These formulas contain Gauss sums and what we will do is evaluate $G(x,\chi)$ using \eqref{gauss} for all the possible cases. When evaluating the Gauss sum we only consider $l\geq1$. The case of $l=0$ is treated separately, Proposition \ref{Won}, using the generalised Atkin-Lehner relation.

\subsection{Supercuspidal representations, Type 1} \label{super} 

As stated in \cite[Section 2.1]{ASS19} We have, 
\begin{subnumcases}
{c_{t,l}(\chi)=}
\eps(1/2, \omega_\pi^{-1}\pi), \ \text{if} \ l=0, \ t=-n=-c(\pi), \ \chi=\1, \label{super1}\\ 
-\frac{1}{q-1}\eps(1/2,\omega_\pi^{-1}\pi), \ \text{if} \ l=1,t=-n, \chi=\1, \label{super2}\\ 
\frac{q^{1-l/2}}{q-1}\eps(1/2,\chi)\eps(1/2,\chi^{-1}\omega_\pi^{-1}\pi), \ \text{if} \ \chi\in\X_l, t=-c(\chi\pi), l>0, \label{super3}\\ 
0, \ \text{otherwise}. \label{super4}
\end{subnumcases}

\subsection{Twist of Steinberg, Type 2a} 
In this setting \cite[Lemma 2.1]{ASS19} states,
\begin{subnumcases}
{c_{t,l}(\chi)=}
\eps(1/2,\chi^{-1}\mu^{-1})^2G(\varpi^{-l},\chi^{-1}), \ \text{if} \ \chi\neq\mu^{-1}, t=-2c(\chi\mu), \label{twist1}\\ 
q^{-1}G(\varpi^{-l},\chi^{-1}), \ \text{if} \ \chi=\mu^{-1}, t=-2, \label{twist2}\\ 
-\zeta_F(2)^{-1}q^{-1-t}G(\varpi^{-l},\chi^{-1}), \ \text{if} \ \chi=\mu^{-1}, t>-2, \label{twist3}\\ 
0, \ \text{otherwise}. \label{twist4}
\end{subnumcases}

We now evaluate the Gauss sum in the above cases. We do this for all possible values of $l\geq1$.  If $l>1$, note from \eqref{gauss} either $\chi\neq\1$ or $G(\varpi^{-l},\chi^{-1})=0$. If $l=1$ and considering \eqref{twist1}, then we may have $\chi=\1$. Evaluating the Gauss sum in this case gives \[c_{-2,1}(\1)=\eps(1/2,\mu^{-1})^2G(\varpi^{-1},\1)=-\frac{1}{q-1}\eps(1/2,\mu^{-1})^2.\] Now suppose $l\geq1$ and $\chi\neq\1$. Starting with \eqref{twist1}, we have $\chi\neq\mu^{-1}$ and using \eqref{gauss} gives \[\ctl=\eps(1/2,\chi^{-1}\mu^{-1})^2G(\varpi^{-l},\chi^{-1})=\frac{q^{1-c(\chi)/2}}{q-1}\eps(1/2,\chi)\eps(1/2,\chi^{-1}\mu^{-1})^2.\] For \eqref{twist2}, then $\chi=\mu^{-1}$ and $t=-2$, we have \[\ctl=q^{-1}G(\varpi^{-l},\chi^{-1})=\frac{q^{-c(\chi)/2}}{q-1}\eps(1/2,\chi).\] With \eqref{twist3}, we have $\chi=\mu^{-1}$ and $t>-2$, then  
\begin{align*}
\ctl=-\zeta_F(2)^{-1}q^{-1-t}G(\varpi^{-l},\chi^{-1})&=-\frac{q^2-1}{q^2}q^{-1-t}G(\varpi^{-l},\chi^{-1})\\
&=-\frac{q^2-1}{q^2(q-1)}q^{-1-t}q^{1-c(\chi)/2}\eps(1/2,\chi)\\
&=-(q+1)q^{-2-t-c(\chi)/2}\eps(1/2,\chi).
\end{align*}
Otherwise, we have $\ctl=0$. Collecting all these cases together we have 
\begin{subnumcases}
{\ctl=}
-\frac{1}{q-1}\eps(1/2,\mu^{-1})^2, \ \text{if} \ \chi=\1, t=-2,  l=1,\label{twistC}\\
\frac{q^{1-c(\chi)/2}}{q-1}\eps(1/2,\chi)\eps(1/2,\chi^{-1}\mu^{-1})^2,\ \text{if} \ \chi\neq\mu^{-1}, t=-2c(\chi\mu), l=c(\chi), \label{twistH}\\ 
\frac{q^{-c(\mu^{-1})/2}}{q-1}\eps(1/2,\mu^{-1}), \ \text{if} \ \chi=\mu^{-1}, t=-2, \label{twistI}, l=c(\mu^{-1}),\\
-(q+1)q^{-2-t-c(\mu^{-1})/2}\eps(1/2,\mu^{-1}), \ \text{if} \ \chi=\mu^{-1}, t>-2, l=c(\mu^{-1}),\label{twistJ}\\ 
0, \ \text{otherwise}. \label{twistK}
\end{subnumcases}

\subsection{Irreducible principal series, Type 3a}\label{IPS} 

Recall this is the case of $c(\mu_i)>0$ for $i=1,2$ and $\bone\neq\btwo$. Using \eqref{twoc} and the fact $c_1=-c_2$ we have 
\begin{equation}\label{epsmu}
\eps(1/2,\mu_1)=q^{-c(\beta_1)c_1}\eps(1/2,\beta_1) \ \text{and} \ \eps(1/2,\mu_2)=q^{c(\beta_2)c_1}\eps(1/2,\beta_2).
\end{equation}
This is the first case of \cite[Lemma 2.2]{ASS19}. We have 
\begin{subnumcases}{\ctl=}
\eps(1/2,\chi^{-1}\mu_1^{-1})\eps(1/2,\chi^{-1}\mu_2^{-1})G(\varpi^{-l},\chi^{-1}), \label{ps1} \\
\quad \text{if} \ c(\chi\mu_1), c(\chi\mu_2)\neq0, t=-c(\chi\mu_1)-c(\chi\mu_2) \nonumber \\
-q^{-1/2}\mu_i(\varpi^{-1})\eps(1/2,\chi^{-1}\mu_j^{-1})G(\varpi^{-l},\chi^{-1}),\label{ps2}\\
\quad \text{if} \ c(\chi\mu_j)\neq c(\chi\mu_i)=0, \{j,i\}=\{1,2\}, t=-c(\chi\mu_j)-1 \nonumber\\
\zeta_F(1)^{-2}q^{-t/2}\mu_i(\varpi^{t+c(\chi\mu_j)})\mu_j(\varpi^{-c(\chi\mu_j)})G(\varpi^{-c(\chi\mu_j)},\chi\mu_j)G(\varpi^{-l},\chi^{-1}),\label{ps3}\\
\quad \text{if} \ c(\chi\mu_j)\neq c(\chi\mu_i)=0, \{j,i\}=\{1,2\}, t\geq-c(\chi\mu_j) \nonumber \\
0, \ \text{otherwise}. \label{ps4}
\end{subnumcases}

We now need to evaluate the Gauss sum for different values of $l$. Once again, we work through the values of $l\geq1$ for the cases \eqref{ps1}-\eqref{ps4}. Starting with $\chi=\1$. This can only occur in the setting of \eqref{ps1} with $t=-n$ and $l=1$. Therefore, by using \eqref{gauss} we have $G(\varpi^{-1},\1)=-\frac{1}{q-1}$. Now, by using \eqref{epsmu} we obtain \[c_{-n,1}(\1)=-\frac{1}{q-1}\eps(1/2,\mu_1^{-1})\eps(1/2,\mu_2^{-1})=-\frac{q^{c(\beta_1^{-1})c_1-c(\beta_2^{-1})c_1}}{q-1}\eps(1/2,\beta_1^{-1})\eps(1/2,\beta_2^{-1}).\] For the remaining cases we see if $\ctl\neq0$ we need $\chi\neq\1$ as $l=c(\chi)$. For \eqref{ps1}, that is $\chi\not\in\{\1,\beta_1^{-1},\beta_2^{-1}\}$ and $t=-c(\chi\mu_1)-c(\chi\mu_2)$. Here, we have \[\ctl=\frac{q^{1-c(\chi)/2}}{q-1}\eps(1/2,\chi)\eps(1/2,\chi^{-1}\mu_1^{-1})\eps(1/2,\chi^{-1}\mu_2^{-1}).\] On using \eqref{epsmu}, we obtain
\begin{align*}
\ctl&=\frac{q^{1-c(\chi)/2}}{q-1}\eps(1/2,\chi)\eps(1/2,\chi^{-1}\mu_1^{-1})\eps(1/2,\chi^{-1}\mu_2^{-1})\\
&=\frac{q^{1-c(\chi)/2}}{q-1}q^{c(\chi^{-1}\beta_1^{-1})c_1-c(\chi^{-1}\beta_2^{-1})c_1}\eps(1/2,\chi)\eps(1/2,\chi^{-1}\beta_1^{-1})\eps(1/2,\chi^{-1}\beta_2^{-1}).
\end{align*}
In the setting of \eqref{ps2}, we have $c(\chi\mu_i)=0$ and $t=-c(\chi\mu_j)-1$ this means that $\chi=\bi^{-1}$. On using this and \eqref{epsmu}, we have
\begin{align*}
\ctl&=-\frac{q^{(1-c(\chi))/2}}{q-1}\mu_i(\varpi^{-1})\eps(1/2,\chi)\eps(1/2,\chi^{-1}\mu_j^{-1})\\
&=-\frac{q^{(1-c(\bi^{-1}))/2}}{q-1}q^{c_i}q^{-c(\bi\bj^{-1})c_i}\eps(1/2,\bi^{-1})\eps(1/2,\bi\bj^{-1})\\
&=-\frac{q^{(1-c(\bi^{-1}))/2}}{q-1}q^{c_i-c(\bi\bj^{-1})c_i}\eps(1/2,\bi^{-1})\eps(1/2,\bi\bj^{-1}). 
\end{align*}
Our final non-trivial case, \eqref{ps3}, where again $\chi=\bi^{-1}$ and now $t>c(\chi\mu_j)$. We have 
\begin{align*}
\ctl&=\frac{(q-1)^2q^{-t/2}q^{1-c(\chi)/2}}{q^2(q-1)}\mu_i(\varpi^{t+c(\chi\mu_j)})\mu_j(\varpi^{-c(\chi\mu_j)})\eps(1/2,\chi)G(\varpi^{-c(\chi\mu_j)},\chi\mu_j)\\
&=(q-1)q^{-1-t/2-c(\chi)/2}\mu_i(\varpi^{t+c(\chi\mu_j)})\mu_j(\varpi^{-c(\chi\mu_j)})\eps(1/2,\chi)G(\varpi^{-c(\chi\mu_j)},\chi\mu_j)\\
&=\frac{(q-1)q^{-1-t/2-c(\chi)/2}q^{1-c(\chi\mu_j)/2}q^{-c_i(t+a(\beta_i^{-1}\beta_j))}}{q-1}\\
&\quad \times\mu_j(\varpi^{-c(\chi\mu_j)})\eps(1/2,\chi)\chi\mu_j(\varpi^{c(\chi\mu_j)})\eps(1/2,\chi^{-1}\mu_j^{-1})\\
&=q^{-t/2-c(\chi)/2-c(\chi\mu_j)/2}\mu_j(\varpi^{-c(\chi\mu_j)})q^{-c_i(t+c(\beta_i^{-1}\beta_j))}\eps(1/2,\chi)\chi\mu_j(\varpi^{c(\chi\mu_j)})\eps(1/2,\chi^{-1}\mu_j^{-1})\\
&=q^{-t/2-c(\bi^{-1})/2-c(\bi^{-1}\bj)/2}q^{-c_i(t+c(\beta_i^{-1}\beta_j))}\eps(1/2,\bi^{-1})\eps(1/2,\bi\mu_j^{-1}).
\end{align*}
Using \eqref{epsmu}, we have
\begin{align*}
\ctl&=q^{-t/2-c(\bi^{-1})/2-c(\bi^{-1}\bj)/2}q^{-c_i(t+c(\beta_i^{-1}\beta_j))}\eps(1/2,\bi^{-1})\eps(1/2,\bi\mu_j^{-1})\\
&=q^{-t/2-c(\bi^{-1})/2-c(\bi^{-1}\bj)/2-c(\bi\bj^{-1})c_i}q^{-c_i(t+c(\beta_i^{-1}\beta_j))}\eps(1/2,\bi^{-1})\eps(1/2,\bi\bj^{-1}). 
\end{align*}
On collecting all these cases together we obtain
\begin{subnumcases}{\ctl=}
-\frac{q^{c(\beta_1^{-1})c_1-c(\beta_2^{-1})c_1}}{q-1}\eps(1/2,\beta_1^{-1})\eps(1/2,\beta_2^{-1}), \label{psC}\\
\quad \text{if} \ \chi=\1,\beta_1^{-1},\beta_2^{-1}\neq\1, t=-c(\mu_1)-c(\mu_2), l=1, \nonumber \\
\frac{q^{1-c(\chi)/2}}{q-1}q^{c(\chi^{-1}\beta_1^{-1})c_1-c(\chi^{-1}\beta_2^{-1})c_1}\eps(1/2,\chi)\eps(1/2,\chi^{-1}\beta_1^{-1})\eps(1/2,\chi^{-1}\beta_2^{-1}), \label{psD} \\
\quad \text{if} \ \chi\notin\{\1,\beta_1^{-1},\beta_2^{-1}\}, t=-c(\chi\mu_1)-c(\chi\mu_2), l=c(\chi), \nonumber \\
-\frac{q^{(1-c(\bi^{-1}))/2}}{q-1}q^{c_i-c(\bi\bj^{-1})c_i}\eps(1/2,\bi^{-1})\eps(1/2,\bi\bj^{-1}), \label{psE}\\
\quad \text{if} \ \chi=\bi^{-1}, \{j,i\}=\{1,2\}, t=-c(\chi\mu_j)-1, l=c(\beta_i^{-1}), \nonumber\\
q^{-t/2-c(\bi^{-1})/2-c(\bi^{-1}\bj)/2-c(\bi\bj^{-1})c_i}q^{-c_i(t+c(\beta_i^{-1}\beta_j))}\eps(1/2,\bi^{-1})\eps(1/2,\bi\bj^{-1}), \label{psF}\\
\quad \text{if} \ \chi=\bi^{-1}, \{j,i\}=\{1,2\}, t\geq-c(\chi\mu_j), l=c(\beta_i^{-1}), \nonumber\\
0, \ \text{otherwise}. \label{psG}
\end{subnumcases}

\subsection{Irreducible principal series, Type 3b} 

We now consider the case of $\bone=\btwo$. This means that we can write \[\pi=\beta| \ |^{c_1}\boxplus\beta| \ |^{-c_1},\] where $\beta\in\X$. In this setting we have from the second part of \cite[Lemma 2.2]{ASS19},

\begin{subnumcases}
{\ctl=}
\eps(1/2,\chi^{-1}\mu_1^{-1})\eps(1/2,\chi^{-1}\mu_2^{-1})G(\varpi^{-l},\chi^{-1}), \label{ps5}\\
\quad \text{if} \ c(\chi\mu_1),c(\chi\mu_2)\neq0, t=-c(\chi\mu_1)-c(\chi\mu_2), \nonumber \\
q^{-1}G(\varpi^{-l},\chi^{-1}), \ \text{if} \ c(\chi\mu_1)=c(\chi\mu_2)=0, t=-2, \label{ps6}\\
-q^{-1/2}\zeta(1)^{-1}G(\varpi^{-l},\chi^{-1})(\mu_1(\varpi)+\mu_2(\varpi)), \label{ps7}\\
\quad \text{if} \ c(\chi\mu_1)=c(\chi\mu_2)=0, t=-1, \nonumber\\
q^{-t/2}G(\varpi^{-l},\chi^{-1})\Big(-q^{-1}\zeta_F(1)^{-1}(\mu_1(\varpi^{t+2})+\mu_2(\varpi^{t+2}))+\label{ps8} \\ 
\quad\quad\quad\quad\quad\quad\quad\quad \zeta_F(1)^{-2}\sum^t_{k=0}\mu_1(\varpi^k)\mu_2(\varpi^{t-k})\Big), \nonumber\\
\quad \text{if} \ c(\chi\mu_1)=c(\chi\mu_2)=0, t\geq0,  \nonumber\\
0, \ \text{otherwise}. \label{ps9}
\end{subnumcases}

We now evaluate the Gauss sum in \eqref{ps5}-\eqref{ps9}. We do this by following the same method as above, that is, considering all possible values for $l\geq1$. Note that \eqref{ps5} has the same formula as that of \eqref{ps1}. This means we can evaluate \eqref{ps5} in the same way as \eqref{ps1}, expect the terms involving $c_1$ will cancel out. As such, we will focus on the cases \eqref{ps6}-\eqref{ps9} and again note $\bone=\btwo$ and so the terms involving $c_1$ will cancel out. In these cases since $c(\chi\mu_1)=c(\chi\mu_2)=0$ we know that $\chi=\beta^{-1}$. More specifically, we do not have non-trivial Gauss sum for $l=0$. Starting with the case of \eqref{ps6}, that is, $\chi=\beta^{-1}$ and $t=-2$ with $l\geq1$, we have \[G(\varpi^{-l},\chi^{-1})=\frac{q^{1-c(\chi)/2}}{q-1}\eps(1/2,\chi)\chi(\varpi^{-l}).\] Therefore, \[q^{-1}G(\varpi^{-l},\beta)=\frac{q^{-1}q^{1-c(\beta^{-1})/2}}{q-1}\eps(1/2,\beta^{-1})=\frac{q^{-c(\beta^{-1})/2}}{q-1}\eps(1/2,\beta^{-1}).\] 
Now looking at the case of \eqref{ps7}, with $\chi=\beta^{-1}$ and $t=-1$, we have 
\begin{align*}
-q^{-1/2}\zeta_F(1)^{-1}G(\varpi^{-l},\chi^{-1})(\mu_1(\varpi)+\mu_2(\varpi))&=-q^{-3/2}(q-1)G(\varpi^{-l},\chi^{-1})(\mu_1(\varpi)+\mu_2(\varpi))\\
&=-\frac{(q-1)q^{-3/2}q^{1-c(\beta^{-1})/2}}{q-1}\eps(1/2,\beta^{-1})\\
&\quad \times(\mu_1(\varpi)+\mu_2(\varpi))\\
&=-q^{-1/2-c(\beta^{-1})/2}(\beta(\varpi)|\varpi|^{c_1}+\beta(\varpi)|\varpi|^{-c_1})\\
& \quad \times\eps(1/2,\beta^{-1})\\
&=-q^{-(1+c(\beta^{-1}))/2}\eps(1/2,\beta^{-1})(q^{-c_1}+q^{c_1}). 
\end{align*}
In the final non-trivial case, \eqref{ps8}, again we have $\chi=\beta^{-1}$ but now $t>0$. We obtain
\begin{align*}
\ctl&=q^{-t/2}G(\varpi^{-l},\chi^{-1})\Big(-q^{-1}\zeta_F(1)^{-1}\left(\mu_1(\varpi^{t+2})+\mu_2(\varpi^{t+2})\right)\\
& \quad\quad\quad\quad\quad\quad\quad\quad\quad +\zeta_F(1)^{-2}\sum^t_{k=0}\mu_1(\varpi^k)\mu_2(\varpi^{t-k})\Big)\\
&=-\frac{(q-1)q^{-t/2}}{q^2}G(\varpi^{-l},\chi^{-1})\Big(\mu_1(\varpi^{t+2})+\mu_2(\varpi^{t+2})-(q-1)\sum^t_{k=0}\mu_1(\varpi^k)\mu_2(\varpi^{t-k})\Big)\\
&=-\frac{(q-1)q^{-t/2}}{q^2}\frac{q^{1-c(\beta^{-1})/2}}{q-1}\eps(1/2,\beta^{-1})\Big(\mu_1(\varpi^{t+2})+\mu_2(\varpi^{t+2})\\
&\quad\quad\quad\quad\quad\quad\quad\quad\quad\quad\quad\quad\quad\quad\quad\quad\quad-(q-1)\sum^t_{k=0}\mu_1(\varpi^k)\mu_2(\varpi^{t-k})\Big)\\
&=-q^{-1-t/2-c(\beta^{-1})/2}\eps(1/2,\beta^{-1})\Big(|\varpi^{t+2}|^{c_1}+|\varpi^{t+2}|^{-c_1}-(q-1)\sum^t_{k=0}|\varpi^k|^{c_1}|\varpi^{t-k}|^{-c_1}\Big)\\
&=-q^{-(2+t+c(\beta^{-1}))/2}\eps(1/2,\beta^{-1})\Big(q^{-c_1(t+2)}+q^{c_1(t+2)}-(q-1)\sum^t_{k=0}q^{-kc_1+c_1(t-k)}\Big). 
\end{align*}
Collecting all these cases together gives 
\begin{subnumcases}
{\ctl=}
-\frac{1}{q-1}\eps(1/2,\beta^{-1})^2, \label{psH}\\
\quad \text{if} \ \chi=\1, l=1, c(\beta)\neq0, t=-c(\mu_1)-c(\mu_2)=-n, \nonumber \\
\frac{q^{1-c(\chi)/2}}{q-1}\eps(1/2,\chi)\eps(1/2,\chi^{-1}\beta^{-1})^2, \label{psI}\\
\quad \text{if} \ \chi\neq\1,\beta^{-1}, t=-2c(\chi\beta), l=c(\chi), \nonumber \\
\frac{q^{-c(\beta^{-1})/2}}{q-1}\eps(1/2,\beta^{-1}), \ \text{if} \ \chi=\beta^{-1}, t=-2, l=c(\beta^{-1}), \label{psJ}\\
-q^{-(1+c(\beta^{-1}))/2}\eps(1/2,\beta^{-1})(q^{-c_1}+\qcone), \ \text{if} \ \chi=\beta^{-1}, t=-1, l=c(\beta^{-1}), \label{psK}\\
-q^{-(2+t+c(\beta^{-1}))/2}\eps(1/2,\beta^{-1})\Big(q^{-c_1(t+2)}+q^{c_1(t+2)} \label{psL} \\ 
\quad\quad\quad\quad\quad\quad\quad -(q-1)\sum^t_{k=0}q^{c_1(t-2k)}\Big), \ \text{if} \ \chi=\beta^{-1}, t\geq0, l=c(\beta^{-1}),  \nonumber\\
0, \ \text{otherwise}. \label{psM}
\end{subnumcases}

\section{$p$-adic valuation}\label{padic}

In this section, we study the $p$-adic valuation of the local Whittaker newforms associated to an irreducible, admissible, infinite-dimensional representation $\pi$ of $\GL_2(F)$. To do this we utilise the Fourier expansion of $W_\pi$. In Section \ref{formulas} we have, in essence, explicitly written down the Fourier coefficients of $W_\pi$. We will see that proving the bounds in Theorem \ref{pi2} and \ref{pi2plus} boils down to finding lower bounds for $\valp(\ctl)$. We start by reviewing the local Whittaker newform and its Fourier expansion, (extending what is done in Chapter \ref{Chap1}).

\subsection{Fourier expansion of $W_\pi$}

In Section \ref{whittaker}, we showed that the local Whittaker model for $\GL_2(\Qp)$ exists and is unique. There is nothing special about $\Qp$ and in fact, the proofs are almost identical with $\Qp$ replaced by any non-archimedean local field, $F$. Here, since we are assuming $\cc\in\X$ then we need to make sure we have the correct group for which $W_\pi$ is invariant. In this case the normalised Whittaker newform $W_\pi$ is the unique vector in $\W(\pi,\psi)$ invariant under $K_1(n)$, where \[K_1(n)=\begin{pmatrix} 1+\m^n & \O_F \\ \m^n & \O_F \end{pmatrix} \cap \GL_2(\O_F),\] that satisfies $W_\pi(1)=1$.

Let $F$ be a non-archimedean local field. Then we have the following decomposition of $\GL_2(F)$ given by \[\GL_2(F)=\bigsqcup_{t\in\Z}\bigsqcup_{0\leq l \leq n}\bigsqcup_{v\in\O_F^\times/(1+\m^{\min\{l,n-l\}})}Z(F)U(F)\gtlv K_1(n),\] where $Z(F)$ is the centre of $\GL_2(F)$, $U(F)$ is the `upper right' unipotent subgroup of $GL_2(F)$ and \[\gtlv=\begin{pmatrix} \varpi^t & \\ & 1 \end{pmatrix} \begin{pmatrix} & 1 \\ -1 & \end{pmatrix} \begin{pmatrix} 1 & v\varpi^{-l}\\ & 1 \end{pmatrix}=\begin{pmatrix} & \varpi^t \\ -1 & -v\varpi^{-l}\end{pmatrix}\in\GL_2(F), \ \text{for} \ t,l\in\Z,v\in\O_F^\times.\] How to obtain this decomposition is described in \cite[Footnote 8]{MCMD} and \cite[Lemma 2.13]{Sah16}. It means that for $n$ fixed for each $g\in\GL_2(F)$ there exists a unique $l\in\Z$ satisfying $0\leq l \leq n$, such that \[g\in Z(F)N(F)\gtlv K_1(n),\] for $t\in\Z$ and $v\in\O_F^\times$. Therefore, we can restrict our attention to $\gtlv\in|GL_2(F)$ and reduce our study to the values $W_\pi(\gtlv)$. For fixed $t\in\Z$ and $l\geq0$ the function sending $v$ to $W_\pi(\gtlv)$ descends to the quotient $\O_F^\times/(1+\m^l)$. Therefore, using Fourier inversion there exists constants $\ctl\in\C$ for $\chi\in\X_{\leq l}$, such that 
\begin{equation}\label{FEW}
W_{\pi}(g_{t,l,v})=\sum_{\chi\in\X_{\leq l}}\ctl\chi(v), \ \text{for every} \ v\in\O_F^\times.
\end{equation}
Therefore, taking $p$-adic valuation of \eqref{FEW} gives
\begin{equation}\label{valW}
\valp(W_\pi(\gtlv))\geq\min_{\chi\in\X_{\leq l}}\{\valp(\ctl)\}.
\end{equation}
Before we prove Theorems \ref{pi2} and \ref{pi2plus}, in which we compute $\valp(W_\pi(\gtlv))$, we can reduce the range in which $l$ needs to vary. Precisely, for $c(\psi)=0$ and $0\leq l \leq c(\pi)$ \cite[Proposition 2.28]{Sah16} states, 
\begin{equation}\label{genAL}
W_{\wt{\pi}}(g_{t,l,v})=\eps(1/2,\pi)\omega_\pi(v)\psi(\varpi^{-t+l}v^{-1})W_\pi(g_{t+2l-n,n-l,-v}),
\end{equation}
where $\wt{\pi}=\omega_\pi^{-1}\pi$. This generalised Atkin-Lehner relation means that if we can prove the desired bounds for a shorter range in which $l$ is varying, we can obtain the results for the entire range of $0\leq l\leq c(\pi)$. Precisely, for $\pi$ being of Type 1, 2a or 3b we can prove bounds for $0\leq l \leq c(\pi)/2$ to obtain results in the full range. In the case of $\pi$ being of Type 3a we will have $l$ varying in the range $0\leq l \leq \max\{c(\bone),c(\btwo)\}$. The proposition below deals with the endpoints $l=0$ and $l=c(\pi)$. In the general setting, we will have $l$ varying in the ranges 
\begin{equation}\label{range}
1\leq l \leq c(\pi)/2, \ \text{if $\pi$ is Type 1, 2a, 3b and} \ 1 \leq l \leq \max\{c(\bone),c(\btwo)\}, \ \text{if $\pi$ is Type 3a}.
\end{equation}
If we consider the $p$-adic valuation of \eqref{genAL} we have 
\begin{equation} \label{valWon}
\valp(W_{\wtpi}(g_{t,l,v}))=\valp(\eps(1/2,\pi))+\valp(W_\pi(g_{t+2l-n,n-l,-v})).
\end{equation}
The relations \eqref{genAL} and \eqref{valWon} enable us to obtain $p$-adic valuations for $W_\pi(\gtlv)$ with $l\in\{0,c(\pi)\}$ without using the Fourier expansion and \eqref{valW}. After the following proposition, we explain how we can restrict our attention to the desired ranges in \eqref{range}.

\begin{Prop} \label{Won}
For a finite extension $F/\Qp$, with $p$ odd,  let $\pi$ be of Type 1, 2a, 3a or 3b. Let $\psi:F\rightarrow\C^\times$ be an additive character with $c(\psi)=0$, an isomorphism $\C\cong\overline{\Qp}$, a $t\in\Z$, an $l\in\{0,c(\pi)\}$, and a $v\in\O_F^\times$. We have 
\begin{enumerate}
\item[(i)] if $\pi$ is of Type 1, 2a, 3a or 3b and $c(\pi)=2$, then \[\valp(W_\pi(\gtlv))\geq-[\F:\F_p],\]
\item[(ii)] if $\pi$ is of Type 1, 2a or 3b with $c(\pi)>2$, then \[\valp(W_\pi(\gtlv))\geq0,\]
\item[(iii)] if $\pi$ is of Type 3a with $c(\pi)>2$ and $c(\bone),c(\btwo)>1$, then \[\valp(W_\pi(\gtlv))\geq -|c(\bone^{-1})-c(\btwo^{-1})|\cdot|\valp(\qcone)|,\]
\item[(iv)] if is of Type 3a with $c(\pi)>2$ and one of $c(\bone)$ or $c(\btwo)$ is equal to one, then \[\valp\left(W_\pi(\gtlv)\right)\geq-|c(\bj^{-1})-1|\cdot|\valp(\qci)|-\frac{[\F:\F_p]}{2}.\]
\end{enumerate}
\end{Prop}

\begin{proof}
Note that $g_{t,c(\pi),v}$ and \[\begin{pmatrix} \varpi^{t+2c(\pi)} & \\ & 1 \end{pmatrix} \begin{pmatrix} v^{-1} & \\ & v^{-1} \end{pmatrix},\] have the same invariants. Hence \[W_{\wtpi}(g_{t,c(\pi),v})=W_{\wtpi}\left(\begin{pmatrix} \varpi^{t+2c(\pi)} & \\ & 1 \end{pmatrix} \begin{pmatrix} v^{-1} & \\ & v^{-1} \end{pmatrix}\right).\] Recall that $W_{\wtpi}(g)=\cc^{-1}(\det(g))W_\pi^*(g)$. All of our representations have $L(s,\pi)=1$ and from \cite[Lemma 2.5]{Sah16} we have \[W_\pi^*(a(\varpi^{t+2c(\pi)}v))=
\begin{cases}
1, \ \text{if} \ t=0, L(s,\pi)=1,\\
0, \ \text{if} \ t\neq0, L(s,\pi)=1.
\end{cases}
\]
Therefore, 
\[\valp\left(W_\pi^*\left(a(\varpi^{t+2c(\pi)}v)\right)\right)=
\begin{cases}
0, \ \text{if} \ t=0, L(s,\pi)=1,\\
\infty, \ \text{if} \ t\neq0, L(s,\pi)=1.
\end{cases}
\]
This means in our setting we have $\valp\left(W_\pi^\ast(a(\varpi^{t+2c(\pi)}v))\right)\geq0$ and so \[\valp(W_{\wtpi}(g_{t,c(\pi),v}))\geq0.\] From \eqref{valWon}, we have \[\valp(\eps(1/2,\pi))+\valp(W_\pi(g_{t+c(\pi),0,-v}))\geq0.\] This means to establish bounds for \[\valp(W_\pi(g_{t+c(\pi),0,-v})),\] we need to compute bounds for \[-\valp(\eps(1/2,\pi)).\] If $c(\pi)=2$ and $\pi$ is of Type 1. Then from \eqref{apiSup}, we have $c(\xi)=1$. Recall from \eqref{epspiSup}, we have \[\eps(1/2,\pi)=\gamma\eps(1/2,\xi). \] Therefore, we have \[-\valp(\eps(1/2,\pi))=-\valp(\eps(1/2,\xi))=\frac{[\F_E:\F_p]}{2}-\frac{s(\xi^{-1})}{p-1}.\] Recall that, $1\leq s(\xi)\leq (p-1)[\F_E:\F_p]$. Thus, 
\begin{align*}
-\valp(\eps(1/2,\pi))&=-\valp(\eps(1/2,\xi))\\
&\geq\frac{[\F_E:\F_p]}{2}-\frac{(p-1)[\F_E:\F_p]}{p-1}\\
&=-[\F:\F_p]-2[\F:\F_p]\\
&=-[\F:\F_p].
\end{align*}
If $\pi$ is Type 2a with $c(\pi)=2$, then $c(\mu)=1$. So, recalling \eqref{l1}and \eqref{epspiSt}, we have 
\begin{align*}
-\valp(\eps(1/2,\pi))&=-2\valp(\eps(1/2,\mu))\\
&\geq\frac{2[\F:\F_p]}{2}-\frac{2(p-1)[\F:\F_p]}{p-1}\\
&=-[\F:\F_p].
\end{align*}
In the case that $\pi$ is Type 3a with $c(\pi)=2$ we have $c(\bone)=1=c(\btwo)$.  From \eqref{epspiPS}, we have 
\begin{align*}
-\valp(\eps(1/2,\pi))&=-\valp\left(q^{c(\bone^{-1})c_1-c(\btwo)c_1)}\right)+\left(\valp(\eps(1/2,\bone^{-1})+\valp(\eps(1/2,\btwo^{-1}\right)\\
&\geq\frac{[\F:\F_p]}{2}-\frac{(p-1)[\F:\F_p]}{p-1}+\frac{[\F:\F_p]}{2}-\frac{(p-1)[\F:\F_p]}{p-1}\\
&=-[\F:\F_p].
\end{align*}
Finally if $\pi$ is of Type 3b with $c(\pi)=2$ then $c(\beta)=1$. Therefore, we have 
\begin{align*}
-\valp(\eps(1/2,\pi))&=-2\valp(\eps(1/2,\beta)\\
&\geq\frac{2[\F:\F_p]}{2}-\frac{2(p-1)[\F:\F_p]}{p-1}.
\end{align*}
Therefore, \[\valp(W_\pi(g_{t+2,0,-v}))\geq-[\F:\F_p],\] giving (i). For the rest of the proof, we assume that $c(\pi)>2$. Returning to the setting when $\pi$ is of Type 1 and using \eqref{apiSup} we have $c(\xi)>1$. Hence, \[-\valp(\eps(1/2,\pi))=-\valp(\eps(1/2,\xi))=0.\] If $\pi$ is of Type 2, then we have $c(\mu)>1$ and similarly we have, \[-\valp(\eps(1/2,\pi))=-2\valp(\eps(1/2,\mu))=0.\] For $\pi$ being of Type 3b we have $c(\beta)>1$ so again \[-\valp(\eps(1/2,\pi))=-2\valp(\eps(1/2,\beta))=0.\] Thus, \[\valp(W_\pi(g_{t+c(\pi),0,-v}))\geq0,\] giving (ii). Now if $\pi$ is of type 3a In the setting where both $c(\bone)$ and $c(\btwo)$ are greater than one we have 
\begin{align*}
-\valp(\eps(1/2,\pi))&=-\valp\left(q^{c(\bone^{-1})c_1-c(\btwo^{-1})c_1}\right)+\valp(\eps(1/2,\bone^{-1}))+\valp(\eps(1/2,\btwo^{-1}))\\
&=-\valp\left(q^{c(\bone^{-1})c_1-c(\btwo^{-1})c_1}\right)\\
&=-(c(\bone^{-1})-c(\btwo^{-1}))\valp(\qcone)\\
&\geq-|c(\bone^{-1})-c(\btwo^{-1})|\cdot|\valp(\qcone)|.
\end{align*}
Therefore, \[\valp\left(W_\pi(g_{t+c(\pi),0,-v})\right)\geq-|c(\bone^{-1})-c(\btwo^{-1})|\cdot|\valp(\qcone)|,\] giving (iii). Finally, consider the case when one of $c(\bone)$ or $c(\btwo)$ is equal to one. In this setting, we can bound \[-\valp(\eps(1/2,\pi))\] in the following way, 
\begin{align*}
-\valp(\eps(1/2,\pi))&=-\valp\left(q^{c(\bj^{-1})c_1-c_1}\right)+\valp(\eps(1/2,\bi^{-1}))+\valp(\eps(1/2,\bj^{-1}))\\
&\geq-\valp\left(q^{c(\bj^{-1})c_i-c_i}\right)+\frac{[\F:\F_p]}{2}-\frac{(p-1)[\F:\F_p]}{p-1}\\
&\geq-|c(\bj^{-1})-1|\cdot|\valp(\qci)|-\frac{[\F:\F_p]}{2}.
\end{align*}
Thus, \[\valp\left(W_\pi(g_{t+c(\pi),0,-v})\right)\geq-|c(\bj^{-1})-1|\cdot|\valp(\qci)|-\frac{[\F:\F_p]}{2},\] giving (iv).
\end{proof}

This means for $c(\pi)=2$ we can use Proposition \ref{Won} to obtain results for $l=0$ and $l=c(\pi)=2$. Therefore, we just need to compute bounds for $l=1$ to understand $\valp(W_\pi(\gtlv))$. For $c(\pi)>2$ with $\pi$ being of Type 1, 2a or 3b and using \eqref{valWon} we have \[\valp(W_{\wtpi}(g_{t,l,v}))=\valp(W_\pi(g_{t+2l-n,n-l,-v})),\] so if $l\leq c(\pi)/2$ and we can compute the left hand side then for free we obtain bounds for the right hand side. When, $\pi$ is of Type 3a with $c(\pi)>2$, we have $\valp(\eps(1/2,\pi))$ can be non-zero. So, we need to include this in our final bounds. As, if we compute bounds for the left-hand side of \eqref{valWon}, say $\mathcal{B}$, then \[\valp(W_\pi(g_{t+2l-n,n-l,-v}))\geq\mathcal{B}-\valp(\eps(1/2,\pi)).\] Therefore, depending on the sign of $\valp(\eps(1/2,\pi))$,  $\valp(W_\pi(g_{t+2l-n,n-l,-v}))$ could have a smaller lower bound than the one computed for $\valp(W_{\wtpi}(g_{t,l,v}))$. This means that for $\pi$ being of Type 3a with $c(\pi)>2$ our final bounds will also include $-\valp(\eps(1/2,\pi))$.

\subsection{Statement of main results} \label{results}

Here we state the main result which gives lower bounds for $\valp(W_\pi(\gtlv))$ for each representation $\pi$, whereby the results are indexed by $\pi$ being of a certain type. We state by stating the result for $c(\pi)=2$ and then the result for $c(\pi)>2$. The reason for this is that the bounds obtained change depending on whether $c(\pi)=2$ or not. 

For $c(\pi)=2$, we have the following result.

\begin{Thm}\label{api2}
For a finite extension $F/\Qp$, with $p$ odd,  let $\pi$ be of Type 1, 2a, 3a or 3b with $c(\pi)=2$ and $\cc\in\X$. An additive character $\psi:F\rightarrow\C^\times$ with $c(\psi)=0$, an isomorphism $\C\cong\overline{\Qp}$, a $t\in\Z$, an $0\leq l\leq 2$, and a $v\in\O_F^\times$. We have 
\begin{enumerate}
\item[(i)] if $\pi$ is Type 1, then \[\valp(W_\pi(\gtlv))\geq-[\F:\F_p],\]
\item[(ii)] if $\pi$ is Type 2a, then  \[\val_p(W_{\pi,\psi}(g_{t,l,v}))\geq -(t+3)[\F:\F_p]+\frac{1}{p-1},\]
\item[(iii)] if $\pi$ is Type 3a, then \[\valp(W_\pi(\gtlv)\geq-\frac{t+4}{2}[\F:\F_p]+\frac{2}{p-1}-(t+2)|\valp(\qci)|,\]
\item[(iv)] if $\pi$ is Type 3b, then \[\valp\left(W_\pi(\gtlv)\right)\geq-\frac{t+4}{2}[\F:\F_p]+\frac{1}{p-1}-(t+2)|\valp(\qcone)|.\]
\end{enumerate}
\end{Thm}

We now state the general result for $c(\pi)>2$.

\begin{Thm}\label{apiplus}
For a finite extension $F/\Qp$, with $p$ odd,  let $\pi$ be of Type 1, 2a, 3a or 3b with $c(\pi)>2$ and $\cc\in\X$. An additive character $\psi:F\rightarrow\C^\times$ with $c(\psi)=0$, an isomorphism $\C\cong\overline{\Qp}$, a $t\in\Z$, an $0\leq l\leq c(\pi)$, and a $v\in\O_F^\times$. We have 
\begin{enumerate}
\item[(i)] if $\pi$ is Type 1, then \[\valp(W_\pi(\gtlv))\geq(1-l/2)[\F:\F_p],\]
\item[(ii)] if $\pi$ is Type 2a, then \[\valp\left(W_\pi(\gtlv)\right)\geq\min\left\{-\frac{c(\pi)}{2}[\F:\F_p]+\frac{2}{p-1},-\left(2+t+\frac{c(\mu^{-1})}{2}\right)[\F:\F_p]\right\},\]
\item[(iii)] if $\pi$ is Type 3a, then 
\begin{enumerate}
\item if both $c(\bone)$ and $c(\btwo)$ are greater than one then
\begin{align*}
\valp(W_\pi(\gtlv))&\geq\min\Bigg\{-\frac{c(\pi)}{2}[\F:\F_p]+\min\left\{\frac{1}{p-1},[\F:\F_p]-c(\pi)|\valp(\qcone)|\right\},\\
&-\frac{t+c(\bone^{-1})+1+c(\bi^{-1}\bj)}{2}[\F:\F_p]+\frac{1}{p-1}\\
&-(t+c(\bi^{-1}\bj)+1)|\valp(\qci)|\Bigg\}\\
&+
\begin{cases}
-|c(\bone^{-1})-c(\btwo^{-1})|\cdot|\valp(\qcone)|,\ \text{if} \ \valp(\eps(1/2,\pi))>0,\\
0, \ \text{if} \ \valp(\eps(1/2,\pi))\leq0,
\end{cases}
\end{align*}
\item if one of $c(\bone)$ or $c(\btwo)$ is equal to one then 
\begin{align*}
\valp\left(W_\pi(\gtlv)\right)&\geq\min\Bigg\{-\frac{c(\pi)}{2}[\F:\F_p]+\frac{2}{p-1}, \\
&-\frac{[\F:\F_p]}{2}+\frac{1}{p-1}-(c(\pi)-2)|\valp(\qcone)|,\\
&-\frac{t+3+c(\bi^{-1}\bj)}{2}[\F:\F_p]+\frac{2}{p-1}\\
&-(t+c(\bi^{-1}\bj)+1)|\valp(\qci)|\Bigg\}\\
&+
\begin{cases}
-(c(\pi)-2)|\valp(\qcone)|-\frac{[\F:\F_p]}{2}, \ \text{if} \ \valp(\eps(1/2,\pi))>0,\\
0, \ \text{if} \ \valp(\eps(1/2,\pi))\leq0.
\end{cases}
\end{align*}
\end{enumerate}
\item[(iv)] if $\pi$ is Type 3b, then  
\begin{align*}
\valp\left(W_\pi(\gtlv)\right)\geq\min\Bigg\{&-\frac{c(\pi)[\F:\F_p]}{2}+\frac{2}{p-1},\\
&-\frac{(t+c(\beta^{-1})+2)[\F:\F_p]}{2}-(t+2)|\valp(\qcone)|\Bigg\}.
\end{align*}
\end{enumerate}
\end{Thm}

\subsection{Overview of proofs}

As the proofs of these theorems both follow a similar strategy we provide a brief overview of them. To obtain the bounds provided in Section \ref{results} we implement the method described below. This is very similar to the one used by the authors in \cite[Section 3]{MCMD} who obtain analogous results for trivial central character. Recall from \eqref{FEW} and \eqref{valW} we have that \[\valp(W_\pi(\gtlv))\geq\min_{\chi\in\X_{\leq l}}\{\valp(\ctl)\}.\] Due to \eqref{genAL} and Proposition \ref{Won}, we can restrict the range of which $l$ varies to those in \eqref{range}, in the case that $c(\pi)\geq2$. This means our first step is to consider all the possible different cases for $\ctl$, for $1\leq l \leq c(\pi)/2$ or $1\leq l \leq \max\{c(\bone),c(\btwo)\}$ for $\pi$ being of Type 3a, computed in Section \ref{formulas} and calculate their $p$-adic valuation. We rely heavily on \eqref{l1} which allows us to compute and obtain lower bounds for the $\eps$-factors which appear in the formulas for $\ctl$. As such, we can calculate lower bounds for $\valp(\ctl)$ in every setting described in Section \ref{formulas}. Once, we have computed these lower bounds for each case we then find the minimum of these as $\chi$ varies over $\X_{\leq l}$. 

\subsection{Proof of Theorems \ref{api2} \& \ref{apiplus}}

Though we stated the theorems for $c(\pi)=2$ and $c(\pi)>2$ separately, we provide the proof of these statements simultaneously. We will always use the bound that for a character being non-trivial, we have $s(\cdot)\geq1$. 

\subsubsection{$\pi$ is of Type 1} \label{type1val} 

Here, we compute the $p$-adic valuation of \eqref{super1}-\eqref{super4}. Note that, since $q$ is odd we have that $\wt{\pi}$ is also a dihedral representation. Therefore, we have that $\wt{\pi}$ is associated to a character $\xi'$ of $E^\times$, given by \[\xi'=\xi((\xi^{-1}|_{F^\times}\chi_{E/F})\circ N_{E/F})=\xi(\xi^{-1}|_{F^\times}\circ N_{E/F}).\] Just like \eqref{apiSup}, we have, $c(\wt{\pi})=[\F_E:\F]c(\xi')+d_{E/F}$. Recall that, we also have $c(\pi)=c(\wt{\pi})$. From \cite[Lemma 2.7]{ACAS} we know, \[c(\chi\pi)\leq\max\{c(\pi),c(\cc)+c(\chi),2c(\chi)\}.\] We start with the case of $c(\pi)=2$. If we first focus on the case of $\chi=\1$, then we are in the setting of \eqref{super2}, which means $l=1$ and $t=-2$. Therefore, we have 
\begin{align*}
\valp(c_{-2,1}(\1))&=\valp\Big(-\frac{1}{q-1}\eps(1/2,\cc^{-1}\pi)\Big)\\
&=\valp(\eps(1/2,\cc^{-1}\pi))\\
&=\valp(\eps(1/2,\xi',\psi\circ\Tr_{E/F})).
\end{align*}
Since $c(\pi)=2=c(\wtpi)$, we have that $c(\xi')=1$ and so using \eqref{l1}, gives \[\valp(c_{-2,1}(\1))\geq-\frac{2[\F:\F_p]}{2}+\frac{1}{p-1}=-[\F:\F_p]+\frac{1}{p-1}.\] We now look at when $\chi\neq\1$. This is the setting of \eqref{super3}. In this case, we have that $c(\chi)=1$. Thus, using \eqref{l1} we obtain, 
\begin{align*}
\valp(c_{-c(\pi),1}(\chi))&=\valp\Big(\frac{q^{1/2}}{q-1}\eps(1/2,\chi)\eps(1/2,\chi^{-1}\cc^{-1}\pi)\Big)\\
&\geq\frac{[\F:\F_p]}{2}-\frac{[\F:\F_p]}{2}+\frac{1}{p-1}+\valp(\eps(1/2,\chi^{-1}\cc^{-1}\pi))\\
&=\frac{1}{p-1}+\valp(\eps(1/2,\chi^{-1}\cc^{-1}\pi))\\
&=\frac{1}{p-1}+\valp(\eps(1/2,\chi^{-1}\wtpi)\\
&=\frac{1}{p-1}+\valp(\eps(1/2,\xi'(\chi^{-1}\circ N_{E/F}),\psi\circ\Tr_{E/F})).
\end{align*}
We have that $c(\xi'(\chi^{-1}\circ N_{E/F}))=1$ and so using \eqref{l1}, we have \[\valp(\eps(1/2,\xi'(\chi^{-1}\circ N_{E/F}),\psi\circ\Tr_{E/F}))\geq-\frac{2[\F:\F_p]}{2}+\frac{1}{p-1}=-[\F:\F_p]+\frac{1}{p-1}.\] Hence, \[\valp(c_{-c(\pi),1}(\chi))\geq-[\F:\F_p]+\frac{2}{p-1}.\] This is the only type for which Proposition \ref{Won} gives the best bounds. Therefore, from Proposition \ref{Won} (i) we have \[\valp(W_\pi(\gtlv))\geq-[\F:\F_p].\] We now turn our attention to the case when $c(\pi)>2$. Here, if $\chi=\1$, then we are in the setting of \eqref{super2} and so $t=-c(\pi)$ and $l=1$. Since, $c(\wtpi)>2$ we have $c(\xi')>1$ and so \[\valp(c_{-c(\pi),1}(\1))=0.\] If $\chi\neq\1$. Then, we know that $c(\chi)=l$, so if $l=1$ we have 
\begin{align*}
\valp(c_{-c(\pi),1}(\chi))&\geq\frac{1}{p-1}+\valp(\eps(1/2,\chi^{-1}\cc^{-1}\pi))\\
&=\frac{1}{p-1}+\valp(\eps(1/2,\xi'(\chi^{-1}\circ N_{E/F}),\psi\circ\Tr_{E/F})).
\end{align*}
Using \eqref{apiSup}, we have that $c(\xi')=c(\xi'(\chi^{-1}\circ N_{E/F}))$. Therefore, we have \[\valp(c_{-c(\pi),1}(\chi))\geq\frac{1}{p-1}.\] If $l>1$, then we have 
\begin{align*}
\valp(c_{-c(\pi),l}(\chi))&=\valp\Big(\frac{q^{1-l/2}}{q-1}\eps(1/2,\chi)\eps(1/2,\chi^{-1}\omega_\pi^{-1}\pi)\Big)\\
&=(1-l/2)[\F:\F_p]+\valp(\eps(1/2,\chi^{-1}\cc^{-1}\pi)).
\end{align*}
Again using \eqref{apiSup}, we have that \[\valp(\eps(1/2,\chi^{-1}\cc^{-1}\pi))=0.\] Thus, \[\valp(c_{-c(\pi),l}(\chi))=(1-l/2)[\F:\F_p].\] 

\subsubsection{$\pi$ is of Type 2a} \label{valT2}   

We now compute the $p$-adic valuation of $W_\pi(\gtlv)$ in the case that $\pi$ is of Type 2a. This means we need to compute $p$-adic valuations for \eqref{twistC}-\eqref{twistK}. We deal with the case of $c(\pi)=2$ first. Recall, that this means that $c(\mu)=1$. Starting with \eqref{twistC}, $\chi=\1$ and $t=-2$, we obtain \[\val_p\left(-\frac{1}{q-1}\eps(1/2,\mu^{-1})^2\right)=\val_p\left(\frac{1}{q-1}\right)+2\val_p\eps(1/2,\mu^{-1})).\] We can bound the second of these terms by using \eqref{l1}. Therefore, we have \[\val_p\left(-\frac{1}{q-1}\eps(1/2,\mu^{-1})^2\right)\geq-[\F:\F_p]+\frac{2}{p-1}.\] Considering $\valp$ for \eqref{twistH}, $\chi\neq\mu^{-1}$ and $t=-2$, we have 
\begin{align*}
\val_p\left(\frac{q^{1/2}}{q-1}\eps(1/2,\chi^{-1}\mu^{-1})^2\eps(1/2,\chi)\right)&=\val_p\left(\frac{1}{q-1}\right)+\val_p(q^{1/2})\\&\quad+2\val_p(\eps(1/2,\chi^{-1}\mu^{-1}))+\val_p(\eps(1/2,\chi)).
\end{align*}
Using \eqref{l1}, we have that \[\val_p(\eps(1/2,\chi^{-1}\mu^{-1}))=-\frac{[\F:\F_p]}{2}+\frac{s(\chi\mu)}{p-1} \ \text{and} \ \val_p(\eps(1/2,\chi))=-\frac{[\F:\F_p]}{2}+\frac{s(\chi^{-1})}{p-1}.\] Since, $\chi$ is non-trivial and $\chi\neq\mu^{-1}$ we know that $\chi\mu\neq\1$. Hence, we can bound both of $s(\chi\mu)$ and $s(\chi^{-1})$ by one. Therefore,
\begin{align*}
\val_p\left(\frac{q^{1/2}}{q-1}\eps(1/2,\chi^{-1}\mu^{-1})^2\eps(1/2,\chi)\right)&\geq\val_p\left(\frac{1}{q-1}\right)+\val_p(q^{1/2})-[\F:\F_p]\\
&\quad -\frac{1}{2}[\F:\F_p]+\frac{3}{p-1}\\
&=\frac{1}{2}[\F:\F_p]-\frac{1}{2}[\F:\F_p]-[\F:\F_p]+\frac{3}{p-1}\\
&=-[\F:\F_p]+\frac{3}{p-1}.
\end{align*}
In the setting of \eqref{twistI}, $\chi=\mu^{-1}$ and $t=-2$, we obtain by using \eqref{l1},
\begin{align*}
\valp\left(\frac{q^{-1/2}}{q-1}\eps(1/2,\chi)\right)&=\valp(q^{-1/2})+\valp\left(\frac{1}{q-1}\right)+\val_p(\eps(1/2,\chi))\\
&\geq-\frac{1}{2}[\F:\F_p]-\frac{1}{2}[\F:\F_p]+\frac{1}{p-1}\\
&=-[\F:\F_p]+\frac{1}{p-1}.
\end{align*}
In the final non-trivial case for $c(\pi)=2$, \eqref{twistJ}, $\chi=\mu^{-1}$ and $t>-2$, we have \[\val_p\big(-(q+1)q^{-(2+t+1/2)}\eps(1/2,\chi)\big)=\val_p\big(-(q+1)q^{-(t+5/2)}\big)+\val_p(\eps(1/2,\chi)).\] We can again use \eqref{l1} for $\val_p(\eps(1/2,\chi))$. So, \[\valp(\eps(1/2,\chi))\geq-\frac{[\F:\F_p]}{2}+\frac{1}{p-1}.\] Therefore,
\begin{align*}
\val_p\big(-(q+1)q^{-(2+t+1/2)}\eps(1/2,\chi)\big)&\geq\valp\big(-(q+1)q^{-(t+5/2)}\big)-\frac{[\F:\F_p]}{2}+\frac{1}{p-1}\\
&\geq\min\left\{\val_p(-q^{-(t+3/2)}),\val_p(q^{-(t+5/2)})\right\}\\
&\quad-\frac{[\F:\F_p]}{2}+\frac{1}{p-1}\\
&=-[\F:\F_p](t+5/2)-\frac{[\F:\F_p]}{2}+\frac{1}{p-1}\\
&=-(t+3)[\F:\F_p]+\frac{1}{p-1}.
\end{align*}
In the remaining case, \eqref{twistK}, we have $\val_p(\ctl)=\infty$. Collecting all these cases, we have that \[\val_p(c_{t,l}(\chi))\geq
\begin{cases}
-[\F:\F_p]+\frac{2}{p-1}, \ \text{if} \ \chi=\1,t=-2,l=1,\\
-[\F:\F_p]+\frac{3}{p-1}, \ \text{if} \ \chi\notin\{\1,\mu^{-1}\},t=-2,l=1,\\
-[\F:\F_p]+\frac{1}{p-1}, \ \text{if} \ \chi=\mu^{-1}, t=-2,l=1,\\
-(t+3)[\F:\F_p]+\frac{1}{p-1}, \ \text{if} \ \chi=\mu^{-1}, t\geq-1,l=1,\\
\infty, \ \text{otherwise.}
\end{cases}
\] We can see here that the minimum of all these cases is given by \[-(t+3)[\F:\F_p]+\frac{1}{p-1},\] thus, \[\valp(W_\pi(\gtlv))\geq-(t+3)[\F:\F_p]+\frac{1}{p-1}.\] If we now consider $c(\pi)>2$ then $c(\mu)>1$. There is just one case when $\chi=\1$, which is \eqref{twistC}. Since $c(\mu)>1$, using \eqref{l1} gives \[\valp(c_{-2,1}(\1))=\valp\left(-\frac{1}{q-1}\eps(1/2,\mu^{-1})^2\right)=0.\] This deals with when $\chi=\1$, therefore we can now assume that $\chi\neq\1$. So, we either have $c(\chi)=1$ or $c(\chi)>1$. If we start with $c(\chi)=1$, then this can only occur in the setting of \eqref{twistH}, $t=-2c(\chi\mu)$ and $\chi\neq\mu^{-1}$. We first deal with the case of $c(\chi\mu)=1$ and so $t=-2$. Here, we have 
\begin{align*}
\valp\left(\frac{q^{1-c(\chi)/2}}{q-1}\eps(1/2,\chi)\eps(1/2,\chi^{-1}\mu^{-1})^2\right)&=\val_p(q^{1/2})+\val_p(\eps(1/2,\chi))\\
&\quad+2\valp(\eps(1/2,\chi^{-1}\mu^{-1}))\\
&\geq\frac{[\F:\F_p]}{2}-\frac{[\F:\F_p]}{2}+\frac{1}{p-1}-[\F:\F_p]+\frac{2}{p-1}\\
&=-[\F:\F_p]+\frac{3}{p-1}.
\end{align*}
Now if $c(\chi\mu)>1$, we have \[\valp(\ctl)\geq\frac{[\F:\F_p]}{2}-\frac{[\F:\F_p]}{2}+\frac{1}{p-1}=\frac{1}{p-1}.\] We can now assume that $c(\chi)>1$ and so $\valp(\eps(1/2,\chi))=0$. Starting with \eqref{twistH}, $t=-2c(\chi\mu)$ with $c(\chi\mu)=1$ and $\chi\neq\mu^{-1}$, we have 
\begin{align*}
\valp\left(\frac{q^{1-c(\chi)/2}}{q-1}\eps(1/2,\chi)\eps(1/2,\chi^{-1}\mu^{-1})^2\right)&=\val_p(q^{1-c(\chi)/2})+2\valp(\eps(1/2,\chi^{-1}\mu^{-1}))\\
&=-\frac{c(\chi)}{2}[\F:\F_p]+\frac{2}{p-1}.
\end{align*}
For $c(\chi\mu)>1$, we have \[\valp(\ctl)=\left(1-\frac{c(\chi)}{2}\right)[\F:\F_p].\] In the case of \eqref{twistI}, $t=-2$ and $\chi=\mu^{-1}$, we have \[\val_p\left(\frac{q^{-c(\mu^{-1})/2}}{q-1}\eps(1/2,\mu^{-1})\right)=\val_p(q^{-c(\mu^{-1})/2})=-\frac{c(\mu^{-1})}{2}[\F:\F_p].\] The final non-trivial case is that of \eqref{twistJ}, $\chi=\mu^{-1}$ and $t>-2$, where we have the following,
\begin{align*}
\val_p(\ctl)&=\val_p\big(-(q+1)q^{-2-t-c(\mu^{-1})/2}\eps(1/2,\mu^{-1})\big)\\
&=\val_p\big(q^{-1-t-c(\mu^{-1})/2}\eps(1/2,\mu^{-1})+q^{-2-t-c(\mu^{-1})/2}\eps(1/2,\mu^{-1})\big)\\
&\geq\min\left\{\val_p\big(q^{-1-t-c(\mu^{-1})/2}\eps(1/2,\mu^{-1})\big),\val_p\big(q^{-2-t-c(\mu^{-1})/2}\eps(1/2,\mu^{-1})\big)\right\}\\
&=\val_p\left(q^{-2-t-c(\mu^{-1})/2}\eps(1/2,\mu^{-1})\right)\\
&=\val_p\left(q^{-2-t-c(\mu^{-1})/2}\right), \ \text{by using \eqref{l1}}\\
&=-\left(2+t+\frac{c(\mu^{-1})}{2}\right)[\F:\F_p].
\end{align*}
By collecting all these cases together we have that
\[\valp(\ctl)\geq
\begin{cases}
0, \ \text{if} \ \chi=\1, t=-2c(\mu),\\
-[\F:\F_p]+\frac{3}{p-1},  \ \text{if} \ \chi\neq\mu^{-1}, c(\chi\mu)=1, t=-2, c(\chi)=1, \\
\frac{1}{p-1}, \ \text{if} \ \chi\neq\mu^{-1}, c(\chi\mu)>1, t=-2c(\chi\mu), c(\chi)=1,\\
-\frac{c(\chi)}{2}[\F:\F_p]+\frac{2}{p-1}, \ \text{if} \ \chi\neq\mu^{-1}, c(\chi\mu)=1, t=-2, c(\chi)>1, \\
\left(1-\frac{c(\chi)}{2}\right)[\F:\F_p], \ \text{if} \ \chi\neq\mu^{-1}, c(\chi\mu)>1, t=-2c(\chi\mu), c(\chi)>1,\\
-\frac{c(\mu^{-1})}{2}[\F:\F_p],  \ \text{if} \ \chi=\mu^{-1}, t=-2, l=c(\mu^{-1}), \\
-\left(2+t+\frac{c(\mu^{-1})}{2}\right)[\F:\F_p],  \ \text{if} \ \chi=\mu^{-1}, t>-2, l=c(\mu^{-1}),\\
\infty, \ \text{otherwise}.
\end{cases}
\] We now need to find the minimum of all of these different possibilities.  We can see that if $t>-2$, then this will be \[-\left(2+t+\frac{c(\mu^{-1})}{2}\right)[\F:\F_p].\] If $t\leq-2$, then we can see that the minimum is \[-\frac{c(\chi)}{2}[\F:\F_p]+\frac{2}{p-1}.\] As $\chi\in\X_{\leq l}$, we have $c(\chi)\leq c(\pi)$. Therefore, \[\valp\left(W_\pi(\gtlv)\right)\geq\min\left\{-\frac{c(\pi)}{2}[\F:\F_p]+\frac{2}{p-1},-\left(2+t+\frac{c(\mu^{-1})}{2}\right)[\F:\F_p]\right\}.\]

\subsubsection{$\pi$ is of Type 3a} 

Here, we evaluate the $p$-adic valuation of \eqref{psC}-\eqref{psG}. Starting with the case $c(\pi)=2$, this means $c(\beta_1),c(\beta_2)=1$. So, using \eqref{l1}, we know \[\valp(\eps(1/2,\beta_1)),\valp(\eps(1/2,\beta_2))\geq-\frac{[\F:\F_p]}{2}+\frac{1}{p-1}.\] Starting with \eqref{psC}, we have \[\valp(c_{-n,1}(\1))=\valp\left(q^{c(\beta_1^{-1})c_1-c(\beta_2^{-1})c_1}\eps(1/2,\beta_1^{-1})\eps(1/2,\beta_2^{-1})\right)\geq-[\F:\F_p]+\frac{2}{p-1}.\] In the second setting, \eqref{psD}, since $\chi\neq\beta_1^{-1},\beta_2^{-1}$ we know that $\chi^{-1}\beta_1^{-1},\chi^{-1}\beta_2^{-1}\neq\1$. Therefore, we have 
\begin{align*}
\valp(\ctl)&=\valp\left(\frac{q^{1/2}}{q-1}q^{c(\chi^{-1}\beta_1^{-1})c_1-c(\chi^{-1}\beta_2^{-1})c_1}\eps(1/2,\chi)\eps(1/2,\chi^{-1}\beta_1^{-1})\eps(1/2,\chi^{-1}\beta_2^{-1})\right)\\
&\geq \frac{[\F:\F_p]}{2}+\valp(q^{c_1-c_1})-\frac{3}{2}[\F:\F_p]+\frac{3}{p-1}\\
&=-[\F:\F_p]+\frac{3}{p-1}.
\end{align*}
Next, we consider the case of \eqref{psE}. Hence, \[\valp(\ctl)=\valp\left(-\frac{q^{c_i+c_j}}{q-1}\eps(1/2,\bi^{-1})\eps(1/2,\bi\bj^{-1})\right)\geq\valp(q^{c_1+c_2})-[\F:\F_p]+\frac{2}{p-1}.\] Finally, we are in the case of \eqref{psF}. In this case, we have
\begin{align*}
\valp(\ctl)&=\valp\left(q^{-t/2-1/2-1/2-c_i}q^{-c_i(t+a(\beta_i^{-1}\beta_j))}\eps(1/2,\bi^{-1})\eps(1/2,\bi\bj^{-1})\right)\\
&\geq-\frac{t+2}{2}[\F:\F_p]-\valp(\qci)-(t+1)\valp(q^{c_i})-[\F:\F_p]+\frac{2}{p-1}\\
&=-\frac{t+4}{2}[\F:\F_p]-(t+2)\valp(\qci)+\frac{2}{p-1}.
\end{align*}
Using the fact that $x\leq|x|$ and $t+2\geq0$, we have \[\valp(\ctl)\geq-\frac{t+4}{2}[\F:\F_p]+\frac{2}{p-1}-(t+2)|\valp(\qci)|.\]
Therefore, by collecting all of these together, 
\[\valp(\ctl)\geq
\begin{cases}
-[\F:\F_p]+\frac{2}{p-1}, \\
\quad \text{if} \ \chi=\1,\beta_1^{-1},\beta_2^{-1}\neq\1, t=-c(\mu_1)-c(\mu_2), l=1,\\
-[\F:\F_p]+\frac{3}{p-1},\\
\quad \text{if} \ \chi\notin\left\{\1,\beta_1^{-1},\beta_2^{-1}\right\}, t=-c(\chi\mu_1)-c(\chi\mu_2), l=c(\chi),\\
-[\F:\F_p]+\frac{2}{p-1}, \\
\quad \text{if} \ \chi=\bi^{-1}, \{j,i\}=\{1,2\}, t=-c(\chi\mu_j)-1, l=c(\beta_i^{-1}), \\
-\frac{t+4}{2}[\F:\F_p]+\frac{2}{p-1}-(t+2)|\valp(\qci)|, \\
\quad \text{if} \ \chi=\bi^{-1}, \{j,i\}=\{1,2\}, t\geq-c(\chi\mu_j), l=c(\beta_i^{-1}),\\
\infty, \ \text{otherwise}.
\end{cases}
\] 
We now need to compute the minimum of all these cases as $\chi\in\X_{\leq l}$ varies. Here, we can see that we have a minimum for all values of $t$ is given by \[-\frac{t+4}{2}[\F:\F_p]+\frac{2}{p-1}-(t+2)|\valp(\qci)|.\] Thus, \[\valp(W_\pi(\gtlv))\geq-\frac{t+4}{2}[\F:\F_p]+\frac{2}{p-1}-(t+2)|\valp(\qci)|.\] Now, we need to consider the case of $c(\pi)>2$. Recall that $c(\pi)=c(\beta_1)+c(\beta_2)$. Since $c(\pi)>2$, we know that at least one of $c(\beta_1)$ and $c(\beta_2)$ have to be greater than one. We first consider the slightly simpler case when both $c(\bone)$ and $c(\btwo)$ are greater than one. This means we have $\valp(\eps(1/2,\bone))=\valp(\eps(1/2,\btwo))=0$.  

In our first case, \eqref{psC}, we obtain 
\begin{align*}
\valp(c_{-n,1}(\1))=c(\bone^{-1})\valp(\qcone)-c(\btwo^{-1})\valp(\qctwo)&=(c(\bone^{-1})-c(\btwo^{-1}))\valp(\qci)\\
&\geq-|c(\bone^{-1})-c(\btwo^{-1})| \cdot |\valp(\qci)|.
\end{align*}
In \eqref{psD} there are two cases to consider these being $c(\chi)=l=1$ or $l>1$. In either of these cases we know that $\chi^{-1}\bone^{-1}\neq\1$ and $\chi^{-1}\btwo^{-1}\neq\1$ and so we can have $c(\chi^{-1}\bone^{-1})=1=c(\chi^{-1}\btwo^{-1})$. So we must compute \[\valp(\ctl)=\valp\left(\frac{q^{1-c(\chi)/2}}{q-1}q^{c(\chi^{-1}\beta_1^{-1})c_1-c(\chi^{-1}\beta_2^{-1})c_1}\eps(1/2,\chi)\eps(1/2,\chi^{-1}\beta_1^{-1})\eps(1/2,\chi^{-1}\beta_2^{-1})\right).\] Starting with $l=1$, we have \[\valp(\ctl)\geq-[\F:\F_p]+\frac{3}{p-1}.\] If we now consider the case of $l>1$, we obtain \[\valp(\ctl)\geq-\frac{l}{2}[\F:\F_p]+\frac{2}{p-1}.\] Now, suppose $c(\chi\beta_1^{-1}),c(\chi\beta_2^{-1})>1$. In addition suppose $c(\chi)=1$, then 
\begin{align*}
\valp(\ctl)&\geq\frac{[\F:\F_p]}{2}+\valp\left(q^{c(\chi^{-1}\beta_1^{-1})c_1-c(\chi\beta_2^{-1})c_1}\right)-\frac{[\F:\F_p]}{2}+\frac{1}{p-1}\\
&=\left(c(\chi^{-1}\beta_1^{-1})-c(\chi^{-1}\beta_2^{-1})\right)\valp(\qcone)+\frac{1}{p-1}\\
&\geq-\left|c(\chi^{-1}\beta_1^{-1})-c(\chi^{-1}\beta_2^{-1})\right| \cdot |\valp(\qcone)|+\frac{1}{p-1}.
\end{align*}
On the other hand if $c(\chi)>1$, gives 
\begin{align*}
\valp(\ctl)&=\left(1-\frac{c(\chi)}{2}\right)[\F:\F_p]+\left(c(\chi^{-1}\beta_1^{-1})-c(\chi^{-1}\beta_2^{-1})\right)\valp(\qcone)\\
&\geq\left(1-\frac{c(\chi)}{2}\right)[\F:\F_p]-\left|c(\chi^{-1}\beta_1^{-1})-c(\chi^{-1}\beta_2^{-1})\right||\valp(\qcone)|.
\end{align*}
In the setting of \eqref{psE}, we can take $c(\bi\bj^{-1})=1$. So, computing $p$-adic valuation gives 
\begin{align*}
\valp(\ctl)&=\valp\left(-\frac{q^{(1-c(\bi^{-1}))/2}}{q-1}q^{c_i-c(\bi\bj^{-1})c_i}\eps(1/2,\bi^{-1})\eps(1/2,\bi\bj^{-1})\right)\\
&\geq-\frac{l}{2}[\F:\F_p]+\frac{1}{p-1}.
\end{align*}
In the final non-trivial case, \eqref{psF}, we again want to take $c(\bi\bj^{-1})=1$. Therefore, 
\begin{align*}
\valp(\ctl)&=\valp\Big(q^{-t/2-c(\bi^{-1})/2-c(\bi^{-1}\bj)/2-c_i}q^{-c_i(t+c(\beta_i^{-1}\beta_j))}\eps(1/2,\bi^{-1})\eps(1/2,\bi\bj^{-1})\Big)\\
&=-\frac{t+l+1+c(\bi^{-1}\bj)}{2}[\F:\F_p]-\valp(\qci)-(t+c(\bi^{-1}\bj))\valp(\qci)+\frac{1}{p-1}\\
&\geq-\frac{t+l+1+c(\bi^{-1}\bj)}{2}[\F:\F_p]+\frac{1}{p-1}-(t+c(\bi^{-1}\bj)+1)|\valp(\qci)|.
\end{align*}
Thus, for $\pi$ of Type 3a with $c(\pi)>2$ and $c(\bone),c(\btwo)>1$, we have 
\[\valp(\ctl)\geq
\begin{cases}
-|c(\bone^{-1})-c(\btwo^{-1})|\cdot|\valp(\qci)|, \\
\quad \text{if} \ \chi=\1,\beta_1^{-1},\beta_2^{-1}\neq\1, t=-c(\mu_1)-c(\mu_2), l=1,\\
-[\F:\F_p]+\frac{3}{p-1}, \\
\quad \text{if} \ \chi\not\in\{\1,\beta_1^{-1},\beta_2^{-1}\}, c(\chi\mu_1),c(\chi\mu_2)=1, t=-2, l=c(\chi)=1,\\
-\left|c(\chi^{-1}\beta_1^{-1})-c(\chi^{-1}\beta_2^{-1})\right| \cdot |\valp(\qcone)|+\frac{1}{p-1},\\
\quad \text{if} \ \chi\not\in\{\1,\beta_1^{-1},\beta_2^{-1}\}, c(\chi\mu_1),c(\chi\mu_2)>1, t=-c(\chi\mu_1)-c(\chi\mu_2), l=c(\chi)=1,\\
-\frac{c(\chi)}{2}[\F:\F_p]+\frac{2}{p-1}, \\
\quad \text{if} \ \chi\not\in\{\1,\beta_1^{-1},\beta_2^{-1}\}, c(\chi\mu_1),c(\chi\mu_2)=1, t=-2, l=c(\chi)>1,\\
\left(1-\frac{c(\chi)}{2}\right)[\F:\F_p]-\left|c(\chi^{-1}\beta_1^{-1})-c(\chi^{-1}\beta_2^{-1})\right|\cdot|\valp(\qcone)|,\\
\quad \text{if} \ \chi\not\in\{\1,\beta_1^{-1},\beta_2^{-1}\}, c(\chi\mu_1),c(\chi\mu_2)>1, t=-c(\chi\mu_1)-c(\chi\mu_2), l=c(\chi)>1,\\
-\frac{l}{2}[\F:\F_p]+\frac{1}{p-1},\\
\quad \text{if} \ \chi=\bi^{-1}, \{j,i\}=\{1,2\}, t=-c(\chi\mu_j)-1, l=c(\beta_i^{-1}), \\
-\frac{t+l+1+c(\bi^{-1}\bj)}{2}[\F:\F_p]+\frac{1}{p-1}-(t+c(\bi^{-1}\bj)+1)|\valp(\qci)|, \\
\quad \text{if} \ \chi=\bi^{-1}, \{j,i\}=\{1,2\}, t\geq-c(\chi\mu_j), l=c(\beta_i^{-1}),\\
\infty, \ \text{otherwise}.
\end{cases}
\] We now need to compute the minimum of all these cases for $\chi$ varying in $\X_{\leq l}$. For $t\geq-c(\chi\bj)$, we see that the minimum of all the possibilities above is 
\begin{equation}\label{bothgreat}
-\frac{t+c(\bone^{-1})+1+c(\bi^{-1}\bj)}{2}[\F:\F_p]+\frac{1}{p-1}-(t+c(\bi^{-1}\bj)+1)|\valp(\qci)|.
\end{equation}
For $t<-c(\chi\bj)$ the bound in \eqref{bothgreat} might no longer be the minimum. Note, we can bound $-|c(\chi^{-1}\bone^{-1})-c(\chi^{-1}\btwo^{-1})|$ by $-c(\pi)$. We are also able to bound $c(\chi)$ by $c(\pi)$. This means that there are two possible cases for the minimum for $\chi\in\X_{\leq l}$ with $t<-c(\chi\bj)$, \[\left(1-\frac{c(\pi)}{2}\right)[\F:\F_p]-c(\pi)|\valp(\qcone)| \ \text{and} \ -\frac{c(\pi)}{2}[\F:\F_p]+\frac{1}{p-1}.\] This can be reformulated as \[-\frac{c(\pi)}{2}[\F:\F_p]+\min\left\{\frac{1}{p-1},[\F:\F_p]-c(\pi)|\valp(\qcone)|\right\}.\] So, 
\begin{align*}
\min_{\chi\in\X_{\leq l}}\{\valp(\ctl)\}=\min\Bigg\{&-\frac{c(\pi)}{2}[\F:\F_p]+\min\left\{\frac{1}{p-1},[\F:\F_p]-c(\pi)|\valp(\qcone)|\right\},\\
&-\frac{t+c(\bone^{-1})+1+c(\bi^{-1}\bj)}{2}[\F:\F_p]+\frac{1}{p-1}\\
&-(t+c(\bi^{-1}\bj)+1)|\valp(\qci)|\Bigg\}.
\end{align*}
For the overall bound for $\valp(W_\pi(\gtlv))$ Recall the discussion provided after Proposition \ref{Won}.  Therefore, from this, \eqref{valWon} and Proposition \ref{Won} (iii), we have 
\begin{align*}
\valp(W_\pi(\gtlv))\geq\min\Bigg\{&-\frac{c(\pi)}{2}[\F:\F_p]+\min\left\{\frac{1}{p-1},[\F:\F_p]-c(\pi)|\valp(\qcone)|\right\},\\
&-\frac{t+c(\bone^{-1})+1+c(\bi^{-1}\bj)}{2}[\F:\F_p]+\frac{1}{p-1}\\
&-(t+c(\bi^{-1}\bj)+1)|\valp(\qci)|\Bigg\}\\
&+
\begin{cases}
-|c(\bone^{-1})-c(\btwo^{-1})|\cdot|\valp(\qcone),\ \text{if} \ \valp(\eps(1/2,\pi))>0,\\
0, \ \text{if} \ \valp(\eps(1/2,\pi))\leq0.
\end{cases}
\end{align*}

We now consider the case when one of $c(\bone)$ or $c(\btwo)$ is equal to one. The case of interest is $c(\bi)=1$ and $c(\bj)>1$. We again, work through the cases of \eqref{psC}-\eqref{psF}. In the setting of \eqref{psC}, we have the following two possibilities, either 
\begin{align*}
\valp(c_{-n,1}(\1))&\geq\valp(\qcone)-c(\btwo^{-1})\valp(\qcone)-\frac{[\F:\F_p]}{2}+\frac{1}{p-1}\\
&\geq-\frac{[\F:\F_p]}{2}+\frac{1}{p-1}-(c(\btwo^{-1})-1)|\valp(\qcone)|,
\end{align*}
or 
\begin{align*}
\valp(c_{-n,1}(\1))&\geq c(\bone^{-1})\valp(\qcone)-\valp(\qcone)-\frac{[\F:\F_p]}{2}+\frac{1}{p-1}\\
&\geq-\frac{[\F:\F_p]}{2}+\frac{1}{p-1}-(c(\bone^{-1})-1)|\valp(\qcone)|.
\end{align*}
The case of \eqref{psD} is the same as in the previous setting. In the case of \eqref{psE} we can again set $c(\bi\bj^{-1})=1$. Thus \[\valp(\ctl)\geq-[\F:\F_p]+\frac{2}{p-1}.\] The final non-trivial setting to consider is \eqref{psF}. We have 
\begin{align*}
\valp(\ctl)&=\valp\bigg(q^{-t/2-1/2-c(\bi^{-1}\bj)/2-c_i}q^{-c_i(t+c(\beta_i^{-1}\beta_j))}\eps(1/2,\bi^{-1})\eps(1/2,\bi\bj^{-1})\bigg)\\
&\geq-\frac{t+1+c(\bi^{-1}\bj)}{2}[\F:\F_p]-\valp(\qci)-(t+c(\bi^{-1}\bj))\valp(\qci)\\
&\quad -[\F:\F_p]+\frac{2}{p-1}\\
&=-\frac{t+3+c(\bi^{-1}\bj)}{2}[\F:\F_p]-(t+c(\bi^{-1}\bj)+1)\valp(\qci)+\frac{2}{p-1}\\
&\geq-\frac{t+3+c(\bi^{-1}\bj)}{2}[\F:\F_p]+\frac{2}{p-1}-(t+c(\bi^{-1}\bj)+1)|\valp(\qci)|.
\end{align*}
Collecting all these together gives
\[\valp(\ctl)\geq
\begin{cases}
\frac{[\F:\F_p]}{2}+\frac{1}{p-1}-(c(\btwo^{-1})-1)|\valp(\qcone)|,\\
\quad \text{if} \ \chi=\1,\beta_1^{-1},\beta_2^{-1}\neq\1, t=-c(\mu_1)-c(\mu_2), c(\bone)=l=1,\\
-\frac{[\F:\F_p]}{2}+\frac{1}{p-1}-(c(\bone^{-1})-1)|\valp(\qcone)|, \\
\quad \text{if} \ \chi=\1,\beta_1^{-1},\beta_2^{-1}\neq\1, t=-c(\mu_1)-c(\mu_2), c(\btwo)=l=1,\\
-[\F:\F_p]+\frac{3}{p-1}, \\
\quad \text{if} \ \chi\neq\1,\beta_1^{-1},\beta_2^{-1}, t=-c(\chi\mu_1)-c(\chi\mu_2), l=c(\chi)=1,\\
-\frac{l}{2}[\F:\F_p]+\frac{2}{p-1}, \\
\quad \text{if} \ \chi\neq\1,\beta_1^{-1},\beta_2^{-1}, t=-c(\chi\mu_1)-c(\chi\mu_2), l=c(\chi)>1,\\
-[\F:\F_p]+\frac{2}{p-1}, \\
\quad \text{if} \ \chi=\bi^{-1}, \{j,i\}=\{1,2\}, t=-c(\chi\mu_j)-1, l=c(\beta_i^{-1}), c(\bi)=1, \\
-\frac{t+3+c(\bi^{-1}\bj)}{2}[\F:\F_p]+\frac{2}{p-1}-(t+c(\bi^{-1}\bj)+1)|\valp(\qci)|, \\
\quad \text{if} \ \chi=\bi^{-1}, \{j,i\}=\{1,2\}, t\geq-c(\chi\mu_j), l=c(\beta_i^{-1})=1,\\
\infty, \ \text{otherwise}.
\end{cases}
\] For $t>-c(\chi\bj)$, we have the minimum of all of the above being \[-\frac{t+3+c(\bi^{-1}\bj)}{2}[\F:\F_p]+\frac{2}{p-1}-(t+c(\bi^{-1}\bj)+1)|\valp(\qci)|.\] If $t\leq-c(\chi\bj)$, then there are several different possibilities.  Note, $c(\bone)-1, c(\btwo)-1=c(\pi)-2$ and $c(\chi)\leq c(\pi)$. So the possible minimums are given by \[-\frac{c(\pi)}{2}[\F:\F_p]+\frac{2}{p-1} \ \text{or} \ -\frac{[\F:\F_p]}{2}+\frac{1}{p-1}-(c(\pi)-2)|\valp(\qcone)|.\] As in the previous setting of Type 3a with $c(\pi)>2$, we need to also include the $p$-adic valuation for $\eps(1/2,\pi)$. Hence, 
\begin{align*}
\valp\left(W_\pi(\gtlv)\right)\geq\min\Bigg\{&-\frac{c(\pi)}{2}[\F:\F_p]+\frac{2}{p-1}, -\frac{[\F:\F_p]}{2}+\frac{1}{p-1}-(c(\pi)-2)|\valp(\qcone)|,\\
&-\frac{t+3+c(\bi^{-1}\bj)}{2}[\F:\F_p]+\frac{2}{p-1}-(t+c(\bi^{-1}\bj)+1)|\valp(\qci)|\Bigg\}\\
&+
\begin{cases}
-(c(\pi)-2)|\valp(\qcone)|-\frac{[\F:\F_p]}{2}, \ \text{if} \ \valp(\eps(1/2,\pi))>0,\\
0, \ \text{if} \ \valp(\eps(1/2,\pi))\leq0.
\end{cases}
\end{align*}

\subsubsection{$\pi$ is of Type 3b}

Here we compute the $p$-adic valuations for \eqref{psH}-\eqref{psM}. Recall that in this case, we have $c(\pi)=2c(\beta)$. Therefore for $c(\pi)>2$, we have $c(\beta)>1$ and so $\val_p(\eps(1/2,\beta))=0$. We start with the case of $c(\pi)=2$, so $c(\beta)=1$. Starting with \eqref{psH}, we have \[c_{-2,1}(1)\geq-[\F:\F_p]+\frac{2}{p-1}.\] The second case to consider is \eqref{psI}. Here, we have $\chi\neq\1,\beta^{-1}$. Since, $\chi\in\X_{\leq l}$ we have $c(\chi)=1$. Therefore, 
\begin{align*}
\val_p(c_{t,1}(\chi))&=\val_p\left(\frac{q^{1-c(\chi)/2}}{q-1}\eps(1/2,\chi)\eps(1/2,\chi^{-1}\beta^{-1})^2\right)\\
&=\val_p\left(\frac{q^{1/2}}{q-1}\eps(1/2,\chi)\eps(1/2,\chi^{-1}\beta^{-1})^2\right)\\
&\geq\frac{[\F:\F_p]}{2}-\frac{3}{2}[\F:\F_p]+\frac{3}{p-1}\\
&=-[\F:\F_p]+\frac{3}{p-1}.
\end{align*}
The next case is that of \eqref{psJ}, where \[\valp\big(c_{-2,1}(\chi)\big)=\valp\left(\frac{q^{-1/2}}{q-1}\eps(1/2,\beta^{-1})\right)\geq-[\F:\F_p]+\frac{1}{p-1}.\] For \eqref{psK}, we have \[c_{-1,1}(\chi)=-q^{-1}\eps(1/2,\beta^{-1})(q^{-c_1}+q^{c_1}).\] Therefore, 
\begin{align*}
\val_p(c_{-1,1}(\chi))&=\valp(q^{-1})+\valp(\eps(1/2,\beta^{-1}))+\valp(\qcone+q^{-c_1})\\
&\geq-[\F:\F_p]-\frac{[\F:\F_p]}{2}+\frac{1}{p-1}+\min\left\{-\valp(\qcone),\valp(\qcone)\right\}\\
&=-\frac{3}{2}[\F:\F_p]+\frac{1}{p-1}+\min\left\{-\valp(\qcone),\valp(\qcone)\right\}\\
&\geq-\frac{3}{2}[\F:\F_p]+\frac{1}{p-1}-|\valp(\qcone)|.
\end{align*}
Finally for \eqref{psL}, we have \[c_{t,1}(\chi)=-q^{-(3+t)/2}\eps(1/2,\beta^{-1})\Big(q^{-c_1(t+2)}+q^{c_1(t+2)}-(q-1)\sum^t_{k=0}q^{c_1(t-2k)}\Big).\] Therefore, we obtain
\begin{align*}
\valp(c_{t,1}(\chi))&=\valp\Big(-q^{-(3+t)/2}\eps(1/2,\beta^{-1})\Big(q^{-c_1(t+2)}+q^{c_1(t+2)}-(q-1)\sum^t_{k=0}q^{c_1(t-2k)}\Big)\Big)\\
&\geq-\frac{3+t}{2}[\F:\F_p]-\frac{[\F:\F_p]}{2}+\frac{1}{p-1}+\\
& \qquad \min\Big\{\valp(q^{-c_1(t+2)}),\valp(q^{c_1(t+2)}),\valp\big((q-1)\sum^t_{k=0}q^{c_1(t-2k)}\big)\Big\}\\
&\geq -\frac{4+t}{2}[\F:\F_p]+\frac{1}{p-1}+\\
& \qquad \min\Big\{-(t+2)\valp(\qcone),(t+2)\valp(\qcone),t\valp(\qcone),\\
&\qquad\qquad(t-2)\valp(\qcone),...,-t\valp(\qcone)\Big\}\\
&\geq -\frac{4+t}{2}[\F:\F_p]+\frac{1}{p-1}-(t+2)|\valp(\qcone)|. 
\end{align*}
Combining all of these cases gives
\[\valp(\ctl)\geq
\begin{cases}
-[\F:\F_p]+\frac{2}{p-1}, \\
\quad \text{if} \ \chi=\1, l=1, c(\beta)\neq0, t=-c(\mu_1)-c(\mu_2)=-n,\\
-[\F:\F_p]+\frac{3}{p-1}, \\
\quad \text{if} \ \chi\neq\1,\beta^{-1}, t=-2c(\chi\beta), l=c(\chi), \\
-[\F:\F_p]+\frac{1}{p-1}, \\
\quad \text{if} \ \chi=\beta^{-1}, t=-2, l=c(\beta^{-1}),\\
-\frac{3}{2}[\F:\F_p]+\frac{1}{p-1}-|\valp(\qcone)|,\\
\quad \text{if} \ \chi=\beta^{-1}, t=-1, l=c(\beta^{-1}),\\
 -\frac{4+t}{2}[\F:\F_p]+\frac{1}{p-1}-(t+2)|\valp(\qcone)|, \\
\quad \text{if} \ \chi=\beta^{-1}, t\geq0, l=c(\beta^{-1}), \\
 \infty, \ \text{otherwise}.
\end{cases}
\]
Taking the minimum of all of these gives \[\valp\left(W_\pi(\gtlv)\right)\geq-\frac{4+t}{2}[\F:\F_p]+\frac{1}{p-1}-(t+2)|\valp(\qcone)|.\] If we now turn our attention to the case where $c(\pi)>2$, then recall that $c(\beta)>1$. We now work through the same cases as above. Starting again, with \eqref{psH} we have \[\valp(\ctl)=0.\] For \eqref{psI}, much like Type 3a there are two cases to consider that of $c(\chi)=1$ and $c(\chi)>1$. In the first of these, with $c(\chi\beta)=1$, we obtain \[\valp(\ctl)\geq-[\F:\F_p]+\frac{3}{p-1}=-[\F:\F_p]+\frac{3}{p-1}.\] If $c(\chi\beta)>1$, then \[\valp(\ctl)=\frac{[\F:\F_p]}{2}.\] For the second of these namely, $c(\chi)>1$ and $c(\chi\beta)=1$, we have \[\valp(\ctl)\geq-\frac{l}{2}[\F:\F_p]+\frac{2}{p-1}=-\frac{l}{2}[\F:\F_p]+\frac{2}{p-1}.\] Now, if $c(\chi\beta)>1$, then \[\valp(\ctl)=\left(1-\frac{c(\chi)}{2}\right)[\F:\F_p].\] For the next setting, \eqref{psJ}, we have \[\valp(\ctl)=-\frac{l}{2}[\F:\F_p].\]For \eqref{psK}, we obtain 
\begin{align*}
\valp(\ctl)&=\valp(-q^{-(1+l)/2})+\val(\eps(1/2,\beta^{-1}))+\valp((q^{-c_1}+q^{-c_2}))\\
&\geq-\frac{1+l}{2}[\F:\F_p]+\min\{\valp(q^{-c_1}),\valp(q^{c_1})\}\\
&\geq-\frac{1+l}{2}[\F:\F_p]-|\valp(\qcone)|.
\end{align*}
In the last non-trivial case, \eqref{psL} we obtain 
\begin{align*}
\valp(\ctl)&=\valp\Big(-q^{-(2+t+c(\beta^{-1}))/2}\eps(1/2,\beta^{-1})\Big(q^{-c_1(t+2)}+q^{c_1(t+2)}-(q-1)\sum^t_{k=0}q^{c_1(t-2k)}\Big)\Big)\\
&\geq-\frac{t+l+2}{2}[\F:\F_p]+\\
&\quad \min\Big\{-(t+2)\valp(q^{c_1}),(t+2)\valp(\qcone),t\valp(\qcone),\\
&\qquad\qquad(t-2)\valp(\qcone),...,-t\valp(\qcone)\Big\}\\
&\geq-\frac{t+l+2}{2}[\F:\F_p]-(t+2)|\valp(\qcone)|.
\end{align*}
Hence, 
\[\valp(\ctl)\geq
\begin{cases}
0,\\
\quad \text{if} \ \chi=\1, l=1, c(\beta)\neq0, t=-c(\mu_1)-c(\mu_2)=-n,\\
-[\F:\F_p]+\frac{3}{p-1},\\
\quad \text{if} \ \chi\neq\1,\beta^{-1}, c(\chi\beta)=1, t=-2, l=c(\chi)=1,\\
\frac{[\F:\F_p]}{2},\\
\quad \text{if} \ \chi\neq\1,\beta^{-1}, c(\chi\beta)>1, t=-2c(\chi\beta), l=c(\chi)=1,\\
-\frac{l}{2}[\F:\F_p]+\frac{2}{p-1},\\
\quad \text{if} \ \chi\neq\1,\beta^{-1}, c(\chi\beta)=1, t=-2, l=c(\chi)>1,\\
\left(1-\frac{c(\chi)}{2}\right)[\F:\F_p],\\
\quad \text{if} \ \chi\neq\1,\beta^{-1}, c(\chi\beta)>1, t=-2c(\chi\beta), l=c(\chi)>1,\\
-\frac{l}{2}[\F:\F_p],\\
\quad \text{if} \ \chi=\beta^{-1}, t=-2, l=c(\beta^{-1}),\\
-\frac{1+l}{2}[\F:\F_p]-|\valp(\qcone)|,\\
\quad \text{if} \ \chi=\beta^{-1}, t=-1, l=c(\beta^{-1}),\\
-\frac{t+l+2}{2}[\F:\F_p]-(t+2)|\valp(\qcone)|,\\
\quad \text{if} \ \chi=\beta^{-1}, t\geq0, l=c(\beta^{-1}), \\
 \infty, \ \text{otherwise}.
\end{cases}
\]
We now need to find the minimum of all of these cases. In the situation that $t>-2$, we have \[\valp(W_\pi(\gtlv))\geq-\frac{t+c(\beta^{-1})+2}{2}[\F:\F_p]-(t+2)|\valp(\qcone)|.\]  For $t\leq-2$,  recall $c(\chi)\leq c(\pi)$, we have the bound \[-\frac{c(\pi)}{2}[\F:\F_p]+\frac{2}{p-1}.\] Therefore, for $\pi$ of Type 3b with $c(\pi)>2$, we have \[\valp\left(W_\pi(\gtlv)\right)\geq\min\left\{-\frac{c(\pi)}{2}[\F:\F_p]+\frac{2}{p-1},\frac{t+c(\beta^{-1})+2}{2}[\F:\F_p]-(t+2)|\valp(\qcone)|\right\}.\]

\begin{Remark}
Throughout the proofs of these theorems, we have used the bound $s(\cdot)>1$. So, it might be that the bounds stated in Theorem \ref{api2} \& \ref{apiplus} are not sharp. One would need to study the properties of the multiplicative characters more carefully in each specific case to obtain possible better bounds. 
\end{Remark}

\begin{Remark}
Note in the setting of $c(\pi)>2$, we do not have a uniform bound that works for every value of $t$.  The problem is for negative values of $t$ the bound which works for positive values of $t$ is no longer the minimum one. However, for most values of $t$ which has $W_\pi(\gtlv)\neq0$, the bound for $t>0$ is the minimum of all possible cases.
\end{Remark} 


\chapter{Fourier Coefficients of Hilbert Modular Forms at Cusps}\label{Chap3}
\chaptermark{Hilbert Modular Forms}

\section{Introduction}

Let $f$ be a normalised Hecke eigenform for $\Gamma_0(N)$ of weight $k$. It is known that the field generated by the Fourier coefficients of $f$ is a number field, see \cite[Proposition 2.8]{Shi78}. This number field we denote by $\Q(f)$. Furthermore, for any matrix $\sigma\in\SL_2(\Z)$, an application of the $q$-expansion principle shows that the Fourier coefficients of $f|_k\sigma$ lie in the cyclotomic extension $\Q(f)(\zeta_N)$, see \cite[Remark 12.3.5]{MFandMC} for more details. In recent years, algorithms have been developed to compute these Fourier coefficients. These were developed by several different authors, see \cite{HC,MDMN} for more details about these algorithms. Their methods use the knowledge that these numbers are in fact algebraic numbers. This means if we know more about the number field where these Fourier coefficients lie, this could speed up the computations. 

Therefore the question is given a Hecke eigenform (which we may assume to be a newform) $f$ of level $N$ and weight $k$ and $\sigma\in\SL_2(\Z)$; what is the number field that the Fourier coefficients of $f|_k\sigma$ generate? Or a slightly weaker question is: can one write down an explicit subfield of $\Q(f)(\zeta_N)$, depending on the entries of $\sigma$, which contains all the Fourier coefficients of $f|_k\sigma$? This was answered in a paper of Brunault and Neururer, who proved the following result \cite[Theorem 4.1]{FEatC}.  

\textit{Let $f$ be a normalised newform on $\Gamma_0(N)$ of weight $k$. Let $\Q(f)$ be the field generated by all the Fourier coefficients of $f$. Let $\sigma=(\begin{smallmatrix} a & b \\ c & d \end{smallmatrix})\in\sl2Z$. Then the Fourier coefficients of $f|_k \sigma$ lie in the cyclotomic extension $\Q(f)(\zeta_{N'})$ where $N'=N/(cd,N)$}.

The above result as stated is optimal. By optimal we mean that this is the field generated by the Fourier coefficients $f|_k\sigma$. The method of Brunault and Neururer was classical. They use a result of Shimura \cite[Theorem 8]{Shi75}, which studied the connections between two actions on spaces of modular forms: the action of $\GL_2^+(\Q)$ via the slash-operator and the action of $\Aut(\C)$ on the Fourier coefficient of a modular form. The proof of Brunault and Neururer also applies to modular forms of $\Gamma_0(N)$ that are not necessarily newforms but in that case it is not known if the field is optimal. In this chapter we give a new proof of the result of Brunault and Neururer as well as a substantial generalisation (to the case of Hilbert modular forms) using adelic and local representation-theoretic methods. Specifically we use local Whittaker functions and their invariance properties. The starting point is an explicit formula for the Fourier coefficients of $f|_k\sigma$ given in \cite[Section 3]{ACAS}, in terms of local Whittaker functions. From there we find sufficient conditions for $\tau\in\Aut(\C)$ to fix this product. The advantage of this method is that it gives an insight into how one could prove analogous results for other families of modular forms, specifically it can be generalised to Hilbert modular forms. One would suspect that a similar method can also be used to prove a result of the same style for Whittaker coefficients (or other factorisable periods) of automorphic forms lying in cohomological automorphic representations of higher rank groups.

The main result of this chapter is an extension of Brunault and Neururer stated above to the case of cuspidal Hilbert newforms. We recall some of the objects defined in Section \ref{hilbert}. Let $F$ be a totally real number field of degree $n$ with narrow class group of size $h$. Let $F_+$ denote the set of all totally positive elements in $F$. Let $\n$ be a fixed integral ideal of $\O_F$. The subgroup $\GL_2^+(F)$ of $\GL_2(F)$ is the subgroup consisting of all elements in $\GL_2(F)$ who have a totally positive determinant. For $\mu=1,...,h$, we define the congruence subgroup $\Gamma_\mu(\n)$ of $\GL_2^+(F)$ as \[\congsub=\left\{\begin{pmatrix} a & b \\  c & d \end{pmatrix}: a,d \in\O_F, b\in(t_\mu)^{-1}\mathfrak{D}_F^{-1}, c\in\n t_\mu\mathfrak{D}_F,ad-bc\in\O_F^\times\right\},\] where $\D_F$ is the absolute different of $F$ and $\{t_\mu\}_{\mu=1}^h$ form a complete set of representatives of the narrow class group. 

We identify $\alpha\in\GL_2^+(F)$ with an element of $\GL_2^+(\R)^n$ via the various embeddings of $F$ in $\R$ and for $y=(y_1,...,y_n)\in\C^n$ and $k=(k_1,...,k_n)\in\Z^n$ we use the notation $y^k:=y_1^{k_1}...y_n^{k_n}$. In the case where $y_j=0$ for some $j=1,...,n$ we set $y^k=0$. For an element $a\in F_+$ we write $a^k:=\eta_1(a)^{k_1}...\eta_n(a)^{k_n}$, where $\eta_1,...,\eta_n$ are the embeddings of $F$ into $\C$.  

Then $f:\H^n\rightarrow\C$ is a Hilbert modular form of weight $k=(k_1,...,k_n)$ and level $\congsub$ if $f$ is holomorphic on $\H^n$ and at the cusps and we have \[f||_k\alpha(z):=(\det\alpha)^{k/2}(cz+d)^{-k}f(\alpha\cdot z)=f(z)\] for every $\alpha=\smatrix\in\Gamma_\mu(\n)$; since $\det\alpha$ is totally positive $(\det\alpha)^{k/2}$ is well defined. Any $f$ of the above form has a Fourier expansion of the form \[f(z)=\sum_{\xi\in t_\mu\O_F} a(\xi;f)e^{2\pi i\Tr(\xi z)},\] where $\Tr(\xi z)=\eta_1(\xi)z_1+...+\eta_n(\xi)z_n$. Following Shimura \cite[2.24]{Shi78}, we define for $f$ as above, $1\leq\mu\leq h$, \[c_\mu(\xi;f)=N(t_\mu\O_F)^{-k_0/2}a(\xi;f)\xi^{(k_0\mathbf{1}-k)/2},\] where $\mathbf{1}=(1,...,1)$ and $k_0=\max\{k_1,...,k_n\}$.

We say $f$ is a Hilbert cuspform if the constant term in the Fourier expansion of $f||_k\gamma$ is zero for every $\gamma\in\GL_2^+(F)$. A cuspidal Hilbert newform of weight $k$ and level $\n$ is a tuple $\f=(f_1,...,f_h)$ where $f_\mu$ is a Hilbert cuspform for $\Gamma_\mu(\n)$ and such that $\f$ does not come from a form of lower level and is a Hecke eigenform, see \cite[Section 2]{Shi78} for more details.  For any integral ideal $\mathfrak{m}$ in $F$,  there exists a unique $\mu\in\{1,...,h\}$ and a totally positive element in $F$ so that $\mathfrak{m}=\xi t_\mu^{-1}\O_F$.  Let $c(\mathfrak{m},\f)=a_\mu(\xi)\xi^{-k/2}$ where $a_\mu(\xi)$ are the Fourier coefficients of $f_\mu$,  note that this is well defined since the right hand side of this expression is invariant under the totally positive elements in $\O_F^\times$.  In the case that $\mathfrak{m}$ is not integral we set $c(\mathfrak{m},\f)=0$.  A normalised cuspidal Hilbert newform $\f$ (also called primitive form) is one such that $c(\O_F,\f)=1$.  A result of Shimura \cite[Proposition 2.8]{Shi78}, tells us that the set of all $\{c_\mu(\xi;f_\mu):1\leq\mu\leq h \ \text{and} \ \xi\in F_+\}$ generates a totally real or CM number field, denoted $\Q(\f)$, under the assumption that $k_1\equiv...\equiv k_n \pmod{2}$. We prove the following result.

\begin{IntroThm}[Theorem \ref{Thm}]\label{HMFThm}
Let $\f=(f_1,...,f_h)$ be a normalised Hilbert newform of weight $k=(k_1,...,k_n)$ with $k_1\equiv...\equiv k_n \pmod{2}$ and level $\n$. Let $1\leq\mu\leq h$ and let $\sigma=\smatrix\in\Gamma_\mu(1)$. Let $f_\mu||_k\sigma$ have the Fourier expansion \[f_\mu||_k\sigma(z)=\sum_{\xi}a_\mu(\xi;\sigma)e^{2\pi i \Tr(\xi z)},\] and define \[c_\mu(\xi;f_\mu||_k\sigma)=N(t_\mu\O_F)^{-k_0/2}a_\mu(\xi;\sigma)\xi^{(k_0\mathbf{1}-k)/2}.\] Let $\n'$ be the integral ideal of $\O_F$ such that $\n'(\n+cdt_\mu^{-1}\mathfrak{D}_F^{-1})=\n$. Then $c_\mu(\xi;f_\mu||_k\sigma)$ lie in the number field $\Q(\f)(\zeta_{N_0})$ where $N_0$ is the integer such that $N_0\Z=\n'\cap\Z$.
\end{IntroThm} 

To prove this theorem we study the classical action of $\Aut(\C)$ on the Fourier coefficients via the action of $\Aut(\C)$ on local Whittaker newforms. We give an explicit formula for the Fourier coefficients of a classical Hilbert automorphic form in terms of a global Whittaker function, Lemma \ref{general}. We can then apply this general result to the specific newform in the theorem, Proposition \ref{prop}. This generalises two results of \cite{ACAS}, namely Lemma 3.1 and Proposition 3.2, where the authors proved analogous results for modular forms. The global Whittaker newform can be broken up into a product of local Whittaker newforms. Therefore we can write the Fourier coefficients of $f_\mu||_k\sigma$ in terms of local Whittaker newforms. Precisely, in Proposition \ref{cmuprop}, we show \[c_\mu(\xi;f_\mu||_k\sigma)=N(t_\mu\O_F)^{-k_0/2}\xi^{k_0\mathbf{1}/2}\prod_{v<\infty}\Wv(a(\xi)\iotaf x_{\mu,v}),\] where $\Wv$ is the local Whittaker newform.  Thus studying the action of $\Aut(\C)$ on these Fourier coefficients is equivalent to studying the action of $\Aut(\C)$ on local Whittaker newforms.

In \cite[Lemma 10.5]{Shi00} Shimura provides a possible different approach to proving Theorem \ref{HMFThm}. Shimura proves a compatibility result between the Galois action and modular action for arbitrary (meromorphic, vector-valued) Siegel modular forms. The heart of our approach is Lemma \ref{key} which is a local result about local Whittaker newforms. Shimura's result is inherently global and relies on the reciprocity-law at CM-points and rationality of automorphic forms, while our method exploits the rational structure of the local Whittaker model. 

\section{Background setting and notation}

Here we discuss the background setting and collect the notation which will be used throughout this chapter.

\subsection{Background notation}\label{action}

Recall that $\GL_2^+(\R)$ acts on $\H$ via the action $\smatrix z=\frac{az+b}{cz+d}$. For a representation $\Pi$ we define the representation ${}^{\tau}\Pi$ as follows. Let $\tau$ be an automorphism of $\C$, then let $V$ be the space of $\Pi$ and let $V'$ be any vector space such that $t:V\rightarrow V'$ is a $\tau$-linear isomorphism. That is, \[t(v_1+v_2)=t(v_1)+t(v_2) \quad \text{and} \quad t(\lambda v)=\tau(\lambda)t(v).\] We define the representation $({}^{\tau}\Pi,V')$ via ${}^{\tau}\Pi(g)=t\circ\Pi(g)\circ t^{-1}$.

\subsection{Number fields}

Let $F$ be a totally real number field of degree $n$, with narrow class group of size $h$. We denote the ring of integers of $F$ by $\O_F$. The embeddings of $F$ in $\C$ we write as $\eta_1,...,\eta_n$. With respect to these embeddings we have that $F$ naturally embeds into $\R^n$. For an element $\alpha$ of $F$ write $(\alpha_1,...,\alpha_n)$ for $(\eta_1(\alpha),...,\eta_n(\alpha))$ as an element of $\R^n$. We write $\alpha^k=\prod_{j=1}^n\alpha_j^{k_j}$ for $k=(k_1,...,k_n)\in\Z^n$. For a subset $S$ of $F$, $S_+$ denotes the totally positive elements of $S$. Throughout $\n$ will be a fixed integral ideal of $F$. We denote the trace map of $F$ to $\Q$ by $\Tr$. Let $\D_F$ denote the absolute different of $F$, that is, $\D_F^{-1}=\{x\in F: \Tr(x\O_F)\subset\Z\}$. Let $F_v$ be the completion of $F$ at a non-archimedean place $v$. The ring of integers of $F_v$ are denoted by $\O_v$. We denote the maximal ideal of $\O_v$ by $\p_v$ and a generator of $\p_v$ by $\varpi_v$. We write $\n_v=\n\otimes_{\O_F}\O_v$ and $\mathfrak{D}_v=\mathfrak{D}_F\otimes_{\O_F}\O_v$ to be the $v$-part of $\n$ and $\D_F$ respectively. For an ideal $\n_v$ of $\O_v$ we define the subgroup $K_v(\n_v)$ of $\GL_2(F_v)$ as
\begin{equation}\label{Kv}
K_v(\n_v):=\left\{\begin{pmatrix} a & b\\ c& d \end{pmatrix}\in\GL_2(F_v): \begin{aligned} & a\O_v+\n_v=\O_v, \ b\in\mathfrak{D}_v^{-1} \\ & c\in\n_v\mathfrak{D}_v, \ d\in\O_v \end{aligned}, ad-bc\in\O_F^\times\right\}.
\end{equation}
By abuse of notation, given an ideal $\n$ of $\O_F$ use $K_v(\n)$ to denote $K_v(\n_v)$ where $\n_v=\n\otimes_{\O_F}\O_v$ is the $v$-part of $\n$. We define $n_v$ and $d_v$ via $\n_v=\varpi_v^{n_v}\O_v$ and $\D_v=\varpi_v^{d_v}\O_v$ respectively. 

\subsection{Ring of adeles of $F$}

Similar to the ring of adeles over $\Q$ we can define the ring of adeles of $F$. The ring of adeles of $F$ is denoted by $\A_F$. In the specific case that $F=\Q$ we denote $\A_\Q=\A$. The narrow class group can be viewed as \[F^\times\bs\A_F^\times/F_{\infty,+}^\times\prod\O_v^\times,\] similar to the result of Theorem \ref{SA}. We fix elements, for $1\leq\mu\leq h$, $t_\mu\in\A_F^\times$ such that \[
(t_\mu)_v
\begin{cases}
=1, \ \text{if}\  v\in\infty \ \text{or} \  v | \n \\
\in\O_v, \ \text{if} \ v\nmid\infty \  \text{and} \ v\nmid\n
\end{cases}
\] and $\{t_\mu\}_{\mu=1}^h$ are representatives of the narrow class group. We let $(t_\mu\O_F)$ denote the ideal of $\O_F$ corresponding to $t_\mu$. We have the following disjoint union decomposition \[\A_F^\times=\bigcup_{1\leq\mu\leq h} t_\mu F^\times F^\times_{\infty,+}\prod_{v<\infty}\O_v^\times.\] We denote $x_\mu=(\begin{smallmatrix}1 & \\ & t_\mu\end{smallmatrix})$ and for a finite place $v$ we denote $x_{\mu,v}=(\begin{smallmatrix}1 & \\ & (t_\mu)_v\end{smallmatrix})$. For $g\in\GL_2(\A_F)$ we define $\iota_{\text{f}}(g)$ to equal $g$ at all finite places and the identity at infinite places.

\subsection{Measures}

We normalise the measures so that Vol$(\O_v)=1$ and in a multiplicative setting Vol$(\O_v^\times)=1$. We also normalise so that Vol$(F\bs\AF)=1$.

\subsection{Additive characters}

Let $\psi_\Q:\Q\bs\A\rightarrow\C^\times$ be an additive character defined by $\psi_\Q=\prod_p\psi_{\Q,p}$ where $\psi_{\Q,\infty}(x)=e(x)$ for $x\in\R$ and $\psi_{\Q,p}(x)=1$ for $x\in\Z_p$. We then define an additive character $\psi$ on $F\bs\A_F$ by composing $\psi_\Q$ with the trace map from $F$ to $\Q$, that is, $\psi=\psi_\Q\circ\Tr:F\bs\A_F\rightarrow\C^\times$. If $\psi=\otimes\psi_v$ then the local characters are defined analogously and we have $\psi_\infty(x)=e(\Tr(x))$. Using \cite[Lemma 2.3.1]{RS} we have a formula for computing the value of the exponent of the character $\psi_v$ where $v$ is a non-archimedean place of $F$. For a place $v$ of $F$, letting $c(\psi_v)$ denote the smallest integer $c_v$ such that $\psi_v$ is trivial on $\p_v^{c_v}$, we have by \cite[Lemma 2.3.1]{RS}
\begin{equation}\label{formula}
c(\psi_v)=d_v,
\end{equation} 
where $d_v$ is defined as above, that is, via $\D_v=\varpi_v^{d_v}\O_v$.

\section{The Whittaker model}

Recall that each infinite dimensional admissible representation of $\GL_2$ over a local field admits a unique Whittaker model see \cite[Chapter 3]{Bump} and \cite[Chapter 4]{DGJH}. Furthermore, each cuspidal automorphic representation of $\GL_2$ admits a unique global Whittaker model. As we are now working over $F$ and not $\Q$ we give below the definition and basic properties of the Whittaker model for representation of trivial central character.

For any place $v$ of $F$ let $\Pi_v$ be an irreducible admissible infinite dimensional representation of $\GL_2(F_v)$. Let $\mathcal{W}(\psi_v)$ be the space of smooth functions $W_{\Pi_v}:\GL_2(F_v)\rightarrow\C$ satisfying \[W_{\Pi_v}\left(\begin{pmatrix} 1 & x \\ & 1 \end{pmatrix} g\right)=\psi_v(x)W_{\Pi_v}(g), \ \text{for} \ x\in F_v \ \text{and} \ g\in\GL_2(F_v).\] The space $\mathcal{W}(\psi_v)$ is a representation space of $\GL_2(F_v)$ with the action given by right-translation. The local Whittaker model, denoted $\mathcal{W}(\Pi_v,\psi_v)$, is the unique subrepresentation of $\mathcal{W}(\psi_v)$ that is equivalent to the representation $\Pi_v$.

Let $(\Pi,V_\Pi)$ be a cuspidal automorphic representation of $\GL_2(\A_F)$. Then the global Whittaker model, $\mathcal{W}(\Pi,\psi)$, for $\Pi$ with respect to a fixed non-trivial additive character $\psi$, consists of the space generated by functions $W_\phi$ on $\GL_2(\AF)$ given by 
\begin{equation}\label{Def}
W_\phi(g):=\int_{\A_F/F}\phi\left(\begin{pmatrix} 1 & x \\ & 1 \end{pmatrix}g\right)\overline{\psi(x)}\ dx,
\end{equation}
as $\phi$ varies in $V_\Pi$. The representation $\mathcal{W}(\Pi,\psi)$ decomposes as a restricted tensor product of local Whittaker models, $\mathcal{W}(\Pi_v,\psi_v)$.

For $\Pi_v$ an irreducible admissible infinite dimensional representation of $\GL_2(F_v)$, with unramified central character, let $c(\Pi_v)$ be the smallest integer $n$ such that $\Pi_v$ has a $K_v(\p_v^n)$-fixed vector. Then the normalised Whittaker newform (with respect to $\psi_v$) is the unique $K_v(\p_v^{c(\Pi_v)})$-invariant vector $\Wv\in\mathcal{W}(\Pi_v,\psi_v)$ satisfying $\Wv(1)=1$.

We now state a key result which will be used throughout this chapter.

\begin{Lemma}\label{key}
Let $\Pi_v$ be an irreducible admissible infinite dimensional representation of $\GL_2(F_v)$ with unramified central character. Let $K_v(\p_v^{c(\Pi_v)})$ be defined as in \eqref{Kv}. Let $\Wv$ be the local normalised Whittaker newform. Let $\tau\in\Aut(\C)$ and $g\in\GL_2(F_v)$. Then we have $\tau(\Wv(g))=W_{{}^{\tau}\Pi_v}(a(\alpha_\tau)g)$ where $\alpha_\tau$ is defined via 
\begin{align*}
\Aut(&\C/\Q) &&\rightarrow && \Gal(\overline{\Q}/\Q) &&\rightarrow &&  \Gal(\Q(\mu_\infty)/\Q)&&\rightarrow  &&\hat{\Z}^\times\cong\Pi_v\Z_v^\times &&\subset && \Pi_v\Pi_{\p|v}\O_\p^\times\\
&\tau &&\mapsto &&  \qquad\tau|_{\overline{\Q}} && \mapsto  &&\qquad\tau|_{\Q(\mu_\infty)} && \mapsto  &&\qquad\alpha_\tau &&\mapsto &&\alpha_\tau=(\alpha_{\tau,\p})_\p.
\end{align*}
\end{Lemma}
\begin{proof}
See \cite[Section 2]{ACAS} and \cite[Section 3.2.3]{HMF}.
\end{proof}

Let $\Pi_v$ be an irreducible admissible infinite dimensional representation of $\GL_2(F_v)$ with unramified central character and normalised Whittaker newform $\Wv$. Then consider the representation $\Pi_v\otimes| \ |_v^{k_0/2}$, for some integer $k_0$. We now apply the above result to this representation. Therefore using Lemma \ref{key} we have \[\tau\big(W_{\Pi_v\otimes| \ |_v^{k_0/2}}(g)\big)=W_{{}^{\tau}(\Pi_v\otimes| \ |_v^{k_0/2})}(a(\alpha_\tau)g).\] This is equivalent to \[\tau\big(\Wv(g)\big)\tau\big(|\det(g)|^{k_0/2}_v\big)=W_{{}^{\tau}(\Pi_v\otimes| \ |_v^{k_0/2})}(a(\alpha_\tau)g).\] So if $\Pi_v\otimes| \ |_v^{k_0/2}\cong{}^{\tau}(\Pi_v\otimes| \ |_v^{k_0/2})$, we have that 
\begin{equation}\label{action}
\tau\big(\Wv(g)\big)\tau\big(|\det(g)|^{k_0/2}_v\big)=\Wv(a(\alpha_\tau)g)|\det(g)|^{k_0/2}_v.
\end{equation}

\section{Modular forms case}\label{MFcase}

In this section, we study the case of modular forms. The reason we do this is because the proof of the theorem stated below follows a similar method to the one used to prove Theorem \ref{HMFThm} but is less technical. As a consequence of this we are able to prove Thereom \ref{mf2} below, which is an alternative version of \cite[Theorem 4.1]{FEatC} whereby we make some simplifying assumptions on the matrix $\sigma\in\SL_2(\Z)$. The general form of \cite[Theorem 4.1]{FEatC} will follow from applying Theorem \ref{HMFThm} in the setting of classical newforms. 

\subsection{Fourier expansion}

Here we briefly recap the discussion of Section \ref{Class}. For a function $f:\H\rightarrow\C$ we define a function  $f|_kg$ on $\H$ for an integer $k$ and some $g=\smatrix\in\GL_2^+(\R)$ as \[f|_kg(z)=(\det g)^{k/2}(cz+d)^{-k}f(g\cdot z).\] Let $N\geq1$. Let $\Gamma$ be a congruence subgroup of $\SL_2(\Z)$. Let $f$ be a modular form on $\Gamma$ of weight $k$. Then we have that $f$ has a Fourier expansion at infinity of the form \[f(z)=\sum_{n\geq0}a_f(n)e^{2\pi i nz/w},\] where $a_f(n)$ are the Fourier coefficients of $f$ and $w$ is the smallest integer such that $(\begin{smallmatrix} 1 & w \\ 0 & 1 \end{smallmatrix})\in\Gamma$. Note that in the case that $\Gamma=\Gamma_0(N)$ we have that $w=1$. If the first Fourier coefficient of $f|_kg$ is zero for every $g\in\SL_2(\Z)$ then we say that $f$ is a cuspform. If we also have that $a_f(1)=1$ we say that $f$ is normalised. Recall that if $f$ is a modular form for $\Gamma_0(N)$ which is non-zero we have that $k$ is even.

Let $\c$ be a cusp of $\Gamma_0(N)\bs\H$. Then $\c$ can be represented by the point $\frac{a}{L}\in\Q$, where $(a,N)=1$ and $L|N$. This is equivalent to $\sigma\infty$, where $\sigma=(\begin{smallmatrix} a & b \\ L & d \end{smallmatrix})\in\SL_2(\Z)$,  with $L|N$ and $(a,N)=1$. If $f$ is a modular form for $\Gamma_0(N)$ then $f|_k\sigma$ is a modular form for the congruence subgroup $\sigma^{-1}\Gamma_0(N)\sigma$. We have that $f|_k\sigma$ has a Fourier expansion of the form \[f|_k\sigma(z)=\sum_{n>0}a_f(n;\sigma)e^{2\pi inz/w(\sigma)},\] where $a_f(n;\sigma)$ are the Fourier coefficients of $f|_k\sigma$ and $w(\sigma)$ is the width of the cusp. The width of the cusp is defined to be the integer $N/(L^2,N)$ where $L$ is the denominator of the cusp $\c$.

\subsection{Result for modular forms}

\begin{Thm}\label{mf2}
Let $f$ be a normalised newform on $\Gamma_0(N)$ of weight $k$. Let $\Q(f)$ be the field generated by all the Fourier coefficients of $f$. Let $\c$ be a cusp of $\Gamma_0(N)\bs\H$ that is equivalent to $\sigma\infty$ with $\sigma=(\begin{smallmatrix} a & b \\ L & d \end{smallmatrix})\in\sl2Z, \ L|N, \ (a,N)=(d,N)=1$, so that $\a$ is represented by the point $a/L$. Then the Fourier coefficients of $f|_k \sigma$ lie in the cyclotomic extension $\Q(f)(\zeta_{N/L})$.
\end{Thm}

\begin{Remark}
Recall that to show a complex number $\rho\in\C$ is an element of a number field it is equivalent to showing that $\rho$ is fixed by the Galois group of that number field.
\end{Remark}

\begin{proof}[Proof of \Cref{mf2}]
Let $\tau\in\Aut(\C)$ fix $\Q(f)(\zeta_{N/L})$. So in particular $\tau$ fixes $\Q(f)$ and all the $\zeta_{N/L}$ roots of unity. 

Let $n=n_0\prod_{p|N}p^{n_p}$ with $(n_0,N)=1$. In the proof of Proposition 3.3 in \cite{ACAS} the authors give an explicit formula for $a_f(n;\sigma)$, given by \[a_f(n;\sigma)=\frac{a_f(n_0)}{n_0^{k/2}}\Big(\frac{n}{\delta(\c)}\Big)^{k/2}\prod_{p|N}\Wp (a(n/\delta(\c))\sigma^{-1}),\] where $\delta(\c)$ is a rational number related to the cusp $\c$ and $\Wp$ are the local Whittaker newforms associated to $f$. Note that ${n}/{\delta(\c)}\in\Q$ and since $k$ is even we have that the second term is automatically fixed by $\tau$. Since $\tau$ fixes $\Q(f)$ we have that $\tau$ fixes $a_f(n_0)$. 

Since $\tau$ fixes $\zeta_{N/L}$ we have that $\alpha_\tau\equiv1 \pmod{p^{n_p-l_p}}$. This implies \[\sigma\begin{pmatrix} \alpha_\tau & 0 \\ 0 & 1\end{pmatrix}\sigma^{-1}=\begin{pmatrix}ad\alpha_\tau-bc & ab(1-\alpha_\tau) \\ dL(\alpha_\tau-1) & ad-bL\alpha_\tau\end{pmatrix}\in K_p(p^{n_p}).\] Since $\tau$ fixes $\Q(f)$ we have that $\pi_p={}^{\tau}\pi_p$. Therefore using Lemma \ref{key} we have that \[\tau(\Wp(a(n/\delta(\c))\sigma^{-1}))=\Wp(a(\alpha_\tau)a(n/\delta(\c))\sigma^{-1}).\] Recall that $\Wp$ is right invariant by $K_p(p^{n_p})$. Thus from the condition that $\sigma a(\alpha_\tau)\sigma^{-1}\in K_p(p^{n_p})$ we have that  \[\Wp(a(n/\delta(\c))\sigma^{-1})=\Wp(a(\alpha_\tau)a(n/\delta(\c))\sigma^{-1}).\] Thus $\tau$ fixes $\af$ and so $\af\in\Q(f)(\zeta_{N/L})$.
\end{proof}

\section{Preliminary results}\label{Prelim}

In this section, we define the classical Hilbert newform and how we can reinterpret this definition in the adelic setting. This is similar to the classical case discussed in Section \ref{ctoa}. We then prove some general results needed to prove Theorem \ref{HMFThm}. Specifically, we show that the Fourier coefficients of Hilbert newform equals the global Whittaker newform. 

\subsection{Hilbert modular forms}\label{Setup}

We define \[\H^n=\{z=(z_1,...,z_n)\in\C^n:\Ima(z_j)>0, \forall j=1,...,n\},\] where we write $z=x+iy$ to mean for $z=(z_1,...,z_n)$ and $z_j=x_j+iy_j$ for each $j$. Recall from Section \ref{hilbert} that $\GL_2^+(F)$ acts on $\H^n$ in the following way \[g\cdot z=\Bigg(\frac{\eta_1(a_1)z_1+\eta_1(b_1)}{\eta_1(c_1)z_1+\eta_1(d_1)},...,\frac{\eta_n(a_n)z_n+\eta_n(b_n)}{\eta_n(c_n)z_n+\eta_n(d_n)}\Bigg).\] We define the congruence subgroup $\congsub$ of $\GL_2^+(F)$ for each $\mu$ as 
\begin{equation}\label{congsub}
\congsub=x_\mu(\GL_2^+(F_\infty)K_0(\n))x_\mu^{-1}\cap\GL_2(F),
\end{equation}
where 
\begin{equation}\label{K0}
K_0(\n)=\prod_{v<\infty}K_v(\n).
\end{equation}

Then $f_\mu:\H^n\rightarrow\C$ is a Hilbert modular form of level $\congsub$ and weight $k=(k_1,...,k_n)$ if $f_\mu$ is holomorphic on $\H^n$ and at cusps and \[f_\mu||_k\alpha(z):=(\det\alpha)^{k/2}(cz+d)^{-k}f_\mu(\alpha\cdot z)=f_\mu(z),\] for every $\alpha\in\congsub$. 

\begin{Remark}
One can also consider this definition with a character $\chi$ on $\A_F^\times/F^\times$. We will assume that this character is trivial. For a treatment of $\chi$ being non-trivial see \cite[Section 4]{HMF} and the references therein.  
\end{Remark}

If the constant term in the Fourier expansion of $f_\mu||_k\gamma$ is zero for every $\gamma\in\GL_2^+(F)$ then we say $f_\mu$ is a cuspform. For the space of cuspforms to be nonempty we require that $k_j\geq1$, see \cite[Section 1.7]{HolHMF} for a proof of this fact. In this case $f_\mu$ has a Fourier expansion at infinity of the form \[f_\mu(z)=\sum_{\xi\in(t_\mu\O_F)_+}a_\mu(\xi)e^{2\pi i \text{Tr}(\xi z)}.\] The space of cuspforms of level $\congsub$ and weight $k$ is denoted $S_k(\congsub)$. For each $\mu$ choose a cuspform in $S_k(\congsub)$ and put $\f=(f_1,...,f_h)$.  Then we say that $\f$ is a cuspform of level $\n$ and weight $k=(k_1,...,k_n)$. The space of all such $\f$ is denoted by $S_k(\n)$ and we have \[S_k(\n)=\bigoplus_{\mu=1}^h S_k(\congsub).\] An element $\f\in S_k(\n)$ is said to be a cuspidal classical Hilbert newform if it lies in the orthogonal complement of the oldspace and is a common eigenfunction for the Hecke operators at almost all primes, see \cite[Section 2]{Shi78}.

\begin{Remark}
For a more detailed basic introduction to Hilbert modular forms, one can see, for example, \cite[Chapter 2]{123}.
\end{Remark}

\subsection{Classical and adelic correspondence}\label{dictionary}

Here we give a description on how one can go between classical Hilbert newforms and adelic newforms (as in Section \ref{Setup} we only consider the case of trivial character $\chi$). We will give an outline of how this description is constructed for a full discussion see \cite[Chapter 3]{HolHMF} and \cite[Section 4]{HMF}, it will be similar to the correspondence given in Section \ref{gobetween}. Let $A_k(\n)$ be the subspace of $A_{\text{cusp}}(\GL_2(F)\bs\GL_2(\A_F))$ that consists of functions such that
\begin{enumerate}
\item[(i)] $\phi(gr(\theta))=e^{-ik\theta}\phi(g)$, where $r(\theta)=\left\{\begin{pmatrix} \cos\theta_j & -\sin\theta_j \\ \sin\theta_j & \cos\theta_j \end{pmatrix}\right\}_j\in \text{SO}(2)^n$,
\item[(ii)] $\phi(g\kappa)=\phi(g)$, where $\kappa\in K_0(\n)$,
\item[(iii)] $\phi$ is an eigenfunction of the Casimir element $\Delta:=(\Delta_1,...,\Delta_n)$ as a function of $\GL_2(\R)^n$ with its eigenvalue $\lambda=\prod_{j=1}^n\frac{k_j}{2}(1-\frac{k_j}{2})$.
\end{enumerate}
Then $A_k(\n)$ is isomorphic to $S_k(\n)$. To see this we define, for each $\f=(f_1,...,f_h)\in S_k(\n)$
\begin{equation}\label{CtoA}
\phi(\gamma x_\mu g_\infty \kappa)=f_\mu||_kg_\infty(\mathbf{i})
\end{equation}
where $\gamma\in\GL_2(F), \  g_\infty\in\GL_2^+(F_\infty), \ \kappa\in K_0(\n)$ and $\mathbf{i}=(i,...,i)$. This is well defined because of (\ref{congsub}) and (\ref{K0}). One can also check that $\phi$ as above gives an element of $A_k(\n)$. Conversely, given $\phi\in A_k(\n)$ define \[f_\mu(z)=y^{-k/2}\phi(x_\mu g_z),\] where $g_z\mathbf{i}=z$. Then we have that $\f:=(f_1,...,f_h)\in S_k(\n)$.

We say that $\phi\in A_k(\n)$ is an adelic newform if it generates an irreducible cuspidal automorphic representation $\Pi=\Pi_\phi$ of $\GL_2(\AF)$ of conductor $\n$. The set of adelic newforms in $A_k(\n)$ is in bijective correspondence with the set of classical cuspidal Hilbert newforms $\f\in S_k(\n)$. Precisely given $\phi$ as above it corresponds to a classical cuspidal Hilbert newform $\f=(f_1,...,f_h)$ of level $\n$ and weight $k$ defined via $f_\mu(z)=y^{-k/2}\phi(x_\mu g_z)$.

Given an adelic newform $\phi_\f$ attached to $\f\in S_k(\n)$, the automorphic representation of $\GL_2(\AF)$ generated by it is known to be irreducible, see for example \cite[Theorem 4.7]{HMF}. The following lemma gives an explicit relation between the action of $\Aut(\C)$ on newforms and representations. 

\begin{Lemma}\label{relation}
Let $\f=(f_1,...,f_h)$ be a cuspidal classical Hilbert newform. Given $\tau\in\Aut(\C)$ let $\f^\tau$ be defined as in \cite[Proposition 2.6]{Shi78}. Let $\Pi$ be the automorphic representation generated by $\phi_\f$ and $\Pi'$ be the automorphic representation generated by $\phi_{\f^\tau}$. Then for each finite place $v$ we have \[\Pi'_v\otimes|\ |_v^{k_0/2}\cong{}^{\tau}(\Pi_v\otimes|\ |_v^{k_0/2}).\]
\end{Lemma}

\begin{proof}
See \cite[Theorem 4.19]{HMF}.
\end{proof}

\subsection{Whittaker expansion}

We now give the explicit relation between the Fourier coefficients of a classical Hilbert newform at a cusp and the global Whittaker newform. We start with a general result which holds for a larger class of Hilbert automorphic forms.

\begin{Lemma}\label{general}
Let $M$ be an ideal of $\O_F$ with localisation $M_v=\varpi_v^{m_v}\O_v$ for every finite place $v$ such that for almost all $v$ we have that $m_v=0$.  Let $\phi$ be a Hilbert automorphic form in the space of $\Pi$ such that $\phi$ is right invariant by $\big(\begin{smallmatrix} 1 & \varpi_v^{m_v}\mathcal{O}_v \\0 & 1\end{smallmatrix}\big)$ for every finite place and satisfies (i) and (iii) in Section \ref{dictionary}. Let $h$ be a holomorphic function on $\mathbb{H}^n$ defined by the equation \[h(z)=j(g_z,\mathbf{i})^k\phi(g_z).\] Then $h$ has a Fourier expansion of the form 
\begin{equation}\label{FE}
h(z)=\sum_{\xi\in M^{-1}\mathfrak{D}_F^{-1}}a_h(\xi)e^{2\pi i \Tr(\xi z)}.
\end{equation}
 Moreover \[W_\phi(a(\xi)g_z)=
\begin{cases}
y^{k/2}a_h(\xi)e^{2\pi i\text{Tr}(\xi z)}, \ \text{if} \ \xi\in M^{-1}\mathfrak{D}_F^{-1},\\
0, \ \text{otherwise}.
\end{cases}\] 
\end{Lemma}

\begin{proof}
From the fact that $\phi$ satisfies (i) and (iii) in Section \ref{dictionary}, it follows that $h$ is holomorphic on $\H^n$ and $h(z+\gamma)=h(z)$, for every $\gamma\in M$. This shows that \[h(z)=\sum_{\xi\in F} a_h(\xi)e^{2\pi i \Tr(\xi z)},\] for some complex numbers $a_h(\xi)$.

Suppose $a_h(\xi)\neq0$. From the property that $h$ is invariant by $M$ we have that $\Tr(\xi \gamma)\in\Z$ for every $\gamma\in M$. This property is equivalent to that of $\xi\in M^{-1}\mathfrak{D}_F^{-1}$. To see this note that for every place $v<\infty$ we have that $\Tr(\xi_v\gamma_v)\in\Z_v$, for every $\gamma_v\in M_v$. Thus $\Tr((\varpi_v^{m_v}\xi_v)n_v)\in\Z_v$ for every $n_v\in\O_v$. This is equivalent to the following \[\varpi_v^{m_v}\xi_v\in\mathfrak{D}_v^{-1}\Longleftrightarrow \xi_v\in\varpi_v^{-m_v}\mathfrak{D}_v^{-1}\Longleftrightarrow\xi\in M^{-1}\mathfrak{D}_F^{-1}.\] This shows that $h$ has a Fourier expansion of the form (\ref{FE}). Next let $\xi\in F^\times$ and let $z=x+iy\in\H^n$. Then using the definition of $W_\phi$, (\ref{Def}), we have that 
\begin{align*} 
W_\phi(a(\xi)g_z)&=\int_{\AF/F}\phi\big(n(x')a(\xi)g_z\big)\overline{\psi(x')} \ dx'\\
&=\int_{\AF/F} \phi\big(a(\xi)n(\xi^{-1}x')g_z)\psi(-x') \ dx'\\
&=\int_{\AF/F} \phi\big(n(\xi^{-1}x')g_z)\psi(-x') \ dx'\\
&=\int_{\AF/F} \phi\big(n(x')g_z)\psi(-\xi x') \ dx'\\
&=\int_{\AF/F} \phi(n(x'+x)a(y))\psi(-\xi x') \ dx'\\
&=\int_{\AF/F} \phi(n(x')a(y))\psi(-\xi x')\psi(\xi x) \ dx'\\
&=e(\Tr(\xi x))\int_{\AF/F} \phi\big(n(x')a(y))\psi(-\xi x') \ dx'.
\end{align*}
We denote the integral in the last line by $I_\phi(y,\xi)$. We now want to split this integral into the product of two integrals, using the fundamental domain \[\AF/F=F_\infty/M\times\prod_{v<\infty}\varpi_v^{m_v}\O_v,\] using strong approximation. Therefore using $y\in F_\infty$ and the right invariance of $\phi$ by $n(\varpi_v^{m_v}\O_v)$ we have that 
\begin{align*}
\text{Vol}(F_\infty/M)I_\phi(y,\xi)&=\int_{F_\infty/M}\phi(n(x_\infty)a(y))\psi_\infty(-\xi x_\infty) \ dx_\infty\prod_{v<\infty}\int_{\O_v} \psi_v(-\xi \varpi_v^{m_v}x_v) \ dx_v,\\
&=\int_{F_\infty/M}y^{k/2}h(x_\infty+iy)e(-\Tr(\xi x_\infty)) \ dx_\infty\prod_{v<\infty}\int_{\O_v} \psi_v(-\xi\varpi_v^{m_v}x_v) \ dx_v.
\end{align*}
We have that $I_\phi(y,\xi)$ will only be non-zero when \[\int_{\O_v}\psi_v(-\xi\varpi_v^{m_v}x_v)\ dx_v\neq0\] for each $v$. Recall that $\psi_v=\psi_p\circ\Tr_{F_v/\Qp}$ and $\psi_p(x)=1$ for every $x\in\Z_p$. If the finite place integral is non-zero for each $v$ we write $\xi=\varpi_v^nu$ where $n\in\Z$ and $u\in\O_v^\times$. Therefore using (\ref{formula}) we have that \[\int_{\O_v}\psi_v(\xi\varpi_v^{m_v}x_v) \ dx_v=\int_{\O_v}\psi_v(\varpi_v^{m_v+n}uk) \ dk=
\begin{cases}
0, \ \text{if} \ m_v+n<d_v,\\
1, \ \text{if} \ m_v+n\geq d_v,
\end{cases}\] 
Therefore we have that \[I_\phi(y,\xi)=
\begin{cases}
\frac{1}{\text{Vol}(F_\infty/M)}\int_{F_\infty/M}y^{k/2}h(x_\infty+iy)e(-\Tr(\xi x_\infty)) \ dx_\infty, \ \text{if} \ \xi\in M^{-1}\mathfrak{D}_F^{-1},\\
0, \ \text{otherwise}.
\end{cases}\] Using the Fourier expansion of $h$, \[h(x_\infty+iy)=\sum_{\xi'\in M^{-1}\mathfrak{D}_F^{-1}}a_h(\xi')e^{2\pi i \Tr(\xi'x_\infty)}e^{-2\pi \Tr(\xi'y)},\] we have that 
\[I_\phi(y,\xi)=
\begin{cases}
y^{k/2}a_h(\xi)e^{-2\pi i\Tr(\xi y)}, \ \text{if} \ \xi\in M^{-1}\mathfrak{D}_F^{-1},\\
0, \ \text{otherwise}.
\end{cases}\] 
Therefore for $\xi\in M^{-1}\mathfrak{D}_F^{-1}$ we have \[W_\phi(a(\xi)g_z)=e(\Tr(\xi x))y^{k/2}a_h(\xi)e(i\Tr(\xi y))=y^{k/2}a_h(\xi)e(\Tr(\xi z)),\] and for $\xi\notin M^{-1}\mathfrak{D}_F^{-1}$ we have $W_\phi(a(\xi)g_z)=0$.
\end{proof}

From now on we let $\phi$ be an adelic newform in the space of $\Pi$ and $\f$ be the corresponding classical Hilbert newform. We shall now also assume that $k_j\geq1$ for every $j=1,...,n$. Let $\sigma=\smatrix\in\Gamma_\mu(1)$, then recall that if $f_\mu$ is a Hilbert modular form for $\congsub$ we have $f_\mu||_k\sigma(z)$ is a Hilbert modular form for $\sigma^{-1}\Gamma_\mu(\n)\sigma$. We have for $\sigma\in\Gamma_\mu(1)$, using \eqref{CtoA}
\begin{align*}
f_\mu||_k\sigma(z)&=y^{-k/2}(f_\mu||_k\sigma g_z)(\mathbf{i})\\
&=y^{-k/2}\phi(x_\mu\iota_\infty(\sigma)g_z)\\
&=y^{-k/2}\phi(\iota_\infty(\sigma)x_\mu g_z)\\
&=y^{-k/2}\phi(\iotaf x_\mu g_z)\\
&=y^{-k/2}\phi(g_z\iotaf x_\mu),
\end{align*}
where $\iotaf x_\mu$ is an element of the finite adeles.

Our aim is to apply Lemma \ref{general} to the specific case of $f_\mu||_k\sigma$ where $\sigma\in\Gamma_\mu(1)$. By the above calculation, this corresponds to $(\iotaf x_\mu)\phi$ in the setup of Lemma \ref{general}. The first step to do this is by finding the ideal $M'$ of $\O_F$ corresponding to $M_v'=\varpi_v^{m'_v}\O_v$ such that $\big(\begin{smallmatrix}1 & M'_v \\ & 1\end{smallmatrix}\big)\in\sigma^{-1}\Gamma_\mu(\n)\sigma$. That is we require $\sigma\big(\begin{smallmatrix}1 & M'_v \\ & 1\end{smallmatrix}\big)\sigma^{-1}\in\congsub$. Equivalently for each place $v$, we require \[\begin{pmatrix}1 & \varpi_v^{m'_v}\O_v \\ & 1\end{pmatrix}\in\sigma^{-1}x_{\mu,v} K_v(\varpi_v^{n_v})x_{\mu,v}^{-1}\sigma.\] We can view 
\begin{equation}\label{sigma}
\sigma=x_{\mu,v}(\begin{smallmatrix}1 & \\ & \varpi_v^{d_v}\end{smallmatrix})(\begin{smallmatrix}a' & b'\\ c' & d' \end{smallmatrix}) (\begin{smallmatrix}1 & \\ & \varpi_v^{-d_v}\end{smallmatrix})x_{\mu,v}^{-1},
\end{equation}
where $(\begin{smallmatrix}a' & b'\\ c' & d'\end{smallmatrix})\in\GL_2(\O_F)$ and $x_{\mu,v}=(\begin{smallmatrix}1 & \\ & \varpi_v^{t_v}\end{smallmatrix})$. Denote $A:=(\begin{smallmatrix}a' & b'\\ c' & d'\end{smallmatrix})\in\GL_2(\O_F)$. Therefore 
\begin{align*}
&\begin{pmatrix} 1 & \varpi_v^{m'_v}\O_v \\ & 1\end{pmatrix} \in x_{\mu,v}\begin{pmatrix}1 & \\ & \varpi_v^{d_v}\end{pmatrix}A^{-1}\begin{pmatrix}1 & \\ & \varpi_v^{-d_v}\end{pmatrix} K_v(\varpi_v^{n_v})\begin{pmatrix}1 & \\ & \varpi_v^{d_v}\end{pmatrix}A\begin{pmatrix}1 & \\ & \varpi_v^{-d_v}\end{pmatrix}x_{\mu,v}^{-1}\\
&\Longleftrightarrow \begin{pmatrix} 1 & \varpi_v^{m'_v+t_v+d_v}\O_v \\ & 1\end{pmatrix}\in A^{-1}\begin{pmatrix}1 & \\ & \varpi_v^{-d_v}\end{pmatrix} K_v(\varpi_v^{n_v})\begin{pmatrix}1 & \\ & \varpi_v^{d_v}\end{pmatrix}A\\
&\Longleftrightarrow A\begin{pmatrix} 1 & \varpi_v^{m'_v+t_v+d_v}\O_v \\ & 1\end{pmatrix}A^{-1}\in \begin{pmatrix}1 & \\ & \varpi_v^{-d_v}\end{pmatrix}K_v(\varpi_v^{n_v})\begin{pmatrix}1 & \\ & \varpi_v^{d_v}\end{pmatrix}.\stepcounter{equation}\tag{\theequation}\label{condition}
\end{align*}
On computing the left hand side of (\ref{condition}) we require \[\begin{pmatrix} \eps^{-1}(a'd'-b'c'-a'c'\varpi_v^{m_v'+t_v+d_v}\alpha) & \eps^{-1}(a')^2\varpi_v^{m_v'+t_v+d_v}\alpha \\ -\eps^{-1}(c')^2\varpi_v^{m_v'+t_v+d_v}\alpha & \eps^{-1}(a'd'-b'c'+a'c'\varpi_v^{m_v'+t_v+d_v}\alpha) \end{pmatrix} \in y_v^{-1}K_v(\varpi_v^{n_v})y_v,\] where $\alpha\in\O_v, \ y_v=(\begin{smallmatrix}1 & \\ & \varpi_v^{d_v}\end{smallmatrix})$ and $\eps=\det A\in\O_F^\times$. If we now study the $\p$-valuation we obtain the following conditions
\begin{enumerate}
\item $2\val_\p(a')+m_v'+t_v\geq 0$,
\item $2\val_\p(c')+m_v'+t_v\geq n_v$.
\end{enumerate}
On rearranging these conditions we obtain \[m_v'\geq-d_v-t_v-2\val_\p(a'), \ m_v'\geq -t_v-d_v+n_v-2\val_\p(c').\] Hence \[m_v'\geq -t_v-d_v+\max\{-2\val_\p(a'), n_v-2\val_\p(c')\}.\] Let \[w_v(\sigma,\n)=\max\{-2\val_\p(a'), n_v-2\val_\p(c')\}=n_v-\min\{2\val_\p(a')+n_v,2\val_\p(c')\}.\] Then let $\mathfrak{w}_v(\sigma,\n)=\varpi_v^{w_v(\sigma,\n)}\O_v$ and $\mathfrak{w}(\sigma,\n)$ the corresponding ideal in $\O_F$, so \[\mathfrak{w}(\sigma,\n)=\n(\gcd((a')^2\n,(c')^2))^{-1}.\] This leads to $f_\mu||_k\sigma$ having a Fourier expansion of the form \[f_\mu||_k\sigma(z)=\sum_{\xi\in((t_\mu\O_F)\mathfrak{w}(\sigma,\n)^{-1})_+} a_\mu(\xi;\sigma)e(\Tr(\xi z)).\] We define 
\begin{equation}\label{FC}
c_\mu(\xi;f_\mu||_k\sigma)=N(t_\mu\O_F)^{-k_0/2}a_\mu(\xi;\sigma)\xi^{(k_0\mathbf{1}-k)/2},
\end{equation}
recall $k_0=\max\{k_1,...,k_n\}$ and $k_0\mathbf{1}=(k_0,...,k_0)$.

\begin{Remark}\label{dash}
Note from \eqref{sigma} we can view a matrix $\sigma=\smatrix\in\Gamma_\mu(1)$ as \[x_{\mu,v}(\begin{smallmatrix}1 & \\ & \varpi_v^{d_v}\end{smallmatrix})(\begin{smallmatrix}a' & b'\\ c' & d' \end{smallmatrix}) (\begin{smallmatrix}1 & \\ & \varpi_v^{-d_v}\end{smallmatrix})x_{\mu,v}^{-1},\] with $(\begin{smallmatrix}a' & b'\\ c' & d' \end{smallmatrix})\in\GL_2(\O_F)$. Therefore we have that \[\sigma_v=\begin{pmatrix} a_v & b_v \\ c_v & d_v \end{pmatrix}=\begin{pmatrix} a'_v & b'_v \\ \varpi_v^{t_v+d_v} c'_v & d'_v \end{pmatrix}.\]
\end{Remark}

\begin{Prop}\label{prop}
Let $\mathfrak{w}(\sigma,\n)$ be the ideal of $\O_F$ defined by \[\mathfrak{w}(\sigma,\n)=\n(\mathrm{gcd}((a')^2\n,(c')^2))^{-1}.\] Let $\xi\in F^\times$. Then
\[W_{\phi}(a(\xi)g_z\iota_{\emph{f}}(\sigma^{-1}) x_\mu)=
\begin{cases}
y^{k/2}a_\mu(\xi;\sigma)e(\Tr(\xi z)), \ \text{if} \ \xi\in \big((t_\mu\O_F)\mathfrak{w}(\sigma,\n)^{-1}\big)_+  \\
0, \ \text{otherwise}.
\end{cases}\] 
\end{Prop}

\begin{proof}
This follows from applying Lemma \ref{general} to the automorphic form \[\phi'=\Pi(\iotaf x_\mu)\phi\] and using the fact $W_{\phi}(a(\xi)g_z\iotaf x_\mu)=W_{\phi'}(a(\xi)g_z)$.
\end{proof}

\section{Main result}\label{MR}

In this section, we prove our main result of this chapter. We do this by applying a similar method to that used in the proof of Theorem \ref{mf2}.

In Proposition \ref{prop} we showed that $a_\mu(\xi;\sigma)$ can be written in terms of the global Whittaker function. Recall that the global Whittaker function can be broken up into a product of local Whittaker functions. This means that finding sufficient conditions for $\tau$ to fix $c_\mu(\xi;f_\mu||_k\sigma)$, boils down to finding sufficient conditions such that $\tau$ fixes the local Whittaker newforms. With this in mind, we have the following decomposition 
\begin{equation}\label{decom}
W_\phi(g)=C\prod_{v|\infty}W_\infty(g_\infty)\prod_{v<\infty}\Wv(g_v),
\end{equation}
where $C$ is a constant. This decomposition is due to the uniqueness of Whittaker functionals. We are able to evaluate $W_\infty$ via, 
\begin{equation}\label{Winfty}
W_\infty(a(\xi)g_z)=(\xi y)^{k/2}e(\Tr(\xi z)),
\end{equation}
see for example \cite[Section 3]{Sah16}. We now wish to compute the constant $C$. To do this let $g_\infty=g_z$ and $g_v=x_{1,v}$ for every $v<\infty$. Then using Proposition \ref{prop} and the fact that $\f$ is normalised, that is, $a_\mu(1;1)=1$ we have \[W_\phi(g_zx_1)=y^{k/2}a_1(1;1)e(\Tr(z))=y^{k/2}e(\Tr(z)).\] On the other hand using (\ref{decom}), (\ref{Winfty}) and the fact that $W_v(1)=1$ we obtain \[W_\phi(g_zx_1)=Cy^{k/2}e(\Tr(z)).\] Thus $C=1$.

\begin{Prop}\label{cmuprop}
Let $c_\mu(\xi;f_\mu||_k\sigma)$ be defined as in \eqref{FC}. Then  
\begin{equation}\label{cmu}
c_\mu(\xi;f_\mu||_k\sigma)=N(t_\mu\O_F)^{-k_0/2}\xi^{k_0\mathbf{1}/2}\prod_{v<\infty}\Wv(a(\xi)\iotaf x_{\mu,v}).
\end{equation}
\end{Prop}

\begin{proof}
From Proposition \ref{prop} and evaluating $W_\infty$ we have that 
\begin{align*}
a_\mu(\xi;\sigma)&=y^{-k/2}e(-\Tr(\xi z))W_{\phi}(a(\xi)g_z\iotaf x_\mu)\\
&=y^{-k/2}e(-\Tr(\xi z))(\xi y)^{k/2}e(\Tr(\xi z))\prod_{v<\infty}\Wv(a(\xi)\iotaf x_{\mu,v})\\
&=\xi^{k/2}\prod_{v<\infty}\Wv(a(\xi)\iotaf x_{\mu,v}).
\end{align*}
Recall from (\ref{FC}) we have \[c_\mu(\xi;f_\mu||_k\sigma)=N(t_\mu\O_F)^{-k_0/2}\xi^{(k_0\mathbf{1}-k)/2}a_\mu(\xi;\sigma).\] Therefore \[c_\mu(\xi;f_\mu||_k\sigma)=N(t_\mu\O_F)^{-k_0/2}\xi^{k_0\mathbf{1}/2}\prod_{v<\infty}\Wv(a(\xi)\iotaf x_{\mu,v}).\] 
\end{proof}

Let $\f=(f_1,...,f_h)$ be a normalised cuspidal Hilbert newform with weight $k=(k_1,...,k_n)$ such that $k_1\equiv...\equiv k_n \pmod{2}$. We let $\Q(\f)$ denote the number field generated by $c_\mu(\xi;f_\mu)$ as $\xi$ varies over $F$ and $\mu$ varies over $1\leq\mu\leq h$. Then for $\f, 1\leq\mu\leq h, \ \text{and}\ \sigma\in\Gamma_\mu(1)$ we let $\Q(\f,\mu,\sigma)$ denote the field generated by $c_\mu(\xi;f_\mu||_k\sigma)$ as $\xi$ varies over $F$. Note that $\Q(\f)$ is the compositum of the fields $\Q(f,\mu,1)$ as $\mu$ varies over $1\leq\mu\leq h$.

We are now able to state and prove the main result of this chapter.

\begin{Thm}\label{Thm}
Let $\f=(f_1,...,f_h)$ be a normalised cuspidal Hilbert newform of level $\n$ and weight $k=(k_1,...,k_n)$ with $k_1\equiv...\equiv k_n \pmod{2}$. Let $1\leq\mu\leq h$ and $\sigma=\smatrix\in\Gamma_\mu(1)$.  Let $\n'$ be the integral ideal of $\O_F$ such that $\n'(\n+cdt_\mu^{-1}\mathfrak{D}_F^{-1})=\n$. Then $\Q(\f,\mu,\sigma)$ lies in the number field $\Q(\f)(\zeta_{N_0})$ where $N_0$ is the integer such that $N_0\Z=\n'\cap\Z$.
\end{Thm}

\begin{proof}
Let $\tau\in\Aut(\C)$ fix $\Q(\f)(\zeta_{N_0})$ where $N_0$ is the unique positive integer such that \[N_0\Z=\n\big(\gcd(cdt_\mu^{-1}\mathfrak{D}_F^{-1},\n)\big)^{-1}\cap\Z.\] Thus specifically $\tau$ fixes $\Q(\f)$ and all the $\zeta_{N_0}$ roots of unity.

From the fact that $\tau$ fixes $\Q(\f)$ and using \cite[Proposition 2.6]{Shi78}, multiplicity one and Lemma \ref{relation} we have that $\f=\f^\tau$. Therefore, using Lemma \ref{relation} we have \[\Pi_v\otimes| \ |_v^{k_0/2}\cong{}^{\tau}(\Pi_v\otimes| \ |_v^{k_0/2}).\] So, by (\ref{action}), we have that for each non-archimedean place $v$
\begin{equation}\label{tauWv}
\tau(\Wv(a(\xi)\iotaf x_{\mu,v}))\tau\big((|\xi|_v|t_\mu|_v)^{k_0/2}\big)=\Wv(a(\alpha_\tau)a(\xi)\iotaf x_{\mu,v})(|\xi|_v|t_\mu|_v)^{k_0/2}.
\end{equation}

Since all the $N_0$ roots of unity are fixed by $\tau$ we have that \[\alpha_\tau\equiv1 \pmod{\n_v\big(\gcd(c_vd_v\varpi_v^{-t_v-d_v},\n_v)\big)^{-1}},\] where $c_v$ and $d_v$ are the $v$ parts of $c$ and $d$ respectively. Consider the product \[x_\mu^{-1} \begin{pmatrix} a & b \\ c & d \end{pmatrix} \begin{pmatrix} \alpha_\tau & \\ & 1 \end{pmatrix} \begin{pmatrix} a & b \\ c & d \end{pmatrix}^{-1}x_\mu.\] On using \eqref{sigma} we have this product is equal to 
\begin{equation}\label{product}
\begin{pmatrix} 1 & \\ & \varpi_v^{d_v}\end{pmatrix}\begin{pmatrix} a' & b' \\ c' & d' \end{pmatrix}\begin{pmatrix} 1 & \\ & \varpi_v^{-d_v} \end{pmatrix}\begin{pmatrix} \alpha_\tau & \\ & 1 \end{pmatrix} \begin{pmatrix} 1 & \\ & \varpi_v^{d_v} \end{pmatrix} \begin{pmatrix} a' & b' \\ c' & d' \end{pmatrix}^{-1} \begin{pmatrix} 1 & \\ & \varpi_v^{-d_v} \end{pmatrix}.
\end{equation}
On evaluating \eqref{product} we have \[\begin{pmatrix}\eps^{-1}(a'd'\alpha_\tau-b'c') & \eps^{-1}a'b'(1-\alpha_\tau)\varpi_v^{-d_v}\\ \eps^{-1}c'd'(\alpha_\tau-1)\varpi_v^{d_v} & \eps^{-1}(a'd'-b'c'\alpha_\tau)\end{pmatrix}, \ \text{where} \ \eps=\det(\begin{smallmatrix}a' & b' \\ c' & d' \end{smallmatrix}).\] We want to show that this matrix is an element of $K_v(\n)$. Therefore, we need the following conditions to be satisfied 
\begin{enumerate}
\item $\eps^{-1}a'b'(1-\alpha_\tau)\varpi_v^{-dv}\in\mathfrak{D}_v^{-1}$,
\item $\eps^{-1}c'd'(\alpha_\tau-1)\varpi_v^{d_v}\in \n_v\mathfrak{D}_v$,
\item $\eps^{-1}(a'd'-b'c'\alpha_\tau)\in\O_v$.
\end{enumerate}
Recall that $a',b',c'$ and $d'$ are elements of $\O_F$ and $\eps^{-1}\in\O_F^\times$ so conditions (1) and (3) are satisfied with no additional assumptions on $\alpha_\tau$. Now consider condition (2). From the fact we have \[\alpha_\tau\equiv1 \pmod{\n_v\big(\gcd(c_vd_v\varpi_v^{-t_v-d_v},\n_v)\big)^{-1}}\] which is equivalent to \[\alpha_\tau\equiv1 \pmod{n_v\big(\gcd(c'_vd'_v,\n_v)\big)^{-1}},\] using Remark \ref{dash}. Thus, we have that condition (2) is also satisfied. Therefore \[\begin{pmatrix}\eps^{-1}(a'd'\alpha_\tau-b'c') & \eps^{-1}a'b'(1-\alpha_\tau)\varpi_v^{-dv}\\ \eps^{-1}c'd'(\alpha_\tau-1)\varpi_v^{d_v} & \eps^{-1}(a'd'-b'c'\alpha_\tau)\end{pmatrix}\in K_v(\n).\] Recall that in general for $W_v(g_1)=W_v(g_2)$ it suffices that $g_1^{-1}g_2\in K_v(\n)$. Therefore we have that 
\begin{equation}\label{Wvfin}
\Wv(a(\xi)\iotaf x_{\mu,v})=\Wv(a(\alpha_\tau)a(\xi)\iotaf x_{\mu,v}),
\end{equation}
for every finite place. Combining (\ref{tauWv}) and (\ref{Wvfin}), we have 
\begin{equation}\label{fraction}
\frac{\tau\big(\Wv(a(\xi)\iotaf x_{\mu,v})\big)}{\Wv(a(\xi)\iotaf x_{\mu,v})}=\frac{|\xi t_\mu|_v^{k_0/2}}{\tau(|\xi t_\mu|_v)^{k_0/2}}.
\end{equation}
We now consider the quotient \[\frac{\tau\big(c_\mu(\xi;f_\mu||_k\sigma)\big)}{c_\mu(\xi;f_\mu||_k\sigma)}.\] To prove the result it now suffices to show this quotient is equal to one. Therefore, by (\ref{cmu}) and (\ref{fraction}) we have 
\begin{align*}
\frac{\tau\big(c_\mu(\xi;f_\mu||_k\sigma)\big)}{c_\mu(\xi;f_\mu||_k\sigma)}&=\frac{\tau\big(N(t_\mu\O_F)^{-k_0/2}\big)\tau(\xi^{k_0/2})}{N(t_\mu\O_F)^{-k_0/2}\xi^{k_0/2}}\prod_{v<\infty}\frac{|\xi t_\mu|_v^{k_0/2}}{\tau(|\xi t_\mu|_v)^{k_0/2}}\\
&=\frac{\tau\big(\prod_{v<\infty}|t_\mu|_v^{k_0/2}\big)}{\prod_{v<\infty}|t_\mu|_v^{k_0/2}}\frac{\tau\big(\prod_{v|\infty}|\xi|_v^{k_0/2}\big)}{\prod_{v|\infty}|\xi|_v^{k_0/2}}\frac{\big(\prod_{v<\infty}|t_\mu|_v|\xi|_v\big)^{k_0/2}}{\tau\big(\prod_{v<\infty}|t_\mu|_v|\xi|_v\big)^{k_0/2}}\\
&=\frac{\tau\big(\prod_{v|\infty}|\xi|_v^{k_0}\big)}{\prod_{v|\infty}|\xi|_v^{k_0}}, \ \text{using global norm formula}.
\end{align*}
We have that $\prod_{v|\infty}|\xi|_v=|N(\xi)|\in\Q$. Hence \[\frac{\tau\big(\prod_{v|\infty}|\xi|_v^{k_0}\big)}{\prod_{v|\infty}|\xi|_v^{k_0}}=\frac{\prod_{v|\infty}|\xi|_v^{k_0}}{\prod_{v|\infty}|\xi|_v^{k_0}}=1.\] Therefore, we have that $\tau$ fixes $c_\mu(\xi;f_\mu||_k\sigma)$ and so \[c_\mu(\xi;f_\mu||_k\sigma)\in\Q(\f)(\zeta_{N_0}).\]
\end{proof}

\begin{Remark}\label{special}
Let $F=\Q$, then we have narrow class group beginning trivial and so $t_\mu=1$. We also have that $\n=N\Z$ and $\D_\Q^{-1}=\Z$. Therefore \[N_0\Z=\n\bigg(\gcd(cd\Z,N\Z)\bigg)^{-1}\cap \Z=\frac{N}{(cd,N)}\Z.\] This is exactly the cyclotomic extension given by Brunault and Neururer in \cite[Theorem 4.1]{FEatC}.
\end{Remark}

\begin{Question}
Our method to prove Theorem \ref{mf2} and Theorem \ref{Thm} was to find sufficient conditions, not necessary ones. Therefore one can ask if the number field $\Q(\f)(\zeta_{N_0})$ is optimal, that is, do we have $\Q(\f,\mu,\sigma)=\Q(\f)(\zeta_{N_0})$?

Note that if $F=\Q$ this is known to be true \cite[Theorem 7.6]{FEatC}, so one can restrict their attention to the case of $n>1$.
\end{Question}

\begin{Question}
It would be interesting to generalise the results in this chapter to the setting of Hilbert newforms with non-trivial central character and specifically what will the number field be.

One can also ask the question of what happens for automorphic forms over different groups, for example, Siegel modular forms. As stated in the introduction one would suspect that a similar method used here would work in that case.
\end{Question}

In the next chapter, we give an outline of possible strategies about how one might go about answering some of these questions. 


\chapter{Conclusions and Future Work}\label{Chap4}
In this chapter, the main conclusions drawn from Chapters \ref{Chap2} and \ref{Chap3} are discussed and how these relate to the aim of our thesis, mentioned in Chapter \ref{Chap1}. Several directions for further work are outlined, which are based on the work carried out in this thesis. 

\section{Conclusions}

This thesis aimed to gain a better understanding of the arithmetic properties of the Fourier coefficients of modular forms via using the adelic representation theory. In Chapter \ref{Chap2}, we were focussed on trying to understand more about the $p$-adic valuation of local Whittaker newforms with non-trivial central character. For Chapter \ref{Chap3}, we change direction slightly and focus our attention on Hilbert modular forms and their Fourier coefficients. 

At the heart of all the work carried out in this thesis is the interplay between the classical and adelic approaches. Moving to the adelic side allows us to tackle problems which if we were working classically we would get little traction. Moreover, it enables us to build a general method that can be adapted to different settings, as we will see in the next section. 

In Chapter \ref{Chap2}, we focussed purely on the local representation side and we were able to generalise what is known to the case of $\cc\in\X$. This relied on using the local Fourier expansion of $W_\pi$ and knowing explicitly what the Fourier coefficients are for different representations $\pi$. All of this accumulated in being able to prove Theorem \ref{api2} and \ref{apiplus}. These theorems along with \eqref{af} could then be used to understand more about the Fourier coefficients of modular forms at cusps. We give an outline of how one can do this in the next section. The method we have used to prove our theorems also gives a possible way to obtain similar results for different groups. 

As in Chapter \ref{Chap2}, a majority of the work in Chapter \ref{Chap3} has been investigating properties of the local Whittaker newform.  In this chapter specifically, our focus was twofold. The first was studying the Galois action on the local Whittaker newform and the second was constructing an explicit formula relating this local (adelic) object to the global (classical) object. We then have all the tools to be able to prove the main result of this chapter namely, Theorem \ref{Thm}.

All of this work fits into the general framework of automorphic representation theory. It is done, in the hope that at least in some small part it has increased our understanding of newforms and the objects associated to newforms.  Which in turn, could help us to answer some of the important questions related to newforms. 

\section{Future work}

In this section, we provide some possible further work that builds on the work carried out in this thesis. 

\subsection{Unanswered questions}\label{unQ}

Here, we briefly discuss the unanswered questions raised in the thesis in a little more detail. This discussion is more of a heuristic on how one might go about tackling the problems raised, rather than working through all of the details.  

\begin{Question}[$p=2$]

In Chapter \ref{Chap2}, we only focus on the case of $p$ being an odd prime. The main reason for this is due to time constraints as there are some technical issues. Although we could obtain more explicit (and perhaps better) bounds than in the case of $p$ odd there is a trade-off, this being that we could in most likelihood only deal with the case of $F=\Q_2$. These technical complications occur when $\pi$ is Type 1 as there appears to be no easy formula for $\eps(1/2,\chi\pi)$ and $\wtpi$ will not be dihedral. This means we need to explicitly work out the characters of $\X_{\Q_2}$ and the corresponding $\eps$-factors. 

For $p$ odd we have followed \cite{MCMD} one suspect that a similar generalisation would be a sensible approach. If we investigate the method they use, they explicitly work out how many non-dihedral representations exist. Then from there, they are able to establish the desired bounds for $\valp(W_\pi(\gtlv))$. For the other representations, they can prove the bounds in a similar way to that of $p$ odd, but computing these bounds are no easier than the case of $p$ odd. As such, the case of $p=2$ would constitute a non-trivial amount of work. 
\end{Question}

\begin{Question}[Global applications]

As mentioned in Chapter \ref{Chap2}, we can use Theorem \ref{api2} and \ref{apiplus} to prove results for $p$-adic valuations for the Fourier coefficients of modular forms at cusps. Here, we give an overview of how this method would go and some possible other applications. 

In our setting, $\pi_f$ relates to newforms $f$ on $\Gamma_1(N)$, since we have non-trivial central character. Recall that, for $\c=\sigma\infty$ we have a Fourier expansion of \[f|_k\sigma(z)=\sum_{n\geq1}\af e^{2\pi nz/\delta(\c)}.\] We define, \[\valp(f|_k\sigma):=\valp(f|_\c)=\inf_{n\geq1}\left\{\valp((\af))\right\}.\] Let $p$ be a fixed prime and $q$ a prime which varies. The local objects that we defined in Chapter \ref{Chap2} now have subscripts. In this global setting we have $N=\prod_qq^{n_q}$, with $n_q=c(\pi_q)$ and $M=\prod_qq^{m_q}$, with $m_q=c(\omega_{\pi_q})$. 
For $n=n_0\prod_{q|N}q^{n_q}$ and $L=\prod_{q|N}q^{l_q}$, then from \cite[Proposition 3.3]{ACAS} we have 
\begin{equation}\label{afc4}
\af=a_f(n_0)\frac{n^{k/2}}{n_0^{k/2}}e\left(\frac{nd}{\delta(\c)L}\right)\frac{\prod_{q|N}W_{\pi_q}(g_{n_q-\max\{n_q,l_q+m_q,2l_q\},l_q,v_q})}{\delta(\c)^{k/2}}.
\end{equation}
There are two cases to consider, the first being $p\nmid N$ and $p|N$.  In the first case we need to check if $\valp(f|_\c)\geq0$ as in \cite[Lemma 4.5]{MCMD} and \cite[12.3.5]{MFandMC}. If we now suppose we are in the second case. Then, rewriting \[\delta(\c)=\frac{L^2\cdot N\cdot L \cdot M}{L^2\gcd(L^2,N,LM)}=\frac{NLM}{\gcd(N,LM,L^2)}\] and taking $p$-adic valuation of $\eqref{afc4}$, gives
\begin{align*}
\valp(\af)&=\valp(a_f(n_0))+\valp\left(e\left(\frac{nd}{\delta(\c)L}\right)\right)+\valp\left(p^{\frac{k}{2}\valp(n)}\right)\\
& \quad -\frac{k}{2}\valp\Big(\frac{NLM}{\gcd(N,LM,L^2)}\Big)+\valp\big(\Wp(g_{n_p-\max\{n_p,l_p+m_p,2l_p\},l_p,v_p})\big).
\end{align*}
Now, we know that \[\valp\left(e\left(\frac{nd}{\delta(\c)L}\right)\right)=0 \ \text{and} \ \valp(a_f(n_0))\geq0.\] Therefore, 
\begin{align*}
\valp(\af)&=\frac{k}{2}\valp(n)-\frac{k}{2}\valp \left(\frac{NLM}{\gcd(N,LM,L^2)}\right)\\
& \quad +\valp\left(\Wp(g_{\valp(n)-\max\{\valp(N),\valp(L)+\valp(M),2\valp(L)\},\valp(L),v_p})\right).
\end{align*}
If we denote $\valp(n)$ by $u\in\Z_{\geq0}$. Then taking minimums over $u\in\Z_{\geq0}$ and $v_p\in\Z_p^\times$, we have 
\begin{align*}
\valp(f|_\c)&\geq-\frac{k}{2}\valp\left(\frac{NLM}{\gcd(N,LM,L^2)}\right)\\
&+\min_{u\in\Z_{\geq0},v_p\in\Z_p^\times}\left\{\frac{ku}{2}+\valp\left(\Wp(g_{u-\max\{\valp(N),\valp(L)+\valp(M),2\valp(L)\},\valp(L),v_p})\right)\right\}.
\end{align*}
So, we can see that to obtain lower bounds for $\valp(f|_\c)$ we need to evaluate \[\min_{u\in\Z_{\geq0},v_p\in\Z_p^\times}\left\{\frac{ku}{2}+\valp\big(\Wp(g_{u-\max\{\valp(N),\valp(L)+\valp(M),2\valp(L)\},\valp(L),v_p})\big)\right\}.\] We can do this by using the bounds in Theorem \ref{api2} and \ref{apiplus}. We compute this for the example of $\pi$ being of Type 2a and $c(\pi)=2$. In this case we have $c(\pi)=\valp(N)=2$ and $\valp(L)=1=\valp(M)$ so \[\max\left\{\valp(N),\valp(L)+\valp(M),2\valp(L)\right\}=2.\] For $F=\Qp$, Theorem \ref{api2} (ii) gives \[\valp\big(\Wp(g_{u-2,1,v_p})\big)\geq-(u-2+3)+\frac{1}{p-1}=-(u+1)+\frac{1}{p-1}.\] Thus, \[\frac{ku}{2}-(u+1)+\frac{1}{p-1}=\frac{u(k-2)-2}{2}+\frac{1}{p-1}.\] We can see that this is linear in $u$ and is increasing as $u$ increases. So, the minimum will occur at $u=0$. Thus, \[\valp(f|_\c)\geq-\frac{k}{2}\valp\Big(\frac{NLM}{\gcd(N,LM,L^2)}\Big)-1+\frac{1}{p-1}.\] Note that for us to use the results for $\pi$ being Type 3 we need to be able to compute $|\valp(p^{c_1})|$. To ensure this is well defined we need to show that $p^{\pm c_1+r}\in\overline{\Z}$, for some $r$, so that $|\valp(p^{c_1})|\leq r$. In \cite[Lemma 4.4]{MCMD} they show that $r=(k-1)/2$. Their method uses the Ramanujan-Peteresson conjecture at all finite places \cite[Theorem 1]{Bla06} and uses another explicit formula for the Whittaker newform \cite[Eq. 121]{PSS14}. We will need to do the same for our setting.  This seems like the most sensible method to prove this result in our setting. 

Finally, the main purpose of the author's work in \cite{MCMD} was to gain new results towards Manin's Constant.  This is about the surjection regarding the modularity of elliptic curves, see  \cite{Ces18} and \cite{manin}. They use the bounds on $\valp(f|_\c)$ in the specific case that $f$ is a newform of weight 2 and level $N_E$. The reason for this method is that they are able to prove better results than using just arithmetic geometry. An interesting question is can we do something similar for our results?  Or, if our bounds can be used to prove different results not just about $\valp(f|_\c)$. 
\end{Question}

\begin{Question}[Optimality]
At the end of Chapter \ref{Chap3} the question is raised as to whether or not the number field $\Q(\f)(\zeta_{N_0})$ is optimal, recall this means $\Q(\f,\mu,\sigma)=\Q(\f)(\zeta_{N_0})$. As also mentioned in Chapter \ref{Chap3} we know the result for modular forms is optimal this is proved in \cite[Theorem 7.6]{FEatC}. The way they prove this result is classical in nature and so would be difficult to reinterpret in the setting of Hilbert modular forms. 

One would suspect that this result is optimal, if not in general, then in the case of $\mu=1$. There are several different methods for investigating this idea. Possibly the most intuitive of these is establishing some computational evidence by taking examples from LMFDB and computing $\Q(\f,\mu,\sigma)$ and $\Q(\f)(\zeta_{N_0})$ for those examples. However, computing these Fourier coefficients explicitly is not an easy process and would require work. It also leaves the problem of still needing to prove the result if indeed no counterexamples are provided. 

Therefore, we turn to some possible other fruitful methods which are linked. In essence we want to find necessary conditions for $\tau\in\Aut(\C)$ to fix the product $c_\mu(\xi;f_\mu||_k\sigma)$. As such, one could try to re-engineer the proof of Theorem \ref{Thm} in such a way instead of finding sufficient conditions we obtain necessary ones. Doing this directly could be a challenge so it would seem a sensible strategy to first consider the setting of $F=\Q$. This strategy was useful in Chapter \ref{Chap3}, as it gave good insight into how one could tackle the problem for Hilbert modular forms and was less technical. This means if one could reprove \cite[Theorem 7.6]{FEatC} but using a more adelic approach, it could well give the desired intuition on how to show the result in general. 
\end{Question}

\subsection{Connecting Chapters \ref{Chap2} \& \ref{Chap3}}

At first glance the work in Chapters \ref{Chap2} and \ref{Chap3} are answering different questions and perhaps are not so connected (apart from the general overall setting). If we examine a little more closely though we see that there is a possible connection between these two chapters.

If we first look at the work in Chapter \ref{Chap2}, the formulas for $\ctl$ in Section \ref{formulas} are for any non-archimedean local field, specifically these results hold for $F_v$ with $F$ being a totally real number field and $v$ a finite place of $F$. 

Now, turning our attention to the work carried out in Chapter \ref{Chap3}. We see that we have proved a similar looking formula for $\af$ found in \cite[Proposition 3.3]{ACAS}, they show \[\af=\frac{a_f(n_0)}{n_0^{k/2}}e\bigg(\frac{rd}{\delta(\c)L}\bigg)\frac{\prod_{p|N}W_{\pi_p}(g_{n_p-d_{\pi_p}(l_p),l_p,v_p})}{\delta(\c)^{k/2}}.\] Comparing this to what we proved, recall \eqref{cmu}, we showed \[c_\mu(\xi;f_\mu||_k\sigma)=N(t_\mu\O_F)^{-k_0/2}\xi^{k_0\mathbf{1}/2}\prod_{v<\infty}W_v(a(\xi)\iotaf x_{\mu,v}).\] The main difference is at what values we are evaluating the local Whittaker newform at and over which places. The reason for this is that due to the work carried out in Chapter \ref{Chap3} we wanted to know about the Galois action on $\Wv$, so we left the input in a way that was most convenient for this. As such, we did not utilise the decomposition of $\GL_2(F)$. Indeed, even in \cite{ACAS} the authors do not fully utilise this decomposition until proposition 3.3, At the beginning of Section 3 they prove results for any $g\in\GL_2(F)$. The first thing that would be required is to use this decomposition for $\Wv$ and compute the corresponding $t,l$ and $v$ in this case. Once we have obtained the equivalent result we will need to go back and work through the results in Chapter \ref{Chap2} which would need adjusting. 

\subsection{Siegel modular forms}

All the work carried out in this thesis has been for $\GL_2$. If we look at $\GL_n$ with $n\geq3$ then the situation is very different, specifically, there are no holomorphic forms for $\GL_n$. Since all of our work has been studying holomorphic forms we need to find a different generalisation that gives holomorphic forms.  The correct setting is the symplectic group, $\text{GSp}_{2n}$, instead of $\GL_n$. Let $n\in\mathbb{N}$ and $R$ a commutative ring, then \[\text{GSp}_{2n}(R):=\big\{g\in\GL_{2n}(R):g^TJg=\mu(g)\in R^\times, J=(\begin{smallmatrix} 0_n & I_n \\ -I_n & 0_n \end{smallmatrix})\big\},\] where $\mu(g):\text{GSp}_{2n}(R)\rightarrow R^\times$ is a homomorphism called the multiplier. We have the subgroup $\text{Sp}_{2n}(R):=\{g\in\text{GSp}_{2n}(R):\mu(g)=1\}$. 

The Siegel upper half plane $\H_n$ is defined by \[\H_n=\{Z\in M_n(\C):Z^T=Z, \ \Ima Z>0\}.\]We can consider an action of \[\text{GSp}_{2n}^+(\R)=\{g\in\text{GSp}_{2n}(\R):\mu(g)>0\}\] on $\H_n$ given in an analogous way to the previous settings, specifically \[g\cdot Z=(AZ+B)(CZ+D)^{-1}, \ g=(\begin{smallmatrix} A & B \\ C & D \end{smallmatrix})\in\text{GSp}_{2n}^+(\R).\] Now, for simplicity we restrict our attention to $\GSp_4$.  Let $N\geq1$, define \[\Gamma^{(2)}(N)=\{g\in\Sp_4(\Z):g\equiv I_4 \pmod{N}\}.\] Then $\Gamma<\Sp_4(\Q)$, is a congruence subgroup if it contains $\Gamma^{(2)}(N)$ for some $N$, with finite index. Let $F:\H_4\rightarrow\C$ be holomorphic then $F$ is a Siegel modular form if $F|_k\gamma=F$, for every $\gamma\in\Gamma$. In the case that $\Gamma=\Sp_4(\Z)$ then $F$ has a Fourier expansion of the form \[F(Z)=\sum_{\substack{M=M^T\\ M \ \text{half integral}}}A_F(T)e^{2\pi i \Tr(MZ)}.\] One can ask the same questions as we have done for Chapters \ref{Chap2} and \ref{Chap3}.  One can attempt to use a similar method of turning the classical question into one using adelic representation theory. We have the same relations as we described in Section \ref{ctoa} for Siegel modular forms. The main difference is if $\pi_F$ is the cuspidal automorphic representation associated to $F$ then $\pi_F$ is not globally generic, this means they do not have a global Whittaker model. Therefore, a lot of the theory about cuspidal automorphic representations of $\GL_2(\A)$ are not applicable.  However, it is known that every irreducible cuspidal automorphic representation of $\GSp_4(\A)$ has a Bessel Model \cite[Chapter 8]{smf}. This means we need to reformulate our methods to this model. This is just one possible way to look at our possible problems discussed, it could be that a direct approach, that is, directly studying $A(F|_k\sigma, T)$ is also possible. 

We start with the problem discussed in Chapter \ref{Chap3}. We can again look at the Fourier expansion of newforms at cusps.  From the $q$-expansion principle we know that the Fourier coefficients of $F$ generate a number field with those at a cusp being in a cyclotomic extension. In \cite[Section 3.3]{DPSS} the authors relate the Bessel model to the Fourier coefficients of $F$. So, one would need to study this for $F|_k\sigma$ with $\sigma\in\Sp_4(\Z)$. We would also need to study how $\tau\in\Aut(\C)$ acts on the local Bessel model. 

The difficulty with the work carried out in Chapter \ref{Chap2} in the setting of Siegel modular forms is obtaining explicit formulas akin to those in Section \ref{formulas}. The reason for this is that these formulas come from evaluating 
\begin{align*}
\sum_{t=-\infty}^\infty q^{(t+c(\chi\pi))(\frac{1}{2}-s)}\ctl&=\cc(-1)\eps(1/2,\chi\pi)^{-1}\frac{L(s,\chi\pi)}{L(1-s,\chi^{-1}\cc^{-1}\pi)}\\
& \quad \cdot \sum_{m=0}^\infty W_\pi(a(\varpi^m))q^{-m(\frac{1}{2}-s)}G(\varpi^{m-l},\chi^{-1}).
\end{align*} 
This identity is given in \cite[Proposition 2.23]{Sah16}. The way that they establish this result is by the local functional equation described by Jacquet and Langlands in \cite{JL70}. In the setting of Siegel modular the difficulty lies in reproducing such an identity. So, the first (and large) step would be to reproduce the work of \cite{ASS19} and \cite{Sah16}. The number of representations for $\GSp_4$ would not be as much of an issue and in \cite{newforms}, the authors have tables for $L$-factors and $\eps$-factors Tables A8 and A9 respectively. If again, we want to then study the $p$-adic valuation of the Fourier coefficients of Siegel modular forms at cusps one would need an explicit formula relating the $p$-adic valuation of the adelic object to that of the Fourier coefficients.


\pagestyle{plain}



\bibliographystyle{plain}
\cleardoublepage 
\phantomsection  

\end{document}